\documentclass[reqno]{amsart}
\usepackage{amsmath} 

\usepackage{amsthm}
\usepackage{hyperref}
\usepackage{amssymb}
\usepackage{graphics}
\usepackage{latexsym}
\usepackage{comment}
\usepackage{extarrows}
\usepackage{wrapfig}
\usepackage{comment}
\usepackage {colortbl,array,xcolor}
\usepackage{hhline}
\usepackage{tikz}
\usepackage{float}
\usetikzlibrary{intersections, calc, math, patterns}

\numberwithin{equation}{section}
\newtheorem{thm}{Theorem}[section]
\newtheorem{prop}[thm]{Proposition}
\newtheorem{lem}[thm]{Lemma}
\newtheorem{cor}[thm]{Corollary}

\theoremstyle{remark}
\newtheorem{rem}{Remark}[section]
\newtheorem{defn}{Definition}

\newtheorem{case}{Case}
\newtheorem{condition}{Condition}

\newcommand{\laplacian}{\Delta}
\newcommand{\R}{{\mathbb R}}
\newcommand{\Z}{{\mathbb Z}}

\newcommand{\N}{{\mathbb N}}
\newcommand{\C}{{\mathbb C}}
\newcommand{\LR}[1]{{\langle {#1} \rangle }}
\newcommand{\cross}{\times}
\newcommand{\e}{\varepsilon}

\newcommand{\F}{\mathcal{F}}

\newcommand{\ha}{\widehat}

\title[Well-posedness of the Zakharov-Kuznetsov equation]{Global Well-posedness\\
for the Cauchy problem of \\
the Zakharov-Kuznetsov equation 
in 2D
}

\author[S.  Kinoshita]{Shinya Kinoshita}
\address[Shinya Kinoshita]{Universit\"{a}t Bielefeld
Fakult\"{a}t f\"{u}r Mathematik
Postfach 10 01 31
33501 Bielefeld
Germany}
\email[Shinya Kinoshita]{kinoshita@math.uni-bielefeld.de}


\begin{document}

\begin{abstract}
This paper is concerned with the Cauchy problem of the $2$D Zakharov-Kuznetsov equation. 
We prove bilinear estimates which imply local in time well-posedness in the Sobolev space $H^s(\R^2)$ 
for $s > -1/4$, and these are optimal up to the endpoint. 
We utilize the nonlinear version of the classical Loomis-Whitney inequality and develop an almost orthogonal decomposition of the set of resonant frequencies. As a corollary, we obtain global well-posedness in $L^2(\R^2)$.
\end{abstract}
\maketitle
\setcounter{page}{001}


\section{Introduction}
We consider the Cauchy problem of the Zakharov-Kuznetsov equation
\begin{equation}
 \begin{cases}
  \partial_t u + \partial_{x_1} \laplacian u =  \partial_{x_1} (u^2), \quad (t,x_1, \cdots, x_n) \in [-T,T] \cross \R^n, \\
  u(0, \cdot) = u_0 \in H^s(\R^n),\label{ZK}
 \end{cases}
\end{equation}
where $u= u(t,x_1, \cdots, x_n)$ is a real valued function and $\laplacian = \partial_{x_1}^2 +
\cdots + \partial_{x_n}^2$ is the Laplacian.
The equation \eqref{ZK} was introduced by Zakharov and Kuznetsov in \cite{ZK74} as a model for the propagation of ion-sound waves in magnetic fields for $n=3$. See also \cite{LS82}. In \cite{LLS13}, Lannes, Linares and Saut derived \eqref{ZK} in dimensions $2$ and $3$ rigorously as a long-wave limit of the Euler-Poisson system.
The Zakharov-Kuznetsov equation can be seen as a multi-dimensional extension of the KdV equation
\begin{equation*}
\partial_t u + \partial_{x}^3 u + u  \partial_{x} u = 0, \quad (t,x) \in \R \cross \R.
\end{equation*}
In contrast to the KdV equation and the KP equation, which is another multi dimensional generalization of KdV equation, the Zakharov-Kuznetsov equation is not completely integrable and has only two conservation laws in $L^2$ and $H^1$.
\begin{equation*}
M(u) := \int_{\R^n} u^2 dx_1 \cdots d x_n, \quad
E(u) :=  \int_{\R^n} \frac{1}{2} |\nabla u|^2 + \frac{1}{3} u^3 dx_1 \cdots d x_n.
\end{equation*}

There are lots of works on the Cauchy problems of the Zakharov-Kuznetsov equation \eqref{ZK}.
In the $2$D case, Faminskii \cite{Fa95} established the local and global well-posedness in $H^1(\R^2)$.  In \cite{LP09}, this results was improved by Linares and Pastor.
They obtained the local well-posedness of \eqref{ZK} in $H^s (\R^2)$ for $s > 3/4$.
Gr\"{u}nrock and Herr \cite{GH14}, and Molinet and Pilod \cite{MP15} proved independently the well-posedness of \eqref{ZK} in $H^s(\R^2)$ for $s>1/2$ by using the Fourier restriction norm method.
In the three dimensional case $(n=3)$, Linares and Saut \cite{LS09} showed the local well-posedness in $H^s(\R^3)$ for $s > 9/8$. In \cite{RV12-ZK}, Ribaud and Vento proved the local well-posedness in $H^s(\R^3)$ for $s>1$ and $B^{1,1}_2(\R^3)$. By employing this local well-posedness result in $B^{1,1}_2(\R^3)$, Molinet and Pilod \cite{MP15} showed that \eqref{ZK} is actually globally well-posed in $H^s(\R^3)$ for $s>1$.

It should be noted that the Cauchy problem of the so-called generalized Zakharov-Kuznetsov equation
\[
\partial_t u + \partial_{x_1} \laplacian u =  \partial_{x_1} (u^{k+1}),  \quad (k \in \N).
\]
has been extensively studied as well as the original Zakharov-Kuznetsov equation.
For $k=2$, this equation, called the modified
Zakharov-Kuznetsov equation, may be seen as a multi-dimensional extension of the modified KdV equation.
Regarding the modified Zakharov-Kuznetsov equation, in the $2$ dimensional case,
Biagioni and Linares \cite{BL03} established the global well-posedness in $H^1(\R^2)$.
In \cite{LP09}, Linares and Pastor proved the local well-posedness in $H^s(\R^2)$ for $s>3/4$, and in \cite{LP11}, they showed the global well-posedness for $s>53/66$. After that, Ribaud and Vento \cite{RV12-gZK} improved the necessary regularity condition by showing the local well-posedness in $H^s(\R^2)$ for $s>1/4$.
In the $3$D and higher dimensional case, Gr\"{u}nrock \cite{Gru14} proved that the Cauchy problem of the modified Zakharov-Kuznetsov equation is locally well-posed in $H^s(\R^n)$ if $s>n/2-1$, which is optimal up to the endpoint.
We refer to the papers \cite{FLP12}, \cite{Gru15arxiv}, \cite{LP11}, \cite{RV12-gZK} for the case $k \geq 3$.

The aim of the paper is to establish well-posedness of \eqref{ZK} in $2$ dimensions for low regularity initial data.
So far, the best known result is the local well-posedenss in $H^s(\R^2)$ for $s>1/2$ due to Gr\"{u}nrock and Herr \cite{GH14}, and Molinet and Pilod \cite{MP15}.
The scaling critical index $s=-1$ of \eqref{ZK} for $n=2$ suggests well-posedness in the range $-1 \leq s$.
In particular, in view of the conservation of mass, it is a natural question whether $L^2$-well-posedness holds true.
Before stating the main result, we introduce the symmetrized equation of \eqref{ZK} by performing a linear change of variables as in \cite{GH14}. Put $x= 4^{-1/3}x_1 + \sqrt{3} 4^{-1/3} x_2$,
$y= 4^{-1/3}x_1 - \sqrt{3} 4^{-1/3} x_2$ and $v(t,x,y) := u(t,x_1,x_2)$,
$v_0(x,y) := u_0(x_1,x_2)$.
Clearly, the linear transformation $(x_1,x_2) \to (x,y)$ is invertible as a mapping $\R^2 \to \R^2$
and we see that
\begin{align*}
& \partial_{x_1} u(x_1,x_2)  = 4^{-\frac{1}{3}} (\partial_x + \partial_y) v(x,y), \\
& \partial_{x_2} u(x_1,x_2)  = \sqrt{3} 4^{-\frac{1}{3}} (\partial_x - \partial_y) v(x,y),\\
& \partial_{x_1} (\partial_{x_1}^2+ \partial_{x_2}^2) u (x_1,x_2)  =
(\partial_x^3 + \partial_y^3) v(x,y).
\end{align*}
Therefore, \eqref{ZK} for $d=2$ can be rewritten as
\begin{equation}
 \begin{cases}
  \partial_t v + (\partial_x^3 + \partial_y^3) v = 4^{-\frac{1}{3}}
(\partial_x+ \partial_y) (v^2), \quad (t,x,y) \in [-T,T] \cross \R^2, \\
  v(0, \cdot) = v_0 \in H^s(\R^2).\label{ZK'}
 \end{cases}
\end{equation}
Since it is slightly more convenient, we consider the symmetrized equation \eqref{ZK'} instead of \eqref{ZK} hereafter.
We now state our main result.
\begin{thm}  \label{mth}
Let $s > -1/4$. Then the Cauchy problem \eqref{ZK'} is locally well-posed in $H^{s}(\R^2)$.
\end{thm}
By making use of the conservation law in $L^2$, this theorem immediately yields the following global result.
\begin{cor}
The Cauchy problem \eqref{ZK'} is globally well-posed in $L^2(\R^2)$.
\end{cor}
The subcritical threshold $-1/4$ is optimal for the Picard iteration approach, as the following theorem shows.
\begin{thm}\label{not-c2}
Let $s < - 1/4$. Then for any $T>0$, the data-to-solution map
$ u_0 \mapsto u$
of \eqref{ZK'}, as a map from the unit ball in
$H^s(\R^2)$ to
$C([0,T]; H^{s})$ fails to be $C^2$.
\end{thm}
The strategy of the proof of Theorem \ref{mth} is the Fourier restriction norm method which is, roughly speaking, a contraction mapping argument in the Fourier restriction space
$X^{s,b}$ to be defined in the next section.
The biggest difficulty in the proof appears when we treat resonant interactions.
The geometry of this set is very complicated due to the linear part of Zakharov-Kuznetsov equation.
To overcome this, we employ a convolution estimate on hypersurfaces which was introduced by Bennet, Carbery, and Wright in \cite{BCW05} and generalized by Bejenaru, Herr, and Tataru in \cite{BHT10}.
This estimate, called the nonlinear version of the classical Loomis-Whitney inequality (see \cite{LW49} for the classical Loomis-Whitney inequality) and was applied
to the Cauchy problem of the Zakharov system on
$\R^2$ and $\R^3$ by Bejenaru, Herr, Holmer, and Tataru in \cite{BHHT09}
and by Bejenaru and Herr in \cite{BH11}, respectively.
The nonlinear Loomis-Whitney inequality is based on the transversality of three characteristic hypersurfaces. For the Zakharov system, it was found in \cite{BHHT09} and \cite{BH11} that
the transversality depends only on the size of the angle between two frequencies of the waves
which cause a resonant interaction.
In contrast, for the Zakharov-Kuznetsov equation, the transversality depends not only on
the relation of the angles but also the sizes of frequencies of the two waves.
Therefore, we need to modify the argument which was performed for the Zakharov system in a suitable manner.
In particular, we introduce a new almost orthogonal decomposition which allows to capture this transversality in a precise form.

The paper is organized as follows. In Section 2, we introduce the solution spaces and some fundamental estimates as preliminary. Section 3 is devoted to the proof of the key bilinear estimate which immediately provides the main theorem by a standard contraction mapping argument.
In Section 4, we show Theorem \ref{not-c2}. In Section 3, we will skip the proof of
Lemma \ref{lemma3.8}. Thus, lastly as Appendix, we completes the proof of Lemma \ref{lemma3.8}.
\section{Preliminaries}
In this section, we introduce estimates which will be utilized for the proof of the key bilinear estimate.
Throughout the paper, we use the following notations.
$A{\ \lesssim \ } B$ means that there exists $C>0$ such that $A \le CB.$
Also, $A\sim B$ means $A{\ \lesssim \ } B$ and $B{\ \lesssim \ } A.$
Let $u=u(t,x,y).\ \F_t u,\ \F_{x,y} u$ denote the Fourier transform of $u$ in time, space, respectively.
$\F_{t, x,y} u = \ha{u}$ denotes the Fourier transform of $u$ in space and time.
Let $N$, $L \geq 1$ be dyadic numbers, i.e. there exist $n_1$, $n_2 \in \N_{0}$ such that
$N= 2^{n_1}$ and $L=2^{n_2}$, and $\psi \in C^{\infty}_{0}((-2,2))$ be an even, non-negative function which satisfies $\psi (t)=1$ for $|t|\leq 1$ and letting
$\psi_N (t):=\psi (t N^{-1})-\psi (2t N^{-1})$,
$\psi_1(t):=\psi (t)$,
the equality $\displaystyle{\sum_{N}\psi_{N}(t)=1}$ holds.
Here we used $\displaystyle{\sum_{N}= \sum_{N \in 2^{\N_0}}}$ for simplicity.
We also use the notations $\displaystyle{\sum_{L}= \sum_{L \in 2^{\N_0}}}$ and
 $\displaystyle{\sum_{N,L}= \sum_{N,L \in 2^{\N_0}}}$ throughout the paper.
We define frequency and modulation projections $P_N$, $Q_L$ as
\begin{align*}
(\F_{x,y} P_{N}u )(\xi, \eta ):= & \psi_{N}(|(\xi, \eta)| )(\F_{x,y}{u})(\xi,\eta ),\\
\widehat{Q_{L} u}(\tau ,\xi, \eta ):= & \psi_{L}(\tau - \xi^3 - \eta^3)\widehat{u}(\tau ,\xi, \eta ).
\end{align*}
We now define $X^{s,b}(\R^3)$ spaces.
Let $s, b \in \R$.
\begin{align*}
&  X^{s,\,b} (\R^3) :=\{ f \in \mathcal{S}'(\R^3) \ | \ \|f \|_{X^{s,\,b}} < \infty \},\\
& \|f  \|_{X^{s,\, b}} := \Bigl( \sum_{N,\, L} N^{2s} L^{2b} \|  P_N Q_L f \|^2_{L_{x, t}^{2}} \Bigr)^{1/2}.
\end{align*}
For convenience, we define the set in frequency as
\begin{equation*}
G_{N, L} := \{ (\tau, \xi, \eta) \in \R^3 \, | \, \psi_{L}(\tau - \xi^3 - \eta^3) \psi_{N}(|(\xi, \eta)| ) \not= 0 \}.
\end{equation*}
Next we observe some fundamental properties of $X^{s,\,b}$.
A simple calculation gives the following.
\begin{equation*}
(i)  \ \
\overline{{X}^{s,\,b}} = X^{s,\,b} , \qquad
(ii)  \ \ (X^{s,\,b})^* =X^{-s,\,-b},
\end{equation*}
for $s$, $b \in \R$.

Recall the Strichartz estimates for the unitary group $\{ e^{-t (\partial_x^3+ \partial_y^3)}\}$.
\begin{lem}[Theorem 3.1. \cite{KPV91}]\label{thm2.1}
Let $\varphi \in L^2(\R^2)$. Then we have
\begin{align}
\| |\nabla_x|^{\frac{1}{2p}} |\nabla_y|^{\frac{1}{2p}} e^{-t (\partial_x^3+ \partial_y^3)} \varphi
\|_{L_t^p L_{x,y}^q} & \lesssim
\|\varphi\|_{L^2_{x,y}}, \quad
\textit{if} \ \ \frac{2}{p} + \frac{2}{q} = 1, \ p >2,\label{Strichartz-01}\\
\|  e^{-t (\partial_x^3+ \partial_y^3)} \varphi  \|_{L_t^p L_{x,y}^q} & \lesssim \|\varphi\|_{L^2_{x,y}}, \quad
\textit{if} \ \ \frac{3}{p} + \frac{2}{q} = 1, \ p >3,\label{Strichartz-02}
\end{align}
where $|\nabla_x|^s := \F^{-1}_x |\xi|^s \F_x$ and $|\nabla_y|^s :=  \F^{-1}_y |\eta|^s \F_y$ denote the Riesz potential operators with respect to $x$ and $y$, respectively.
\end{lem}
The above Strichartz estimates provide the following.
\begin{align}
\| |\nabla_x|^{\frac{1}{2p}} |\nabla_y|^{\frac{1}{2p}} Q_L u \|_{L_t^p L_{x,y}^q} & \lesssim L^{\frac{1}{2}}
\| Q_L u\|_{L^2_{x,y,t}}, \quad
\textnormal{if} \ \ \frac{2}{p} + \frac{2}{q} = 1, \ p >2,\label{Strichartz-1}\\
\|  Q_L u \|_{L_t^p L_{x,y}^q} & \lesssim L^{\frac{1}{2}} \| Q_L u\|_{L^2_{x,y,t}}, \quad
\textnormal{if} \ \ \frac{3}{p} + \frac{2}{q} = 1, \ p >3.\label{Strichartz-2}
\end{align}
See the proof of Lemma 2.3 in \cite{GTV97}. By interpolation \eqref{Strichartz-2} with the trivial equation
$\| Q_L u \|_{L_{x,y,t}^2} = \| Q_L u \|_{L_{x,y,t}^2}$, we get the following.
\begin{equation}
\| Q_L u \|_{L_t^p L_{x,y}^q} \lesssim L^{\frac{2}{3 p} + \frac{1}{q}} \|Q_L u\|_{L_{x,y,t}^2}, \quad
\textnormal{if} \ \ \frac{2}{p} + \frac{2}{q} = 1, \ p \geq 4.\label{Strichartz-3}
\end{equation}
\begin{rem}
Since the estimates \eqref{Strichartz-1}-\eqref{Strichartz-3} are almost equivalent to \eqref{Strichartz-01} and \eqref{Strichartz-02}, we frequently call \eqref{Strichartz-1}-\eqref{Strichartz-3} Strichartz estimates in the paper.
\end{rem}
\section{Proof of the Key estimate}
In this section, we establish the key estimate which gives Theorem \ref{mth} by the standard iteration argument, see \cite{GTV97}, \cite{KPV96}, and \cite{Tao06}.
In this paper, we omit the details of the proof of Theorem \ref{mth} and focus on showing the following key estimate.
\begin{thm}\label{nonlinearity-estimate}
For any $s > -1/4$, there exist $b \in (1/2, 1)$, $\e >0$ and $C>0$ such that
\begin{equation}
\|(\partial_x + \partial_y) (uv) \|_{X^{s,\,b-1+\e}} \quad  \leq C \|u \|_{X^{s,\,b}} \|v \|_{X^{s,\,b}} \label{goal-1-1}.
\end{equation}
\end{thm}
By a duality argument, \eqref{goal-1-1} is implied by
\[
\Bigl| \int{ w (\partial_x + \partial_y) (uv)} dtdxdy \Bigr| \lesssim \|u \|_{X^{s,\,b}} \|v \|_{X^{s,\,b}}
\|w \|_{X^{-s,\,1-b-\e}}.
\]
Furthermore, by dyadic decompositions, it suffices to show
\begin{equation}\label{2018-06-05-01}
\begin{split}
\sum_{{{\substack{N_j, L_j\\ \tiny{(j=0,1,2)}}}}}
 \Bigl|\int{ \left((\partial_x + \partial_y)( Q_{L_0} P_{N_0}w )\right)
(Q_{L_1} P_{N_1} u) (Q_{L_2} P_{N_2}v)
}
dt dx dy & \Bigr|\\
\lesssim \|u \|_{X^{s,\,b}} \|v \|_{X^{s,\,b}}  &
\|w \|_{X^{-s,\,1-b-\e}}.
\end{split}
\end{equation}
Thus, we focus on establishing \eqref{2018-06-05-01} in this section.
For simplicity, throughout the section, we use the notations
\begin{align*}
& L_{012}^{\max} := \max (L_0, L_1, L_2), \quad
N_{012}^{\max} := \max (N_0, N_1, N_2),\\
& w_{N_0, L_0} := Q_{L_0} P_{N_0}w, \ u_{N_1, L_1} := Q_{L_1} P_{N_1} u, \
v_{N_2, L_2} := Q_{L_2} P_{N_2}v.
\end{align*}
We first note that if $N_0 \sim N_1 \sim N_2 \sim 1$ we easily obtain \eqref{2018-06-05-01} by using Strichartz estimates. Thus we assume $1 \ll N_{012}^{\max}$ hereafter.

We divide the proof into the following three cases:

Case 1: high modulation,

Case 2: low modulation, non-parallel interactions,

Case 3: low modulation, parallel interactions.\\
Cases 1, 2, and 3 are treated in Subsections \ref{Case 1}, \ref{Case 2}, and \ref{Case 3}, respectively.
We will introduce the conditions of each case and sketch the outline of the proof of
\eqref{2018-06-05-01} at the beginning of each subsection.
\subsection{Case 1: high modulation}\label{Case 1}
In this subsection we will show \eqref{2018-06-05-01} under the condition $L_{012}^{\max} \gtrsim (N_{012}^{\max})^3$. The proof is quite standard. We just use the Strichartz estimates.
\begin{prop}\label{high-modulations}
Let $L_{012}^{\max} \gtrsim (N_{012}^{\max})^3$. Then \eqref{2018-06-05-01} holds true.
\end{prop}
\begin{proof}
It suffices to show
\begin{equation}
\Bigl|\int{ w_{N_0, L_0}
u_{N_1, L_1}  v_{N_2, L_2}
}
dt dx dy \Bigr|
 \lesssim  (N_{012}^{\max})^{-\frac{5}{4}}  (L_0 L_1 L_2)^{\frac{5}{12}}
\|u \|_{L^2} \|v \|_{L^2}
\|w \|_{L^2}.\label{2018-06-05-02}
\end{equation}
By the Strichartz estimate \eqref{Strichartz-3} with $p=q=4$, we get
\begin{equation}
\|  Q_L  f \|_{L_t^4 L_{x,y}^4} \lesssim L^{\frac{5}{12}}\|f\|_{L^2}. \label{2018-06-05-03}
\end{equation}
If $L_0 = L_{012}^{\max} \gtrsim (N_{012}^{\max})^3$, by H\"{o}lder's inequality and \eqref{2018-06-05-03}, we have
\begin{align*}
\Bigl|\int{ w_{N_0, L_0}
u_{N_1, L_1}  v_{N_2, L_2}
}
dt dx dy \Bigr| & \lesssim \|w_{N_0, L_0} \|_{L^2} \|u_{N_1, L_1} \|_{L^4} \| v_{N_2, L_2} \|_{L^4} \\
 \lesssim & (N_{012}^{\max})^{-\frac{5}{4}}  (L_0 L_1 L_2)^{\frac{5}{12} } \|u \|_{L^2} \|v \|_{L^2}
\|w \|_{L^2}.
\end{align*}
Similarly, we can get \eqref{2018-06-05-02} for $L_1 = L_{012}^{\max}$ or $L_2 = L_{012}^{\max}$.
\end{proof}

\subsection{Case 2: low modulation, non-parallel interactions}\label{Case 2}
Proposition \ref{high-modulations} implies that we only need to consider the case
$L_{012}^{\max} \ll (N_{012}^{\max})^3.$
By Plancherel's theorem, the key estimate \eqref{2018-06-05-01} holds if we show
\[
 \sum_{{{\substack{N_j, L_j\\ \tiny{(j=0,1,2)}}}}}
 \Bigl|\int{  N_0 \, \ha{w}_{{N_0, L_0}}
(
\ha{u}_{N_1, L_1} * \ha{v}_{N_2, L_2}
)}
d\tau d\xi d\eta \Bigr| \lesssim \|u \|_{X^{s,\,b}} \|v \|_{X^{s,\,b}}
\|w \|_{X^{-s,\,1-b-\e}}.
\]
Since $s>-1/4$, it suffices to show
\begin{equation}
\begin{split}
 & \Bigl|\int_{*}{  \ha{w}_{{N_0, L_0}}(\tau, \xi, \eta)
\ha{u}_{N_1, L_1}(\tau_1, \xi_1, \eta_1)  \ha{v}_{N_2, L_2}(\tau_2, \xi_2, \eta_2)
}
d\sigma_1 d\sigma_2 \Bigr| \\
& \qquad \qquad \qquad  \lesssim   (N_{012}^{\max})^{-\frac{5}{4}}  (L_0 L_1 L_2)^{\frac{5}{12}} \|\ha{u}_{N_1, L_1} \|_{L^2} \| \ha{v}_{N_2, L_2} \|_{L^2}
\|\ha{w}_{{N_0, L_0}} \|_{L^2}.\label{desired-est-11-17}
\end{split}
\end{equation}
where $d \sigma_j = d\tau_j d \xi_j d \eta_j$ and $*$ denotes $(\tau, \xi, \eta) = (\tau_1 + \tau_2, \xi_1+ \xi_2, \eta_1 + \eta_2).$
In this subsection, we focus on showing the estimate \eqref{desired-est-11-17} under the following conditions.
\setcounter{case}{1}
\begin{case}[low modulation, non-parallel interactions]\label{ass-1} $\quad$\\
(i) $ \, \, $ $L_{012}^{\max} \leq 2^{-100} (N_{012}^{\max})^3$,\\
(ii) $ \, $ $ N_{012}^{\max} \leq 2^{22} \min(N_0, N_1, N_2) $,\\
(iii) $| \sin \angle \left( (\xi_1, \eta_1), (\xi_2, \eta_2) \right)| \geq 2^{-22}$,\\
where $\angle \left( (\xi_1, \eta_1), (\xi_2, \eta_2) \right) \in [0, \ \pi]$ is the angle between $(\xi_1, \eta_1)$ and $(\xi_2, \eta_2)$.
\end{case}

We remark on the above conditions. First condition is natural since we already saw \eqref{2018-06-05-01}
holds if $L_{012}^{\max} \gtrsim(N_{012}^{\max})^3$. The second and third conditions imply that space frequencies $N_0$, $N_1$, $N_2$ are all high and we do not treat (near) parallel interactions.
In other words, we here consider interactions of three waves which propagate in different directions. See Figure \ref{fig:int} below.
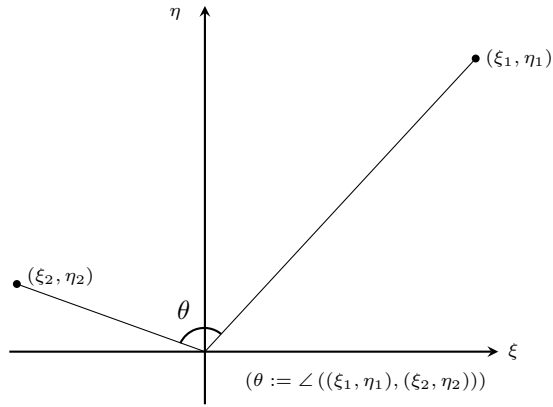
\begin{figure}[H]
\caption{Condition (iii) in \textit{Case} \textnormal{\ref{ass-1}}.}
\label{fig:int}
\centering
\begin{tikzpicture}[thick]
  \useasboundingbox (0,0) rectangle (6.5,6);
   \draw [thick, -stealth](0,0.7)--(6.5,0.7) node [anchor=west, font=\scriptsize]{$\xi$};
   \draw [thick, -stealth](2.6,0)--(2.6,5.3) node [anchor=west, font=\scriptsize]{};
\node [anchor=west, font=\scriptsize] at(2,5.2){$\eta$};
\draw [thin](2.6,0.7)--(6.2,4.6);
\fill (6.2,4.6) circle (1.5pt);
\node [anchor=west, font=\scriptsize] at(6.2,4.6){$(\xi_1,\eta_1)$};
\draw [thin](2.6,0.7)--(0.1,1.6);
\fill (0.1,1.6) circle (1.5pt);
\node [anchor=west, font=\scriptsize] at(0.1,1.7) {$(\xi_2,\eta_2)$};
\draw (2.6+0.36*0.6, 0.7+0.39*0.6) arc (50.1:155.1:0.35cm);
\node [anchor=west] at(2.1,1.25){$\theta$};
\node [anchor=west, font=\scriptsize] at(3,0.3){ $(\theta:= \angle
\left( (\xi_1, \eta_1), (\xi_2, \eta_2) \right))$};
\end{tikzpicture}
\end{figure}


We roughly sketch the outline of the proof.
As we mentioned in the introduction, the convolution estimate on hypersurfaces, which is called the
nonlinear Loomis-Whitney inequality, plays a crucial role. The key ingredient when we apply the
nonlinear Loomis-Whitney is a transversality condition (see $d$ in Proposition \ref{prop2.7} (iii) below) and we will find that it depends on the size of product
\begin{equation*}
|\xi_1 \eta_2 - \xi_2 \eta_1| \, |\xi_1 \eta_2 +  \xi_2 \eta_1 + 2 (\xi_1 \eta_1 + \xi_2 \eta_2)|.
\end{equation*}
The former function $|\xi_1 \eta_2 - \xi_2 \eta_1|$ is comparable to
$N_1^2 | \sin \angle \left( (\xi_1, \eta_1), (\xi_2, \eta_2) \right)|$ if $1 \ll N_1 \sim N_2$,
which implies that, in this subsection, the function depends only on the size of the angle between $(\xi_1, \eta_1)$ and $(\xi_2, \eta_2)$. In \cite{BHHT09}, for the $2$D Zakharov system case, it was found that the transversality condition is comparable to the function $|\xi_1 \eta_2 - \xi_2 \eta_1|$ and therefore the nonlinear Loomis-Whitney inequality could be applied to the non-parallel interactions without difficulties.
Here, however, we need to treat the latter function $|\xi_1 \eta_2 +  \xi_2 \eta_1 + 2 (\xi_1 \eta_1 + \xi_2 \eta_2)|$, which makes the proof complicated.
In this subsection, by the assumptions in \textit{Case} \ref{ass-1}, the former function is harmless.
Thus, the key point in the proof is to handle the latter function in a suitable way.
\begin{figure}[H]
\caption{Square tile decomposition.}
\label{fig:tile}
\centering
\begin{tikzpicture}[thick,xshift=-7]
  \useasboundingbox (7,-0.8) rectangle (14,6);
\draw[step=6mm,gray,very thin] (7.5,2) grid (11.6,5.7);
   \draw [thick, -stealth](7.5,2.4)--(11.6,2.4) node [anchor=west, font=\scriptsize]{};
   \draw [thick, -stealth](8.4,2)--(8.4,5.7) node [anchor=west, font=\scriptsize]{};
\filldraw[fill=gray, draw=black, fill opacity=0.5] (9.6,3.6) rectangle (10.2,4.2);
   \draw [thick, dashed](9.6,3.6)--(11.8,0.2);
   \draw [thick, dashed](10.2,4.2)--(14,2.4);
\filldraw[fill=gray, draw=black, fill opacity=0.5] (11.8,0.2) rectangle (14,2.4);
\node [anchor=west, font=\Large] at(12.4,1.2){$\mathcal{T}_k^A$};
\draw (11.8,0.2) to [out=310,in=175] (12.25,0);
\draw (14,0.2) to [out=230,in=5] (13.55,0);
\node [anchor=west] at(12.2,-0.07){$A^{-1}N_1$};
\end{tikzpicture}
\end{figure}
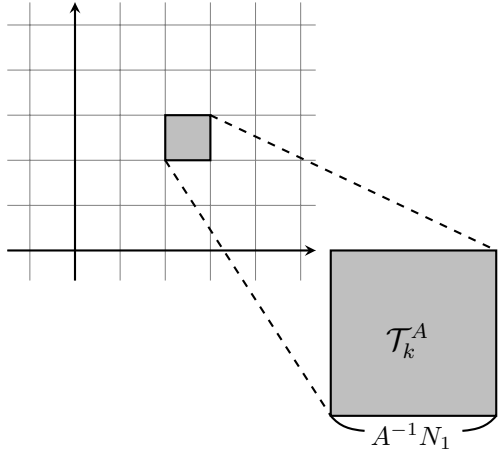
To do so, we first decompose $\R^2$ into square tiles. The decompositions are quite simple. See Figure \ref{fig:tile} above.
\begin{defn}
Let $A \geq 2^{100}$ be a dyadic number and $k = ( k_{(1)}, k_{(2)} ) \in \Z^2$. We define square-tiles
$\{ \mathcal{T}_k^A\}_{k \in \Z^2}$ whose side length is $A^{-1} N_1 $ and prisms $\{ \tilde{\mathcal{T}}_k^A\}_{k \in \Z^2}$ as follows:
\begin{align*}
& \mathcal{T}_k^A : = \{ (\xi,\eta) \in \R^2 \ | \ (\xi,\eta) \in A^{-1} N_1 \bigl(
[ k_{(1)}, k_{(1)} + 1) \times [ k_{(2)}, k_{(2)} + 1) \bigr) \},\\
& \tilde{\mathcal{T}}_k^A  : = \R \times \mathcal{T}_k^A.
\end{align*}
\end{defn}
\begin{rem}
Clearly, we have
\begin{equation*}
\R^2 = \bigcup_{k \in \Z^2} \mathcal{T}_k^A \quad \textnormal{and} \quad  k_1 \not= k_2 \iff
\mathcal{T}_{k_1}^A \cap \mathcal{T}_{k_2}^A = \emptyset,
\end{equation*}
for any $A$.
\end{rem}
The proof of \eqref{2018-06-05-01} in \textit{Case} \ref{ass-1} will be verified by utilizing the two Propositions \ref{prop3.3} and \ref{prop3.4} below with the suitable decomposition which is called the Whitney type decomposition. First, we introduce the bilinear Strichartz estimates. $\chi_A$ denotes the characteristic function of a set $A$.
\begin{prop}\label{prop3.2}
Assume \textnormal{(i)-(iii)} in \textit{Case} \textnormal{\ref{ass-1}}. Let $A \geq 2^{100}$ be dyadic and $k_1$, $k_2 \in \Z^2$. Then we have
\begin{align}
& \Bigl\| \chi_{G_{N_0, L_0}} \int \bigl( \chi_{\tilde{\mathcal{T}}_{k_1}^A}\ha{u}_{N_1, L_1} \bigr)(\tau_1, \xi_1, \eta_1) \bigl( \chi_{\tilde{\mathcal{T}}_{k_2}^A} \ha{v}_{N_2, L_2} \bigr) (\tau- \tau_1, \xi-\xi_1, \eta- \eta_1) d\sigma_1 \Bigr\|_{L_{\xi, \eta, \tau}^2} \notag \\
& \qquad \qquad \qquad \qquad \qquad
\lesssim (AN_1 )^{-\frac{1}{2}} (L_1 L_2)^{\frac{1}{2}} \|\chi_{\tilde{\mathcal{T}}_{k_1}^A}\ha{u}_{N_1, L_1} \|_{L^2}
\|\chi_{\tilde{\mathcal{T}}_{k_2}^A} \ha{v}_{N_2, L_2} \|_{L^2},\label{bilinearStrichartz-1}\\
& \Bigl\| \chi_{G_{N_1, L_1} \cap \tilde{\mathcal{T}}_{k_1}^A} \int \bigl( \chi_{\tilde{\mathcal{T}}_{k_2}^A} \ha{v}_{N_2, L_2}
\bigr)
(\tau_2, \xi_2, \eta_2)
\ha{w}_{N_0, L_0} (\tau_1+ \tau_2, \xi_1+\xi_2, \eta_1+ \eta_2) d\sigma_2 \Bigr\|_{L_{\xi_1, \eta_1, \tau_1}^2}
\notag \\
& \qquad \qquad \qquad \qquad \qquad
\lesssim (AN_1)^{-\frac{1}{2}} (L_0 L_2)^{\frac{1}{2}}
\|\chi_{\tilde{\mathcal{T}}_{k_2}^A} \ha{v}_{N_2, L_2} \|_{L^2}
\|\ha{w}_{N_0, L_0}\|_{L^2}, \label{bilinearStrichartz-2}\\
& \Bigl\| \chi_{G_{N_2,L_2} \cap \tilde{\mathcal{T}}_{k_2}^A} \int \ha{w}_{N_0, L_0} (\tau_1+ \tau_2, \xi_1+\xi_2, \eta_1+ \eta_2)  \bigl( \chi_{\tilde{\mathcal{T}}_{k_1}^A}\ha{u}_{N_1, L_1} \bigr)(\tau_1, \xi_1, \eta_1) d \sigma_1 \Bigr\|_{L_{\xi_2, \eta_2, \tau_2}^2} \notag \\
& \qquad \qquad \qquad \qquad \qquad
\lesssim (AN_1)^{-\frac{1}{2}} (L_0 L_1)^{\frac{1}{2}} \|\ha{w}_{N_0, L_0} \|_{L^2}
\|\chi_{\tilde{\mathcal{T}}_{k_1}^A}\ha{u}_{N_1, L_1} \|_{L^2} .\label{bilinearStrichartz-3}
\end{align}
\end{prop}
\begin{proof}
We consider only \eqref{bilinearStrichartz-1} since \eqref{bilinearStrichartz-2} and \eqref{bilinearStrichartz-3} follow by duality. First we observe that (i)-(iii) in \textit{Case} \textnormal{\ref{ass-1}} give
\begin{equation}
\max \left( |\xi_1^2 - (\xi - \xi_1)^2|, \, |\eta_1^2 - (\eta - \eta_1)^2| \right)  \geq 2^{-200} N_1^2.
\label{est-space-modu}
\end{equation}
Indeed, if \eqref{est-space-modu} does not hold, clearly we may assume one of the following:
\begin{align*}
(1) \quad |\xi_1 - (\xi - \xi_1) | \leq 2^{-100} N_1  \quad \textnormal{and} \quad
|\eta_1 - (\eta - \eta_1) | \leq 2^{-100} N_1,\\
(2) \quad |\xi_1 - (\xi - \xi_1) | \leq 2^{-100} N_1  \quad \textnormal{and} \quad
|\eta_1 + (\eta - \eta_1) | \leq 2^{-100} N_1,\\
(3) \quad |\xi_1 + (\xi - \xi_1) | \leq 2^{-100} N_1  \quad \textnormal{and} \quad
|\eta_1 - (\eta - \eta_1) | \leq 2^{-100} N_1,\\
(4) \quad |\xi_1 + (\xi - \xi_1) | \leq 2^{-100} N_1  \quad \textnormal{and} \quad
|\eta_1 + (\eta - \eta_1) | \leq 2^{-100} N_1.
\end{align*}
It is obvious that both (1) and (4) contradict the assumption (iii). We show (2) contradicts at least one of (i)-(iii) in \textit{Case} \ref{ass-1}. We first observe that
$\max( |\xi_1|, |\xi- \xi_1|) \geq 2^{-30} N_1$, otherwise (iii) will not hold again.
Thus, without loss of generality, we assume $|\xi_1| \geq 2^{-30} N_1$. Since $ |\xi_1 - (\xi - \xi_1) | \leq 2^{-100}  N_1$ in (2), we can see $\min (|\xi|, |\xi- \xi_1|) \geq 2^{-31} N_1$. This and
$|\eta| =  |\eta_1 + (\eta - \eta_1) | \leq 2^{-100} N_1$ in (2) yield
\begin{align*}
 3 \max &  \bigl(|\tau - \xi^3 - \eta^3|, |\tau_1 - \xi_1^3 - \eta_1^3|,
|\tau-\tau_1 - (\xi- \xi_1)^3 - (\eta-\eta_1)^3 |\bigr)\\
& \geq
|\xi \xi_1 (\xi-\xi_1) + \eta \eta_1 (\eta- \eta_1)|\\
& \geq
|\xi \xi_1 (\xi-\xi_1)| - | \eta \eta_1 (\eta- \eta_1)|\\
& \geq 2^{-92} N_1^3 - 2^{-98} N_1^3 \geq 2^{-93} N_1^3
\end{align*}
which contradicts (i).
Similarly, we can show that (3) contradicts at least one of (i)-(iii) in \textit{Case} \ref{ass-1}.
Thus \eqref{est-space-modu} always holds. By symmetry, we may assume $|\xi_1^2 - (\xi - \xi_1)^2|\geq 2^{-200} N_1^2.$

Now we turn to prove \eqref{bilinearStrichartz-1}. By following a standard argument, we get
\begin{align*}
& \Bigl\| \chi_{G_{N_0, L_0}} \int \bigl( \chi_{\tilde{\mathcal{T}}_{k_1}^A}\ha{u}_{N_1, L_1} \bigr)(\tau_1, \xi_1, \eta_1) \bigl( \chi_{\tilde{\mathcal{T}}_{k_2}^A} \ha{v}_{N_2, L_2} \bigr) (\tau- \tau_1, \xi-\xi_1, \eta- \eta_1) d\sigma_1 \Bigr\|_{L_{\xi, \eta, \tau}^2} \\
\leq & \Bigl\| \chi_{G_{N_0, L_0}}   \Bigl(\bigl|
 \chi_{\tilde{\mathcal{T}}_{k_1}^A}\ha{u}_{N_1, L_1} \bigr|^2 *
\bigl| \chi_{\tilde{\mathcal{T}}_{k_2}^A} \ha{v}_{N_2, L_2} \bigr|^2
 \Bigr)^{1/2} |E(\tau, \xi, \eta)|^{1/2} \Bigr\|_{L_{\xi, \eta, \tau}^2} \\
\leq & \sup_{(\tau, \xi, \eta) \in G_{N_0, L_0}} |E(\tau, \xi, \eta)|^{1/2}
 \Bigl\| \bigl|
 \chi_{\tilde{\mathcal{T}}_{k_1}^A}\ha{u}_{N_1, L_1} \bigr|^2 *
\bigl| \chi_{\tilde{\mathcal{T}}_{k_2}^A} \ha{v}_{N_2, L_2} \bigr|^2
\Bigr\|_{L^1_{\xi,\eta,\tau}}^{1/2}\\
\leq & \sup_{(\tau, \xi, \eta) \in G_{N_0, L_0}} |E(\tau, \xi, \eta)|^{1/2}
\|\chi_{\tilde{\mathcal{T}}_{k_1}^A}\ha{u}_{N_1, L_1} \|_{L^2}
\|\chi_{\tilde{\mathcal{T}}_{k_2}^A} \ha{v}_{N_2, L_2} \|_{L^2},
\end{align*}
where $E(\tau, \xi, \eta) \subset \R^3$ is defined by
\begin{equation*}
E(\tau, \xi, \eta) := \{ (\tau_1, \xi_1, \eta_1) \in G_{N_1, L_1} \cap \tilde{\mathcal{T}}_{k_1}^A
\, | \, (\tau-\tau_1, \xi- \xi_1, \eta-\eta_1) \in G_{N_2,L_2} \cap \tilde{\mathcal{T}}_{k_2}^A \}.
\end{equation*}
Thus, it suffices to show
\begin{equation}
\sup_{(\tau, \xi, \eta) \in G_{N_0, L_0}} |E(\tau, \xi, \eta)| \lesssim (A N_1)^{-1} L_1 L_2.\label{est-prop3.2}
\end{equation}
For fixed $(\xi_1, \eta_1)$, we obtain
\begin{equation}
\sup_{(\tau, \xi, \eta) \in G_{N_0, L_0}} | \{ \tau_1 \, | \, (\tau_1, \xi_1, \eta_1) \in E(\tau, \xi, \eta) \}|
\lesssim \min(L_1, L_2).\label{est1-prop3.2}
\end{equation}
Next, since
\begin{align*}
\max (L_1, L_2) &
\gtrsim |(\tau_1 - \xi_1^3 - \eta_1^3) + (\tau- \tau_1) - (\xi- \xi_1)^3 - (\eta- \eta_1)^3|\\
& = |(\tau- \xi^3 -\eta^3) + 3(\xi \xi_1 (\xi-\xi_1) + \eta \eta_1 (\eta- \eta_1))|
\end{align*}
and $ |\partial_{\xi_1} \left( \xi \xi_1 (\xi-\xi_1) \right)| = |\xi_1^2 - (\xi - \xi_1)^2| \gtrsim N_1^{2}$, for fixed $\eta_1$, we get
\begin{equation}
\sup_{(\tau, \xi, \eta) \in G_{N_0, L_0}} | \{ \xi_1 \, | \, (\tau_1, \xi_1, \eta_1) \in E(\tau, \xi, \eta) \}|
\lesssim N_1^{-2} \max(L_1, L_2).\label{est2-prop3.2}
\end{equation}
Lastly, $(\tau_1, \xi_1, \eta_1) \in \tilde{\mathcal{T}}_{k_1}^A$ gives
\begin{equation}
\sup_{(\tau, \xi, \eta) \in G_{N_0, L_0}} | \{ \eta_1 \, | \, (\tau_1, \xi_1, \eta_1) \in E(\tau, \xi, \eta) \}|
\lesssim N_1A^{-1}.\label{est3-prop3.2}
\end{equation}
The estimates \eqref{est1-prop3.2}-\eqref{est3-prop3.2} complete the proof of \eqref{est-prop3.2}.
\end{proof}
By Proposition \ref{prop3.2}, we immediately obtain the following estimate.
\begin{prop}\label{prop3.3}
Assume \textnormal{(i)-(iii)} in \textit{Case} \textnormal{\ref{ass-1}}. Let $A \geq 2^{100}$ be dyadic. Suppose that $k_1$, $k_2 \in \Z^2$ satisfy
\begin{equation*}
| \xi_1 \xi_2 (\xi_1+\xi_2) +  \eta_1 \eta_2 (\eta_1+\eta_2)| \geq A^{-1} N_1^3 \quad \textnormal{for any } \ (\xi_j, \eta_j) \in
\mathcal{T}_{k_j}^A
\end{equation*}
where $j=1,2$. Then we have
\begin{align*}
& \Bigl|\int_{*}{  \ha{w}_{{N_0, L_0}}(\tau, \xi, \eta)
\ha{u}_{N_1, L_1}|_{\tilde{\mathcal{T}}_{k_1}^A}(\tau_1, \xi_1, \eta_1)  \ha{v}_{N_2, L_2}|_{\tilde{\mathcal{T}}_{k_2}^A}(\tau_2, \xi_2, \eta_2)
}
d\sigma_1 d\sigma_2 \Bigr|\\
& \qquad \qquad \lesssim N_1^{-2} (L_0 L_1 L_2)^{1/2} \|\ha{u}_{N_1, L_1}|_{\tilde{\mathcal{T}}_{k_1}^A} \|_{L^2}
\| \ha{v}_{N_2, L_2}|_{\tilde{\mathcal{T}}_{k_2}^A} \|_{L^2}
\|\ha{w}_{{N_0, L_0}} \|_{L^2}
\end{align*}
where $d \sigma_j = d\tau_j d \xi_j d \eta_j$ and $*$ denotes $(\tau, \xi, \eta) = (\tau_1 + \tau_2, \xi_1+ \xi_2, \eta_1 + \eta_2).$
\end{prop}
\begin{proof}
We see that
\begin{align*}
3 L_{012}^{\max} & \geq | \xi_1 \xi_2 (\xi_1+\xi_2) +  \eta_1 \eta_2 (\eta_1+\eta_2)| \\
& \geq A^{-1}N_1^3.
\end{align*}
For $L_0 = L_{012}^{\max}$, by using \eqref{bilinearStrichartz-1} in Proposition \ref{prop3.2}, we get
\begin{align*}
& \Bigl|\int_{*}{  \ha{w}_{{N_0, L_0}}(\tau, \xi, \eta)
\ha{u}_{N_1, L_1}|_{\tilde{\mathcal{T}}_{k_1}^A}(\tau_1, \xi_1, \eta_1)  \ha{v}_{N_2, L_2}|_{\tilde{\mathcal{T}}_{k_2}^A}(\tau_2, \xi_2, \eta_2)
}
d\sigma_1 d\sigma_2 \Bigr|\\
& \leq ( \textnormal{(LHS) of \eqref{bilinearStrichartz-1}}) \|\ha{w}_{{N_0, L_0}} \|_{L^2}   \\
& \lesssim (N_1 A)^{-\frac{1}{2}} (L_1 L_2)^{\frac{1}{2}} \|\ha{u}_{N_1, L_1}|_{\tilde{\mathcal{T}}_{k_1}^A} \|_{L^2}
\| \ha{v}_{N_2, L_2}|_{\tilde{\mathcal{T}}_{k_2}^A} \|_{L^2}  \|\ha{w}_{{N_0, L_0}} \|_{L^2}   \\
& \lesssim N_1^{-2} (L_0 L_1 L_2)^{\frac{1}{2}} \|\ha{u}_{N_1, L_1}|_{\tilde{\mathcal{T}}_{k_1}^A} \|_{L^2}
\| \ha{v}_{N_2, L_2}|_{\tilde{\mathcal{T}}_{k_2}^A} \|_{L^2}
\|\ha{w}_{{N_0, L_0}} \|_{L^2}.
\end{align*}
Similarly, in the cases $L_1= L_{012}^{\max}$ and $L_2= L_{012}^{\max}$,
by utilizing \eqref{bilinearStrichartz-2}
and \eqref{bilinearStrichartz-3}, respectively, we obtain the desired estimate.
\end{proof}
The following proposition enables us to deal with near-resonant interactions.
\begin{prop}\label{prop3.4}
Assume \textnormal{(i)-(iii)} in \textit{Case} \textnormal{\ref{ass-1}}. Let $A \geq 2^{100}$ be dyadic. Suppose that $k_1$, $k_2 \in \Z^2$ satisfy
\begin{equation*}
| \xi_1 \eta_2 +  \xi_2 \eta_1 + 2 (\xi_1 \eta_1 + \xi_2 \eta_2)| \geq A^{-1} N_1^2 \quad \textnormal{for any } \ (\xi_j, \eta_j) \in
{\mathcal{T}}_{k_j}^A
\end{equation*}
where $j=1,2$. Then we have
\begin{align*}
& \Bigl|\int_{*}{  \ha{w}_{{N_0, L_0}}(\tau, \xi, \eta)
\ha{u}_{N_1, L_1}|_{\tilde{\mathcal{T}}_{k_1}^A}(\tau_1, \xi_1, \eta_1)  \ha{v}_{N_2, L_2}|_{\tilde{\mathcal{T}}_{k_2}^A}(\tau_2, \xi_2, \eta_2)
}
d\sigma_1 d\sigma_2 \Bigr|\\
& \qquad \qquad \lesssim A^{\frac{1}{2}} N_1^{-2} (L_0 L_1 L_2)^{\frac{1}{2}} \|\ha{u}_{N_1, L_1}|_{\tilde{\mathcal{T}}_{k_1}^A} \|_{L^2}
\| \ha{v}_{N_2, L_2}|_{\tilde{\mathcal{T}}_{k_2}^A} \|_{L^2}
\|\ha{w}_{{N_0, L_0}} \|_{L^2}
\end{align*}
where $d \sigma_j = d\tau_j d \xi_j d \eta_j$ and $*$ denotes $(\tau, \xi, \eta) = (\tau_1 + \tau_2, \xi_1+ \xi_2, \eta_1 + \eta_2).$
\end{prop}
For the proof of the above proposition, we now recall the nonlinear version of the classical Loomis-Whitney inequality. See \cite{BH11} for a more general version.
\begin{prop}[\cite{BHT10} Corollary 1.5] \label{prop2.7}
Assume that the surface $S_i$ $(i=1,2,3)$
is an open and bounded subset of $S_i^*$ which
satisfies the following conditions \textnormal{(Assumption 1.1 in \cite{BHT10})}.

\textnormal{(i)} $S_i^*$ is defined as
\begin{equation*}
S_i^* = \{ {\lambda_i} \in U_i \ | \ \Phi_i({\lambda_i}) = 0 , \nabla \Phi_i \not= 0, \Phi_i \in C^{1,1} (U_i) \},
\end{equation*}
for a convex $U_i \subset \R^3$ such that \textnormal{dist}$(S_i, U_i^c) \geq$ \textnormal{diam}$(S_i)$;

\textnormal{(ii)} the unit normal vector field $\mathfrak{n}_i$ on $S_i^*$ satisfies the H\"{o}lder condition
\begin{equation*}
\sup_{\lambda, \lambda' \in S_i^*} \frac{|\mathfrak{n}_i(\lambda) -
\mathfrak{n}_i(\lambda')|}{|\lambda - \lambda'|}
+ \frac{|\mathfrak{n}_i(\lambda) ({\lambda} - {\lambda}')|}{|{\lambda} - {\lambda}'|^2} \lesssim 1;
\end{equation*}

\textnormal{(iii)} there exists $d >0$ such that the matrix ${N}({\lambda_1}, {\lambda_2}, {\lambda_3}) = ({\mathfrak{n}_1}
({\lambda_1}), {\mathfrak{n}_2}({\lambda_2}), {\mathfrak{n}_3}({\lambda_3}))$
satisfies the transversality condition
\begin{equation*}
d \leq |\textnormal{det} {N}({\lambda_1}, {\lambda_2}, {\lambda_3}) | \leq 1
\end{equation*}
for all $({\lambda_1}, {\lambda_2}, {\lambda_3}) \in {S_1^*} \cross {S_2^*} \cross {S_3^*}$.

We also assume \textnormal{diam}$({S_i}) \lesssim d$.
Then for functions $f \in L^2 (S_1)$ and $g \in L^2 (S_2)$, the restriction of the convolution $f*g$ to
$S_3$ is a well-defined $L^2(S_3)$-function which satisfies
\begin{equation*}
\| f *g \|_{L^2(S_3)} \lesssim \frac{1}{\sqrt{d}} \| f \|_{L^2(S_1)} \| g\|_{L^2(S_2)}.
\end{equation*}
\end{prop}
\begin{proof}[Proof of Proposition \ref{prop3.4}]
Since $\ha{u}_{N_1, L_1}|_{\tilde{\mathcal{T}}_{k_1}^A}$ and $\ha{v}_{N_2, L_2}|_{\tilde{\mathcal{T}}_{k_2}^A}$ are supported in
$\tilde{\mathcal{T}}_{k_1}^A$ and ${\tilde{\mathcal{T}}}_{k_2}^A$, respectively, we may assume that there exists $k_3 \in \Z^2$ such that $\ha{w}_{{N_0, L_0}}$ is supported in $\tilde{\mathcal{T}}_{k_3}^A$.
In addition, by performing a harmless decomposition, $\ha{u}_{N_1, L_1}$, $\ha{v}_{N_2, L_2}$ and
$\ha{w}_{{N_0, L_0}}$ are supported in square prisms whose square's side length is $2^{-100} N_1 A^{-1}$, respectively.
Let $A' = 2^{-100}A$ and put $f$, $g$, $h \in L^2(\R^3)$ which satisfy
\begin{equation*}
\operatorname{supp} f \subset G_{N_1, L_1} \cap \tilde{\mathcal{T}}_{k_1}^{A'}, \quad
\operatorname{supp} g \subset G_{N_2, L_2} \cap \tilde{\mathcal{T}}_{k_2}^{A'}, \quad
\operatorname{supp} h \subset G_{N_0, L_0} \cap \tilde{\mathcal{T}}_{k_3}^{A'}.
\end{equation*}
By the above observation, it suffices to show that if
\begin{align}
| \sin \angle \left( (\xi_1, \eta_1), (\xi_2, \eta_2) \right)| & \geq 2^{-23},\label{ang-condition}\\
| \xi_1 \eta_2 +  \xi_2 \eta_1 + 2 (\xi_1 \eta_1 + \xi_2 \eta_2)| & \geq A^{-1} N_1^2\label{trans-condition},
\end{align}
for any $(\xi_j, \eta_j) \in \mathcal{T}_{k_j}^{A'}$ with $j=1,2$, then the following estimate holds
\begin{equation}
\begin{split}
& \Bigl|\int_{\R^3 \times \R^3} { h (\tau_1+\tau_2, \xi_1+\xi_2, \eta_1+\eta_2)
f(\tau_1, \xi_1, \eta_1)  g(\tau_2, \xi_2, \eta_2)
}
d\sigma_1 d\sigma_2 \Bigr|\\
& \qquad \qquad \qquad \qquad \lesssim A^{\frac{1}{2}} N_1^{-2} (L_0 L_1 L_2)^{\frac{1}{2}} \|f \|_{L^2}
\|g\|_{L^2}
\| h \|_{L^2}.
\end{split}\label{desired-est-prop3.4}
\end{equation}
We apply the same strategy as
that of the proof of
Proposition 4.4 in \cite{BHHT09}.
Applying the transformation $\tau_1 = \xi_1^3 + \eta_1^3 + c_1$ and $\tau_2 = \xi_2^3 + \eta_2^3 + c_2$ and Fubini's theorem, we find that it suffices to prove
\begin{equation}
\begin{split}
& \Bigl| \int h (\phi_{c_1} (\xi_1, \eta_1) + \phi_{c_2} (\xi_2, \eta_2))  f (\phi_{c_1} (\xi_1, \eta_1) ) g (\phi_{c_2}(\xi_2, \eta_2)) d \xi_1d \eta_1 d\xi_2 d\eta_2 \Bigr| \\
&\qquad \qquad \qquad \qquad \qquad \qquad  \lesssim  A^{\frac{1}{2}}N_1^{-2} \| f \circ \phi_{c_1}\|_{L_{\xi, \eta}^2} \|g \circ \phi_{c_2} \|_{L_{\xi, \eta}^2} \|h \|_{L_{\xi, \eta, \tau}^2}, \label{est01-prop3.4}
\end{split}
\end{equation}
where $h(\tau, \xi, \eta)$ is supported in $c_0 \leq \tau - \xi^3 - \eta^3 \leq c_0 +1$ and
\begin{equation*}
\phi_{c_j} (\xi) = (\xi^3 + \eta^3 + c_k, \, \xi, \, \eta) \quad \textnormal{for} \ j=1,2.
\end{equation*}
We use the scaling $(\tau, \, \xi, \, \eta) \to (N_1^3 \tau , \, N_1 \xi, \, N_1 \eta)$ to define
\begin{align*}
 \tilde{f} (\tau_1, \xi_1 , \eta_1) & = f (N_1^3 \tau_1 , N_1 \xi_1, N_1 \eta_1), \\
 \tilde{g} (\tau_2, \xi_2, \eta_2) & = g (N_1^3 \tau_2, N_1 \xi_2, N_1 \eta_2), \\
 \tilde{h} (\tau, \xi, \eta) & = h (N_1^3 \tau, N_1 \xi, N_1 \eta).
\end{align*}
If we set $\tilde{c_j} = N_1^{-3} c_j$, inequality (\ref{est01-prop3.4}) reduces to
\begin{equation*}
\begin{split}
& \Bigl| \int \tilde{h} (\phi_{\tilde{c_1}} (\xi_1, \eta_1) + \phi_{\tilde{c_2}} (\xi_2, \eta_2))  \tilde{f}
 (\phi_{\tilde{c_1}} (\xi_1, \eta_1) ) \tilde{g} (\phi_{\tilde{c_2}}(\xi_2, \eta_2)) d \xi_1d \eta_1 d\xi_2 d\eta_2 \Bigr| \\
&\qquad \qquad \qquad \qquad \qquad \qquad  \lesssim  A^{\frac{1}{2}}N_1^{-\frac{3}{2}} \| \tilde{f}
 \circ \phi_{\tilde{c_1}}\|_{L_{\xi, \eta}^2}
\| \tilde{g} \circ \phi_{\tilde{c_2}} \|_{L_{\xi, \eta}^2} \|\tilde{h} \|_{L_{\xi, \eta, \tau}^2},
\end{split}
\end{equation*}
Note that $\operatorname{supp} \tilde{f} \subset {\tilde{\mathcal{T}}}_{\tilde{k_1}}^{N_1^{-1}A'}$,
$\operatorname{supp} \tilde{g} \subset {\tilde{\mathcal{T}}}_{\tilde{k_2}}^{N_1^{-1}A'}$, and $\tilde{h}$ is supported in $S_3 (N_1^{-3})$ where
\begin{equation*}
S_3 (N_1^{-3}) = \Bigl\{ (\tau, \xi, \eta) \in \tilde{\mathcal{T}}_{\tilde{k_3}}^{N_1^{-1}A'}
\ | \  \xi^3 + \eta^3 + \frac{c_0}{N_1^{3}}  \leq \tau \leq \xi^3 + \eta^3 + \frac{c_0+1}{N_1^{3}} \Bigr\}
\end{equation*}
where $\tilde{k_i} = k_i/N_1$ with $i = 1,2,3$.
By density and duality, it suffices to show for continuous
$\tilde{f}$ and $\tilde{g}$ that
\begin{equation}
\| \tilde{f} |_{S_1} * \tilde{g} |_{S_2} \|_{L^2(S_3 (N_1^{-3}))} \lesssim A^{\frac{1}{2}} N_1^{-\frac{3}{2}}
\| \tilde{f} \|_{L^2(S_1)} \| \tilde{g} \|_{L^2(S_2)}\label{est03-prop3.4}
\end{equation}
where $S_1$, $S_2$ denote the following surfaces
\begin{align*}
S_1 =&  \bigl\{ \phi_{\tilde{c_1}} (\xi_1, \eta_1) \in \R^3 \ | \
(\xi_1, \eta_1) \in \mathcal{T}_{\tilde{k_1}}^{N_1^{-1}A'} \bigr\}, \\
S_2 =&  \bigl\{ \phi_{\tilde{c_2}} (\xi_2, \eta_2) \in \R^3 \ | \ (\xi_2, \eta_2) \in \mathcal{T}_{\tilde{k_2}}^{N_1^{-1}A'} \bigr\}.
\end{align*}
(\ref{est03-prop3.4}) is immediately established by the following.
\begin{equation}
\| \tilde{f} |_{S_1} * \tilde{g} |_{S_2} \|_{L^2(S_3)} \lesssim A^{\frac{1}{2}}
\| \tilde{f} \|_{L^2(S_1)} \| \tilde{g} \|_{L^2(S_2)}\label{est04-prop3.4}
\end{equation}
where
\begin{equation*}
S_3 = \Bigl\{ (\psi (\xi,\eta), \xi, \eta) \in \R^3  \ | \  (\xi, \eta) \in \mathcal{T}_{\tilde{k_3}}^{N_1^{-1}A'},
\  \psi (\xi,\eta) =  \xi^3 + \eta^3 + \frac{c_0'}{N_1^{3}} \Bigr\},
\end{equation*}
for any fixed $c_0' \in [c_0, \, c_0 +1]$.
We deduce from the assumption (ii) in \textit{Case} \textnormal{\ref{ass-1}} that
\begin{equation}
\textnormal{diam} (S_i) \leq 2^{-80} A^{-1} \qquad \textnormal{for} \ \, i=1,2,3.\label{diam-prop3.4}
\end{equation}
For any $\lambda_i \in S_i$, there exist $(\xi_1, \eta_1)$, $(\xi_2, \eta_2)$, $(\xi, \eta)$ such that
\begin{equation*}
\lambda_1=\phi_{\tilde{c_1}} (\xi_1, \eta_1), \quad \lambda_2 =  \phi_{\tilde{c_2}} (\xi_2,\eta_2), \quad
\lambda_3 = (\psi (\xi,\eta), \xi,\eta),
\end{equation*}
and the unit normals ${\mathfrak{n}}_i$ on $\lambda_i$ are written as
\begin{align*}
& {\mathfrak{n}}_1(\lambda_1) = \frac{1}{\sqrt{1+ 9 \xi_1^4 + 9 \eta_1^4}}
\left(-1, \ 3  \xi_1^2, \ 3 \eta_1^2 \right), \\
& {\mathfrak{n}}_2(\lambda_2) = \frac{1}{\sqrt{1+ 9 \xi_2^4 + 9 \eta_2^4}}
\left(-1, \ 3  \xi_2^2, \ 3 \eta_2^2 \right),\\
& {\mathfrak{n}}_3(\lambda_3) = \frac{1}{\sqrt{1+ 9 \xi^4 + 9 \eta^4}}
\left(-1, \ 3  \xi^2, \ 3 \eta^2 \right).
\end{align*}
Clearly, the surfaces $S_1$, $S_2$, $S_3$ satisfy the following
H\"{o}lder condition.
\begin{equation}
\sup_{\lambda_i, \lambda_i' \in S_i} \frac{|\mathfrak{n}_i(\lambda_i) -
\mathfrak{n}_i(\lambda_i')|}{|\lambda_i - \lambda_i'|}
+ \frac{|\mathfrak{n}_i(\lambda_i) (\lambda_i - \lambda_i')|}{|\lambda_i - \lambda_i'|^2} \leq 2^3.\label{normals00-prop3.4}
\end{equation}
We may assume that there exist $(\xi_1', \eta_1')$, $(\xi_2', \eta_2')$, $(\xi', \eta')$ such that
\begin{align*}
& (\xi_1', \eta_1') + (\xi_2', \eta_2') = (\xi', \eta'), \\
\phi_{\tilde{c_1}} (\xi_1', \eta_1') & \in S_1, \ \ \phi_{\tilde{c_2}} (\xi_2', \eta_2')\in S_2, \ \
(\psi (\xi',\eta'), \xi', \eta') \in S_3,
\end{align*}
otherwise the left-hand side of \eqref{est04-prop3.4} vanishes.
Let $\lambda_1' = \phi_{\tilde{c_1}} (\xi_1', \eta_1')$, $\lambda_2' = \phi_{\tilde{c_2}} (\xi_2', \eta_2')$,
$\lambda_3' = (\psi (\xi',\eta'), \xi',\eta')$.
For any $\lambda_1 = \phi_{{{\tilde{c_1}}}}(\xi_1, \eta_1) \in S_1$,
we deduce from $\lambda_1$, $\lambda_1' \in S_1$
and \eqref{diam-prop3.4} that
\begin{equation}
|{\mathfrak{n}}_1(\lambda_1) - {\mathfrak{n}}_1(\lambda_1')| \leq 2^{-70} A^{-1}.\label{normal01-prop3.4}
\end{equation}
Similarly, for any $\lambda_2 \in S_2$ and $\lambda_3 \in S_3$ we have
\begin{align}
& |{\mathfrak{n}}_2(\lambda_2) - {\mathfrak{n}}_2(\lambda_2')| \leq 2^{-70} A^{-1}.\label{normal02-prop3.4}\\
& |{\mathfrak{n}}_3(\lambda_3) - {\mathfrak{n}}_3(\lambda_3')| \leq 2^{-70} A^{-1}.\label{normal03-prop3.4}
\end{align}
From \eqref{diam-prop3.4} and \eqref{normals00-prop3.4},
once the following transversality condition \eqref{trans-prop3.4} is verified, we obtain the desired estimate \eqref{est04-prop3.4}  by applying Proposition \ref{prop2.7} with $d = 2^{-70} A^{-1}$.
\begin{equation}
2^{-70} A^{-1} \leq |\textnormal{det} N(\lambda_1, \lambda_2, \lambda_3)| \quad
\textnormal{for any} \ \lambda_i \in S_i.\label{trans-prop3.4}
\end{equation}
From \eqref{normal01-prop3.4}-\eqref{normal03-prop3.4}, it suffices to show
\begin{equation*}
2^{-65} A^{-1} \leq |\textnormal{det} N(\lambda_1', \lambda_2', \lambda_3')| .
\end{equation*}
We deduce from $\lambda_1' =  \phi_{\tilde{c_1}} (\xi_1', \eta_1')$, $\lambda_2' = \phi_{\tilde{c_2}} (\xi_2', \eta_2')$, $\lambda_3' =  (\psi (\xi',\eta'), \xi',\eta') $ and
$(\xi_1', \eta_1') + (\xi_2', \eta_2') = (\xi', \eta')$ that
\begin{align*}
|\textnormal{det} N(\lambda_1', \lambda_2', \lambda_3')| \geq &
2^{-25} \frac{1}{\LR{(\xi_1, \eta_1) }^2 \LR{(\xi_2, \eta_2)}^2} \left|\textnormal{det}
\begin{pmatrix}
-1 & -1 & - 1 \\
 3 (\xi_1')^2  &  3 (\xi_2')^2  & 3 (\xi')^2 \\
 3 (\eta_1')^2   & 3 (\eta_2')^2  & 3 (\eta')^2
\end{pmatrix} \right| \notag \\
\geq & 2^{-25}\frac{|\xi_1' \eta_2' - \xi_2' \eta_1' |}{\LR{(\xi_1, \eta_1) }^2 \LR{(\xi_2, \eta_2)}^2}
| \xi_1' \eta_2' +  \xi_2' \eta_1' + 2 (\xi_1' \eta_1' + \xi_2' \eta_2')|
\\
\geq & 2^{-25} \frac{|(\xi_1', \eta_1')| \, |(\xi_2', \eta_2')|}{\LR{(\xi_1', \eta_1') }^2 \LR{(\xi_2', \eta_2')}^2}
\frac{|\xi_1' \eta_2' - \xi_2' \eta_1' |}{|(\xi_1', \eta_1')| \, |(\xi_2', \eta_2')|}A^{-1}\\
\geq & 2^{-65}A^{-1}.
\end{align*}
Here we used \eqref{ang-condition}, \eqref{trans-condition} and
\begin{equation*}
\left| \sin \angle \left( (\xi_1', \eta_1'), (\xi_2', \eta_2') \right) \right|
= \frac{|\xi_1' \eta_2' - \xi_2' \eta_1' |}{|(\xi_1', \eta_1')| \, |(\xi_2', \eta_2')|}.
\end{equation*}
\end{proof}
In Propositions \ref{prop3.3} and \ref{prop3.4}, we assume that the supports of $\ha{u}_{N_1, L_1}$ and
$\ha{v}_{N_2, L_2}$ are restricted to the square prisms $\tilde{\mathcal{T}}_{k_1}^A$ and ${\tilde{\mathcal{T}}}_{k_2}^A$, respectively. If we simply try to sum up the prisms, the regularity loss which depends on $A$ appears.
Thus, we introduce a suitable decomposition which is called the Whitney type decomposition.
This decomposition depends on the two functions
$\xi_1 \xi_2 (\xi_1+\xi_2) +  \eta_1 \eta_2 (\eta_1+\eta_2)$ and
$\xi_1 \eta_2 +  \xi_2 \eta_1 + 2 (\xi_1 \eta_1 + \xi_2 \eta_2)$ which appear in Propositions \ref{prop3.3} and \ref{prop3.4}, respectively.
As we mentioned, this is a similar strategy to that for the Zakharov system.
In \cite{BHHT09}, the Whitney decomposition of angular variables was performed.
\begin{defn}[Whitney type decomposition]
 Let $A \geq 2^{100}$ be dyadic and
\begin{align*}
\Phi (\xi_1, \eta_1, \xi_2 ,\eta_2) & = \xi_1 \xi_2(\xi_1 + \xi_2) + \eta_1 \eta_2 (\eta_1 + \eta_2), \\
F (\xi_1, \eta_1, \xi_2 ,\eta_2) & = \xi_1 \eta_2 +  \xi_2 \eta_1 + 2 (\xi_1 \eta_1 + \xi_2 \eta_2).
\end{align*}
We define
\begin{align*}
Z_A^1 & = \{ (k_1, k_2) \in \Z^2 \times \Z^2 \, | \,
|\Phi(\xi_1, \eta_1, \xi_2, \eta_2)| \geq A^{-1} N_1^3  \ \ \textnormal{for any} \ (\xi_j, \eta_j) \in
\mathcal{T}_{k_j}^A \},\\
Z_A^2  & = \{ (k_1, k_2) \in \Z^2 \times \Z^2 \, | \,
| F (\xi_1, \eta_1, \xi_2 ,\eta_2)| \geq A^{-1} N_1^2  \ \ \textnormal{for any} \ (\xi_j, \eta_j) \in
\mathcal{T}_{k_j}^A \},\\
Z_A & = Z_A^1 \cup Z_A^2 \subset \Z^2 \times \Z^2,
\qquad R_A = \bigcup_{(k_1, k_2) \in Z_A} \mathcal{T}_{k_1}^A \times
\mathcal{T}_{k_2}^A \subset \R^2 \times \R^2.
\end{align*}
It is clear that $A_1 \leq A_2 \Longrightarrow R_{A_1} \subset R_{A_2}$.
Further, we define
\begin{equation*}
Q_A =
\begin{cases}
R_A \setminus R_{\frac{A}{2}} \quad \textnormal{for} \  A \geq 2^{101},\\
 \ R_{2^{100}}  \ \qquad \textnormal{for} \  A = 2^{100},
\end{cases}
\end{equation*}
and a set of pairs of integer pair $Z_A' \subset Z_A$ as
\begin{equation*}
\bigcup_{(k_1, k_2) \in Z_A'} \mathcal{T}_{k_1}^A \times
\mathcal{T}_{k_2}^A = Q_A.
\end{equation*}
We easily see that $Z_A'$ is uniquely defined and
\begin{equation*}
A_1 \not= A_2 \Longrightarrow Q_{A_1} \cap Q_{A_2} = \emptyset, \quad
\bigcup_{2^{100} \leq A \leq A_0} Q_{A} = R_{A_0}
\end{equation*}
where $A_0 \geq 2^{100}$ is dyadic. Thus, we can decompose $\R^2 \times \R^2$ as
\begin{equation*}
\R^2 \times \R^2 = \biggl( \bigcup_{2^{100} \leq A \leq A_0} Q_{A}\biggr) \cup (R_{A_0})^c.
\end{equation*}
Lastly, we define
\begin{align*}
\mathcal{A}  & = \{ (\tau_1, \xi_1, \eta_1) \times (\tau_2, \xi_2, \eta_2) \in \R^3 \times \R^3 \, | \,
| \sin \angle \left( (\xi_1, \eta_1), (\xi_2, \eta_2) \right)| \geq 2^{-22}  \},\\
\tilde{Z}_{A} & = \{ (k_1, k_2) \in Z_A' \, | \,
\bigl( \tilde{\mathcal{T}}_{k_1}^A \times \tilde{\mathcal{T}}_{k_2}^A\bigr) \cap \left(
G_{N_1, L_1} \times G_{N_2, L_2} \right) \cap \mathcal{A} \not=
\emptyset \}.
\end{align*}
Here we assume that $N_1$, $L_1$, $N_2$, $L_2$ satisfy (i) and (ii) in \textit{Case} \textnormal{\ref{ass-1}}.
\end{defn}
To sum up tiles without loss, the following condition, the almost orthogonality of $k_1$ and $k_2$ such that
$(k_1, k_2) \in \tilde{Z}_{A}$ seems to be necessary.
\begin{condition}\label{condition-1} $\quad$\\
For fixed $k_1 \in \Z^2$, the number of $k_2 \in \Z^2$ such that
$(k_1, k_2) \in \tilde{Z}_{A}$ is less than $2^{1000}$.
\end{condition}
\begin{rem}
By symmetry, one can see that for fixed  $k_2 \in \Z^2$, the number of $k_1 \in \Z^2$ such that
$(k_1, k_2) \in \tilde{Z}_{A}$ is less than $2^{1000}$. Thus we can say that $k_1$ and $k_2$ which satisfy $(k_1, k_2) \in \tilde{Z}_{A}$ are almost one-to-one correspondence.
\end{rem}
Unfortunately, \textit{Condition} \ref{condition-1} does not hold true for certain pairs of tiles,
which will be observed in \textit{Remark} \ref{pairs-of-tiles} below.
Thus, we first exclude such pairs of tiles and, to do so, we divide
$\R^2 \times \R^2$.
\begin{defn}
Let $\mathcal{K}_0$, $\mathcal{K}_1$, $\mathcal{K}_2$, $\mathcal{K}_0'$, $\mathcal{K}_1'$,
$\mathcal{K}_2' \subset \R^2$ and  $\tilde{\mathcal{K}}_0$, $\tilde{\mathcal{K}}_1$,
$\tilde{\mathcal{K}}_2$,
$\tilde{\mathcal{K}}_0'$,
$\tilde{\mathcal{K}}_1'$, $\tilde{\mathcal{K}}_2' \subset \R^3$ be defined as follows:
\begin{align*}
\mathcal{K}_0 & = \left\{ (\xi, \eta) \in \R^2 \, | \, \left|
\eta -(\sqrt{2} - 1)^{\frac{4}{3}} \xi \right|
\leq 2^{-20} N_1 \right\},\\
\mathcal{K}_1 & = \left\{ (\xi, \eta) \in \R^2 \, | \, \left|
\eta - ( \sqrt{2}+ 1 )^{\frac{2}{3}} (\sqrt{2} + \sqrt{3} ) \xi \right|
\leq 2^{-20} N_1 \right\},\\
\mathcal{K}_2 & = \left\{ (\xi, \eta) \in \R^2 \, | \, \left|
\eta + ( \sqrt{2}+ 1 )^{\frac{2}{3}} (\sqrt{3} - \sqrt{2} ) \xi \right|
\leq 2^{-20} N_1 \right\},\\
\mathcal{K}_0' & =  \left\{ (\xi, \eta) \in \R^2 \ | \
 (\eta, \xi) \in \mathcal{K}_0  \right\},\\
\mathcal{K}_1' & = \left\{ (\xi, \eta) \in \R^2 \ | \ (\eta, \xi) \in \mathcal{K}^1 \right\},\\
\mathcal{K}_2' & = \left\{ (\xi, \eta) \in \R^2 \ | \ (\eta, \xi) \in \mathcal{K}^2 \right\},\\
\tilde{\mathcal{K}}_j & = \R \times \mathcal{K}_j, \quad \tilde{\mathcal{K}}_j' = \R \times \mathcal{K}_j'
 \ \ \textnormal{for} \ j = 0, 1,2.
\end{align*}
We define the subsets of $\R^2 \times \R^2$ and $\R^3 \times \R^3$ as
\begin{align*}
\mathcal{K} \, = & ( \mathcal{K}_0 \times ( \mathcal{K}_1\cup \mathcal{K}_2 ) ) \cup
( ( \mathcal{K}_1\cup \mathcal{K}_2 ) \times  \mathcal{K}_0 ) \subset \R^2 \times \R^2,\\
\tilde{\mathcal{K}} \, = & ( \tilde{\mathcal{K}}_0 \times ( \tilde{\mathcal{K}}_1
\cup \tilde{\mathcal{K}}_2 ) ) \cup
(
( \tilde{\mathcal{K}}_1\cup \tilde{\mathcal{K}}_2 ) \times  \tilde{\mathcal{K}}_0 ) \subset \R^3 \times \R^3,\\
\mathcal{K}' = & \left( \mathcal{K}_0' \times \left( \mathcal{K}_1' \cup \mathcal{K}_2' \right) \right) \cup
\left(
\left( \mathcal{K}_1' \cup \mathcal{K}_2' \right) \times  \mathcal{K}_0' \right) \subset \R^2 \times \R^2,\\
\tilde{\mathcal{K}}'  = & ( \tilde{\mathcal{K}}_0' \times ( \tilde{\mathcal{K}}_1'
\cup \tilde{\mathcal{K}}_2' ) ) \cup
(
( \tilde{\mathcal{K}}_1'\cup \tilde{\mathcal{K}}_2' ) \times  \tilde{\mathcal{K}}_0' ) \subset \R^3 \times \R^3,
\end{align*}
and those complementary sets as
\begin{align*}
(\mathcal{K})^c & =  (\R^2 \times \R^2) \setminus \mathcal{K} , \quad \
(\tilde{\mathcal{K}})^c =  (\R^3 \times \R^3) \setminus \tilde{\mathcal{K}}\\
(\mathcal{K}')^c & = (\R^2 \times \R^2) \setminus \mathcal{K}', \quad
(\tilde{\mathcal{K}}')^c = (\R^3 \times \R^3) \setminus \tilde{\mathcal{K}}'.
\end{align*}
Lastly, we define
\begin{equation*}
\widehat{Z}_{A} = \{ (k_1, k_2) \in \tilde{Z}_{A} \, | \,
\left( \mathcal{T}_{k_1}^A \times \mathcal{T}_{k_2}^A\right) \cap \left(
(\mathcal{K})^c \cap (\mathcal{K}')^c \right) \not=
\emptyset \},
\end{equation*}
and $\overline{Z}_{A}$ as the collection of $(k_1, k_2) \in \Z^2 \times \Z^2$ which satisfies
\begin{align*}
\mathcal{T}_{k_1}^{A} & \times \mathcal{T}_{k_2}^{A} \not\subset \bigcup_{2^{100} \leq A' \leq A}
\bigcup_{(k_1', k_2') \in \widehat{Z}_{A}} \bigl( \mathcal{T}_{k_1'}^{A'} \times \mathcal{T}_{k_2'}^{A'} \bigr),
\\
 \bigl( \tilde{\mathcal{T}}_{k_1}^{A} \times \tilde{\mathcal{T}}_{k_2}^{A} \bigr) & \cap \left(
G_{N_1, L_1} \times G_{N_2, L_2} \right) \cap \mathcal{A}  \cap \bigl(
(\tilde{\mathcal{K}})^c \cap (\tilde{\mathcal{K}}')^c \bigr) \not=
\emptyset.
\end{align*}
\end{defn}
\begin{lem}\label{lemma3.6}
For fixed $k_1 \in \Z^2$, the number of $k_2 \in \Z^2$ such that
$(k_1, k_2) \in \widehat{Z}_{A}$ is less than $2^{1000}$. Furthermore, the same claim holds true if we replace $\widehat{Z}_{A}$ by $\overline{Z}_{A}$.
\end{lem}
\begin{proof}
It is clear that we may assume that $A \geq 2^{300}$ and
\begin{equation}
\mathcal{T}_{k_1}^A \times \mathcal{T}_{k_2}^A \not\subset \left( \mathcal{K} \cup \mathcal{K}' \right).\label{ass01-lemma3.6}
\end{equation}
Define $\check{Z}_A = \check{Z}_A(k_1) \in \Z^2$ as the collection of $k_2 \in \Z^2$ which satisfies
\begin{equation*}
\mathcal{T}_{k_1}^A \times \mathcal{T}_{k_2}^A \subset \left( \mathcal{K} \cup \mathcal{K}' \right)
\cup \{ (\xi_1, \eta_1) \times (\xi_2, \eta_2) \, | \,
| \sin \angle \left( (\xi_1, \eta_1), (\xi_2, \eta_2) \right)| < 2^{-22}  \}.
\end{equation*}
For simplicity, we assume that $N_2 \leq N_1$.
The case $N_1 \leq N_2$ can be treated in a similar way.
We can find $k_1'=k_1'(k_1) \in \Z^2$ and $k_2'=k_2'(k_2) \in \Z^2$ such that
$ \mathcal{T}_{k_1}^A \subset  \mathcal{T}_{k_1'}^{A/2}$ and
$ \mathcal{T}_{k_2}^A \subset  \mathcal{T}_{k_2'}^{A/2}$, respectively. From the definition, $(k_1, k_2) \in
\widehat{Z}_A$ implies $(k_1', k_2') \notin {Z}_{A/2}'$ which means that there exist
$(\xi_1, \eta_1)$, $(\tilde{\xi}_1, \tilde{\eta}_1) \in \mathcal{T}_{k_1'}^{A/2}$,
$(\xi_2, \eta_2)$, $(\tilde{\xi}_2, \tilde{\eta}_2) \in \mathcal{T}_{k_2'}^{A/2}$ which satisfy
\begin{equation*}
|\Phi(\xi_1, \eta_1, \xi_2, \eta_2)| \leq 2 A^{-1} N_1^3  \quad \textnormal{and} \quad
 |F (\tilde{\xi}_1, \tilde{\eta}_1, \tilde{\xi}_2 ,\tilde{\eta}_2)| \leq 2 A^{-1} N_1^2.
\end{equation*}
Thus, letting $(\xi_1', \eta_1')$ be the center of $\mathcal{T}_{k_1}^A$, it suffices to show that there exist
$k_{2, {(\ell)}} \in \Z^2$ $(\ell= 1,2,3,4)$ such that
\begin{equation}
\left\{ (\xi_2, \eta_2)  \in \R^2 \setminus \bigcup_{k_2 \in \check{Z}_A}  \mathcal{T}_{k_2}^A \, \left| \,
\begin{aligned} & |\Phi(\xi_1', \eta_1', \xi_2, \eta_2) |\leq 2^4 A^{-1} N_1^3, \\
& |F (\xi_1', \eta_1', \xi_2 ,\eta_2)| \leq 2^4 A^{-1} N_1^2.
 \end{aligned} \right.
\right\}
\subset \bigcup_{\ell =1}^4 \mathcal{T}_{k_{2, {(\ell)}}}^{2^{-200}A}.\label{desired02-lemma3.6}
\end{equation}
In addition, it is clear that \eqref{desired02-lemma3.6} implies the almost one-to-one correspondence of
$(k_1, k_2) \in \overline{Z}_{A}$.
Recall that $|\Phi(\xi_1', \eta_1', \xi_2, \eta_2)| \leq 2^4 A^{-1} N_1^3$ and $ |F (\xi_1', \eta_1', \xi_2 ,\eta_2)| \leq 2^4 A^{-1} N_1^2$ mean
\begin{align}
|\Phi (\xi_1', \eta_1', \xi_2 ,\eta_2) |& = |\xi_1' \xi_2(\xi_1' + \xi_2) + \eta_1' \eta_2 (\eta_1' + \eta_2)|
\leq 2^4 A^{-1} N_1^3 ,\label{info01-lemma3.6} \\
|F (\xi_1', \eta_1', \xi_2 ,\eta_2)| & = |\xi_1' \eta_2 +  \xi_2 \eta_1' + 2 (\xi_1' \eta_1' + \xi_2 \eta_2)|
\leq 2^4 A^{-1} N_1^2,\label{info02-lemma3.6}
\end{align}
respectively. By the transformation $\xi_2' = \xi_2 + \xi_1'/2$, $\eta_2' = \eta_2 + \eta_1'/2$, we see that
\eqref{info01-lemma3.6} and \eqref{info02-lemma3.6} are equivalent to
\begin{align}
|\tilde{\Phi} (\xi_2' ,\eta_2')| & := \Bigl| \xi_1' {\xi_2'}^2 + \eta_1' {\eta_2'}^2 - \frac{{\xi_1'}^3 + {\eta_1'}^3}{4} \Bigr|
\leq 2^4 A^{-1} N_1^3 , \label{info03-lemma3.6} \\
|\tilde{F} (\xi_2' ,\eta_2')| & := \Bigl| \frac{3}{2} \, \xi_1' \, \eta_1' + 2 \,  \xi_2'  \, \eta_2' \Bigr|
\leq 2^4 A^{-1} N_1^2,\label{info04-lemma3.6}
\end{align}
respectively. Therefore we will show that there exist
$k_{2, {(\ell)}}' \in \Z^2$ such that the set of $(\xi_2', \eta_2')$ which satisfies
$(\xi_2'-\xi_1'/2, \eta_2'-\eta_1'/2) \notin \displaystyle{\bigcup_{k_2 \in \check{Z}_A}  \mathcal{T}_{k_2}^A}$,
\eqref{info03-lemma3.6}
and \eqref{info04-lemma3.6} are contained in ${\displaystyle \bigcup_{\ell =1}^4 \mathcal{T}_{k_{2, {(\ell)}}'}^{2^{-200}A}}$.

First we observe that
\begin{align}
\min (|\xi_1'|, \, |\eta_1'| ) \geq 2^{-55} N_1,\label{est01-lemma3.6}\\
\max (|\xi_2'|, \, |\eta_2'| ) \geq 2^{-30} N_1.\label{est02-lemma3.6}
\end{align}
\eqref{est01-lemma3.6} can be seen as follows. If $|\xi_1'| < 2^{-55} N_1$, we may assume that $|\eta_1'| \geq 2^{-2}N_1$. Thus, we have
\begin{align*}
\eqref{info03-lemma3.6} & \iff \Bigl| \xi_1' {\xi_2'}^2 + \eta_1' {\eta_2'}^2 - \frac{{\xi_1'}^3 + {\eta_1'}^3}{4} \Bigr|
\leq 2^4 A^{-1} N_1^3\\
& \ \Longrightarrow |\eta_1'| \, \Bigl| {\eta_2'}^2 - \frac{{\eta_1'}^2}{4} \Bigr| \leq 2^{-50} N_1^3\\
& \ \Longrightarrow \Bigl| |\eta_2'| - \frac{|\eta_1'|}{2} \Bigr| \leq 2^{-45}N_1\\
& \ \Longrightarrow |\eta_2'| \geq 2^{-4}N_1.
\end{align*}
This gives $|\xi_2'| \geq 2^{-50}N_1$. Indeed, since $\xi_2' = \xi_2 + \xi_1'/2$, we only need to show
$|\xi_2| > 2^{-49} N_1$.
Since $A \geq 2^{300}$, it is observed that $(\xi_2, \eta_2)  \notin \bigcup_{k_2 \in \check{Z}_A}  \mathcal{T}_{k_2}^A$ provides
\begin{equation}\label{condition:anglelemma3.7}
\frac{|\xi_1' \eta_2 - \xi_2 \eta_1'|}{|(\xi_1',\eta_1')| \, |(\xi_2,\eta_2)|} > 2^{-23}.
\end{equation}
It follows from $N_1 \leq 2^{22} N_2$ that
\begin{equation*}
\frac{|\xi_1' \eta_2 - \xi_2 \eta_1'|}{|(\xi_1',\eta_1')| \, |(\xi_2,\eta_2)|} > 2^{-23}
\Longrightarrow \frac{|\xi_2 \eta_1'| + | \xi_1' \eta_2|}{|(\xi_1',\eta_1')|} > 2^{-47} N_1 \Longrightarrow
 |\xi_2| > 2^{-49}N_1.
\end{equation*}
Thus, it holds that $|\xi_2| >2^{-49} N_1$ which implies $|\xi_2'| \geq 2^{-50}N_1$.
While, it holds that
\begin{align*}
|\tilde{F} (\xi_2' ,\eta_2')| = & \Bigl| \frac{3}{2} \, \xi_1' \, \eta_1' + 2 \,  \xi_2'  \, \eta_2' \Bigr|\\
\geq  & 2|  \xi_2'  \, \eta_2'| - \frac{3}{2} |\xi_1' \, \eta_1'|\\
\geq & 2^{-54} N_1^2,
\end{align*}
which contradicts \eqref{info04-lemma3.6}. Similarly, if $\max (|\xi_2'|, \, |\eta_2'| ) < 2^{-30} N_1$, since
$\xi_2 = \xi_2' - \xi_1'/2$ and $\eta_2 = \eta_2' - \eta_1'/2$, we see
\begin{equation*}
\frac{|\xi_1' \eta_2 - \xi_2 \eta_1'|}{|(\xi_1', \eta_1')| \, |(\xi_2, \eta_2)|} < 2^{-23},
\end{equation*}
which contradicts \eqref{condition:anglelemma3.7}. Thus \eqref{est02-lemma3.6} holds. Without loss of generality, we can assume $2^{-30}N_1 \leq |\xi_2'|$. It follows from \eqref{info04-lemma3.6} that
\begin{align}
& \Bigl| \frac{3}{2} \, \xi_1' \, \eta_1' + 2 \,  \xi_2'  \, \eta_2' \Bigr|
 \leq 2^4 A^{-1} N_1^2 \notag\\
\Longrightarrow & \Bigl| \eta_2' + \frac{3  \xi_1' \, \eta_1'}{4 \xi_2'} \Bigr|  \leq \frac{2^3 A^{-1} N_1^2}{|\xi_2'|}
\leq 2^{33} A^{-1} N_1.\label{est003-lemma3.6}
\end{align}
\eqref{info03-lemma3.6} and \eqref{est003-lemma3.6} yield
\begin{align}
& \Bigl| \xi_1' {\xi_2'}^2 + \eta_1' {\eta_2'}^2 - \frac{{\xi_1'}^3 + {\eta_1'}^3}{4} \Bigr|
\leq 2^4 A^{-1} N_1^3 \notag\\
 \xLongrightarrow[ \ ]{\eqref{est003-lemma3.6}} & \Bigl| \xi_1' {\xi_2'}^2 + \eta_1' \frac{9 {\xi_1'}^2 {\eta_1'}^2}{16 {\xi_2'}^2} - \frac{{\xi_1'}^3 + {\eta_1'}^3}{4} \Bigr|
\leq 2^{70} A^{-1} N_1^3.\label{est004-lemma3.6}
\end{align}
Define
\begin{equation*}
G(\xi_2') := \xi_1' {\xi_2'}^2 +  \frac{9 {\xi_1'}^2 {\eta_1'}^3}{16 {\xi_2'}^2} -
\frac{{\xi_1'}^3 + {\eta_1'}^3}{4}.
\end{equation*}
Clearly, if we show
\begin{equation}
\Bigl| \Bigl( \frac{dG}{d\xi_2'} \Bigr)
 (\xi_2') \Bigr| = \Bigl| \frac{2 \xi_1'}{{\xi_2'}^3} \Bigl( {\xi_2'}^4 -  \frac{9 {\xi_1'} {\eta_1'}^3}{16}
\Bigr) \Bigr| \geq 2^{-100}N_1^2,\label{est005-lemma3.6}
\end{equation}
for any $\xi_2'$ which satisfies \eqref{est004-lemma3.6}, then there exist at most four constants
$c_{(\ell)} (\xi_1', \eta_1')$ $(\ell=1,2,3,4)$ such that
\begin{equation*}
|\xi_2' - c_{(\ell)} (\xi_1', \eta_1')| \leq 2^{170} A^{-1} N_1 \quad \textnormal{for any $\xi_2'$ which satisfies \eqref{est004-lemma3.6}},
\end{equation*}
which completes the proof since \eqref{est003-lemma3.6} gives a similar restriction of $\eta_2'$. Thus, it suffices to show
\eqref{est005-lemma3.6} for any $\xi_2'$ which satisfies \eqref{est003-lemma3.6}, \eqref{est004-lemma3.6} and
$(\xi_2'-\xi_1'/2, \eta_2'-\eta_1'/2) \notin
\displaystyle{\bigcup_{k_2 \in \check{Z}_A}  \mathcal{T}_{k_2}^A}$.
For the sake of contradiction, we assume that $\xi_2'$ satisfies
\begin{equation}
| G(\xi_2') | \leq 2^{70} A^{-1} N_1^3 \quad \textnormal{and} \quad | G'(\xi_2')| =
\Bigl| \Bigl( \frac{dG}{d\xi_2'} \Bigr)
 (\xi_2') \Bigr| \leq 2^{-100}N_1^2.\label{est007-lemma3.6}
\end{equation}
Obviously, $\xi_1' \, \eta_1' \leq 0$ implies $ | G'(\xi_2')| \geq 2|\xi_1' \xi_2'| \geq 2^{-75} N_1^2$.
Thus, we may assume $\xi_1' \, \eta_1' >0$. For simplicity, we assume $\xi_1' >0$ and $\eta_1' >0$.
The case $\xi_1' <0$ and $\eta_1'<0$ can be treated similarly. We calculate that
\begin{align}
&  | G'(\xi_2')|\leq 2^{-100}N_1^2 \notag \\
\iff & \left| \frac{2 \xi_1'}{{\xi_2'}^3} \left( {\xi_2'}^2 +  \frac{3}{4 } \sqrt{\xi_1' \, {\eta_1'}^3}
\right) \left( {\xi_2'}^2 -  \frac{3}{4 } \sqrt{\xi_1' \, {\eta_1'}^3}
\right)\right|\leq 2^{-100}N_1^2 \notag \\
\Longrightarrow \ & \left|  {\xi_2'}^2 -  \frac{3}{4 } \sqrt{\xi_1' \, {\eta_1'}^3}  \right| \leq 2^{-65}N_1^2.
\label{est006-lemma3.6}
\end{align}
We deduce from $| G(\xi_2') | \leq 2^{70} A^{-1} N_1^3$ and \eqref{est006-lemma3.6} that
\begin{align*}
& |G  (\xi_2')|\leq 2^{70} A^{-1} N_1^3 \\
 \xLongrightarrow[ \ ]{\eqref{est006-lemma3.6}} &
\left| \frac{3}{2 } \sqrt{{\xi_1'}^3 \, {\eta_1'}^3} - \frac{{\xi_1'}^3 + {\eta_1'}^3}{4} \right|
\leq 2^{-55}  N_1^3\\
 \Longrightarrow & \bigl| ( {\xi_1'}^{\frac{3}{2}} )^2 - 6  {\xi_1'}^{\frac{3}{2}}  {\eta_1'}^{\frac{3}{2}} +
({\eta_1'}^{\frac{3}{2}} )^2 \bigr| \leq 2^{-50}  N_1^3\\
 \Longrightarrow & \bigl| \bigl( {\xi_1'}^{\frac{3}{2}}  - ( \sqrt{2} + 1 )^2 {\eta_1'}^{\frac{3}{2}}\bigr)  \bigl( {\xi_1'}^{\frac{3}{2}}  - ( \sqrt{2}  - 1 )^2 {\eta_1'}^{\frac{3}{2}}\bigr)  \bigr|
\leq 2^{-50}  N_1^3.
\end{align*}
Last inequality implies
\begin{equation*}
\bigl| \eta_1'  - ( \sqrt{2} - 1 )^{\frac{4}{3}}\xi_1' \bigr| \leq 2^{-45} N_1 \quad \textnormal{or}
\quad
\bigl| \eta_1'  - ( \sqrt{2} + 1 )^{\frac{4}{3}}\xi_1' \bigr| \leq 2^{-45} N_1.
\end{equation*}
\underline{Case $|   \eta_1'  - ( \sqrt{2} - 1 )^{\frac{4}{3}}\xi_1'  | \leq 2^{-45}N_1$}

First we treat the case $\bigl| \eta_1'  - ( \sqrt{2} - 1 )^{\frac{4}{3}}\xi_1' \bigr| \leq 2^{-45} N_1$.
It follows from \eqref{est006-lemma3.6} that
\begin{equation*}
 \Bigl|  {\xi_2'}^2 -  \frac{3}{4 } ( \sqrt{2} - 1 )^2 {\xi_1'}^2  \Bigr| \leq 2^{-40}N_1^2
\end{equation*}
which gives
\begin{equation*}
 \Bigl|  {\xi_2'} -  \frac{\sqrt{3}}{2 } ( \sqrt{2} - 1 ) {\xi_1'}  \Bigr| \leq 2^{-35}N_1 \quad
\textnormal{or} \quad
\Bigl|  {\xi_2'} +  \frac{\sqrt{3}}{2 }  ( \sqrt{2} - 1 ) {\xi_1'}  \Bigr| \leq 2^{-35}N_1.
\end{equation*}
If we assume $ \bigl|  {\xi_2'} -  \frac{\sqrt{3}}{2 } ( \sqrt{2} - 1 ) {\xi_1'}  \bigr| \leq 2^{-35}N_1$,
it follows from \textnormal{\eqref{est003-lemma3.6}} that
\begin{equation*}
\Bigl| \eta_2' + \frac{3  \xi_1' \, \eta_1'}{4 \xi_2'} \Bigr|
\leq 2^{16} A^{-1} N_1
\Longrightarrow  \Bigl| \eta_2' + \frac{\sqrt{3}}{2} ( \sqrt{2} -1 )^{\frac{1}{3}} \xi_1' \Bigr|
\leq 2^{-30} N_1.
\end{equation*}
Since $\xi_2 = \xi_2' - \xi_1'/2$ and $\eta_2 = \eta_2' - \eta_1'/2$, we observe
\begin{align*}
& \begin{cases}
\bigl|  {\xi_2'} -  \frac{\sqrt{3}}{2 } \left( \sqrt{2} - 1\right) {\xi_1'}  \bigr| \leq 2^{-35}N_1,\\
\bigl| \eta_2' + \frac{\sqrt{3}}{2} \left( \sqrt{2} -1 \right)^{\frac{1}{3}} \xi_1' \bigr|
\leq 2^{-30} N_1.
\end{cases}\\
\Longrightarrow &
\begin{cases}
\bigl|  {\xi_2} +  \frac{1-\sqrt{3}(\sqrt{2} - 1) }{2 }  {\xi_1'}  \bigr| \leq 2^{-30}N_1,\\
\bigl| \eta_2 + \frac{\sqrt{3}+\sqrt{2} -1}{2} ( \sqrt{2} -1  )^{\frac{1}{3}} \xi_1' \bigr|
\leq 2^{-25} N_1.
\end{cases}
\end{align*}
This implies
\begin{equation*}
\bigl| \eta_2 - ( \sqrt{2}+ 1 )^{\frac{2}{3}} (\sqrt{2} + \sqrt{3} ) \xi_2 \bigr| \leq
2^{-23} N_1.
\end{equation*}
Similarly, if we assume $ \bigl|  {\xi_2'} +  \frac{\sqrt{3}}{2 } ( \sqrt{2} - 1 ) {\xi_1'}  \bigr| \leq 2^{-35}N_1$,
by the almost same calculation as above, we get
\begin{equation*}
\bigl| \eta_2 + ( \sqrt{2}+ 1 )^{\frac{2}{3}} (\sqrt{3} - \sqrt{2} ) \xi_2 \bigr| \leq
2^{-23} N_1.
\end{equation*}
\underline{Case $|   \eta_1'  - ( \sqrt{2} + 1 )^{\frac{4}{3}}\xi_1'  | \leq 2^{-45}N_1$}

Next we consider the case $\bigl| \eta_1'  - ( \sqrt{2} + 1 )^{\frac{4}{3}}\xi_1' \bigr| \leq 2^{-45} N_1$.
Following the same argument as that for the former case,
we can prove that $(\xi_2, \eta_2)$ satisfies one of the followings.
\begin{align*}
& \bigl| \eta_2 + ( \sqrt{2}- 1 )^{\frac{2}{3}} (\sqrt{2} + \sqrt{3} ) \xi_2 \bigr| \leq
2^{-23} N_1,\\
& \bigl| \eta_2 - ( \sqrt{2}- 1 )^{\frac{2}{3}} (\sqrt{3} - \sqrt{2} ) \xi_2 \bigr| \leq
2^{-23} N_1.
\end{align*}
To summarize the above, we conclude that if $(\xi_2', \eta_2')$ satisfies \eqref{est003-lemma3.6} and
\eqref{est007-lemma3.6} then for any $k_1$, $k_2 \in \Z^2$ such that $(\xi_1', \eta_1') \in \mathcal{T}_{k_1}^A$, $(\xi_2, \eta_2) \in
\mathcal{T}_{k_2}^A$, we get
\begin{equation*}
\mathcal{T}_{k_1}^A \times \mathcal{T}_{k_2}^A \subset \left( \mathcal{K} \cup \mathcal{K}' \right),
\end{equation*}
which contradicts the assumption \eqref{ass01-lemma3.6}.
\end{proof}
\begin{rem}\label{pairs-of-tiles}
We here sketch why we needed to exclude $\mathcal{K} \cup \mathcal{K}'$ in the proof of the almost orthogonality of $k_1$ and $k_2$
such that
$(k_1, k_2) \in \tilde{Z}_{A}$.
In the proof of Lemma \ref{lemma3.6}, we saw that the almost orthogonality is equivalent to the existence of
$k_{2, {(\ell)}}' \in \Z^2$ such that the set of $(\xi_2', \eta_2')$ which satisfies
\begin{align}
|\tilde{\Phi} (\xi_2' ,\eta_2')| & := \Bigl| \xi_1' {\xi_2'}^2 + \eta_1' {\eta_2'}^2 - \frac{{\xi_1'}^3 + {\eta_1'}^3}{4} \Bigr|
\leq 2^4 A^{-1} N_1^3 ,\label{rem3.4-01} \\
|\tilde{F} (\xi_2' ,\eta_2')| & := \Bigl| \frac{3}{2} \, \xi_1' \, \eta_1' + 2 \,  \xi_2'  \, \eta_2' \Bigr|
\leq 2^4 A^{-1} N_1^2,\label{rem3.4-02}
\end{align}
are contained in ${\displaystyle \bigcup_{\ell =1}^4 \mathcal{T}_{k_{2, {(\ell)}}'}^{2^{-200}A}}$.
We consider the conditions \eqref{rem3.4-01} and \eqref{rem3.4-02} by seeing the figures below.
Figure \ref{fig:modulation} describes the condition \eqref{rem3.4-01} for $\xi_1' \eta_1'>0$ and $\xi_1' \eta_1' <0$. $(\xi_2',\eta_2')$ which satisfies the condition \eqref{rem3.4-01} is confined to the gray areas whose width is $\sim A^{-1} N_1$.
\begin{figure}[H]
\caption{Condition \eqref{rem3.4-01} with $\xi_1' \, \eta_1'>0$ (left) and $\xi_1' \, \eta_1'<0$ (right).}
\label{fig:modulation}
\centering
    \begin{tikzpicture}
\begin{scope}[xshift=-3.75cm]
  \useasboundingbox (-3.5,-3.7) rectangle (3.5,3.5);
   \draw [thick, -stealth](-3.2,0)--(3.2,0) node [anchor=west, font=\scriptsize]{$\xi_2'$};
   \draw [thick, -stealth](0,-3.2)--(0,3.2) node [anchor=west, font=\scriptsize]{};
\node [anchor=west, font=\scriptsize] at(0,3.2){$\eta_2'$};
\path[draw,thick, fill=gray, fill opacity=0.5] plot[domain=0:6.28,variable=\t, samples=50] ({1.1*cos(\t r)},{1.1*2.5*sin(\t r)});
\path[draw,thick, fill=white] plot[domain=0:6.28,variable=\t, samples=70] ({0.9*cos(\t r)},{2.5*sin(\t r)});
   \draw [thick](-0.9,0)--(0.9,0) node [anchor=west, font=\scriptsize]{};
   \draw [thick](0,-2.5)--(0,2.5) node [anchor=west, font=\scriptsize]{};
   \draw [thin](0.748,1.27)--(0.95,1.33) node [anchor=west, font=\scriptsize]{};
\draw (0.849,1.3) to [out=70,in=195] (2.03,2.5);
\node [anchor=west] at(2,2.55){$\sim A^{-1}N_1$};
\end{scope}
\begin{scope}[xshift=3.75cm]
   \draw [thick, -stealth](-3.2,0)--(3.2,0) node [anchor=west, font=\scriptsize]{$\xi_2'$};
   \draw [thick, -stealth](0,-3.2)--(0,3.2) node [anchor=west, font=\scriptsize]{};
\node [anchor=west, font=\scriptsize] at(0,3.2){$\eta_2'$};
\path[draw,thick, fill=gray, fill opacity=0.5] plot[domain=-1.35:1.35,variable=\t, samples=70] ({0.7/cos(\t r)},{0.7*tan(\t r)});
\path[draw,thick, fill=gray, fill opacity=0.5] plot[domain=3.1415-1.35:3.1415+1.35,variable=\t, samples=50] ({0.7/cos(\t r)},{0.7*tan(\t r)});
\path[draw,thick, fill=white] plot[domain=-1.33:1.33,variable=\t, samples=70, xshift=8pt] ({0.7/cos(\t r)},{0.7*tan(\t r)});
\path[draw,thick, fill=white] plot[domain=3.1415-1.33:3.1415+1.33,variable=\t, samples=50, xshift=-8pt] ({0.7/cos(\t r)},{0.7*tan(\t r)});
   \draw [thick](0.99,0)--(3.2,0) node [anchor=west, font=\scriptsize]{};
   \draw [thick](-3.2,0)--(-0.99,0) node [anchor=west, font=\scriptsize]{};
   \draw [thin](-2.4,2)--(-2.25,2.14) node [anchor=west, font=\scriptsize]{};
\draw (-2.325,2.07) to [out=140,in=355] (-3.9,2.6);
\end{scope}
    \end{tikzpicture}
\end{figure}
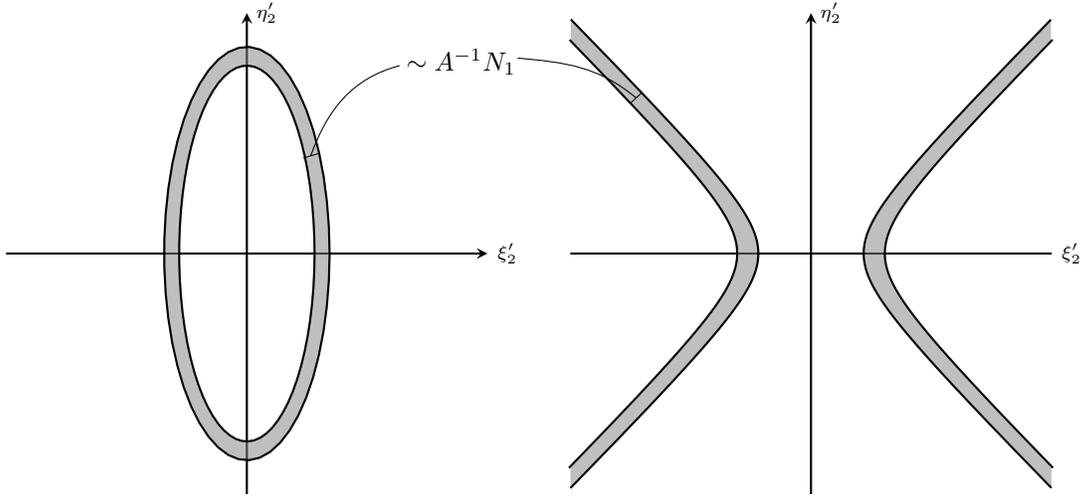
Figures \ref{fig:transversality} describes the condition \eqref{rem3.4-02} with $\xi_1' \eta_1'>0$ and $\xi_1' \eta_1' <0$, respectively. $(\xi_2',\eta_2')$ which satisfies the condition \eqref{rem3.4-02} is confined to the gray areas whose width is $\sim A^{-1} N_1$.
\begin{figure}[H]
\caption{Condition \eqref{rem3.4-02} with $\xi_1' \, \eta_1'>0$ (left) and $\xi_1' \, \eta_1'<0$ (right).}
\label{fig:transversality}
\centering
    \begin{tikzpicture}
\begin{scope}[xshift=-3.75cm]
  \useasboundingbox (-3.5,-3.7) rectangle (3.5,3.5);
   \draw [thick, -stealth](-3.2,0)--(3.2,0) node [anchor=west, font=\scriptsize]{$\xi_2'$};
   \draw [thick, -stealth](0,-3.2)--(0,3.2) node [anchor=west, font=\scriptsize]{};
\node [anchor=west, font=\scriptsize] at(0,3.2){$\eta_2'$};
\path[draw,thick, fill=gray, fill opacity=0.5] plot[domain=0.223:3.2, samples=50] ({\x},{-0.7/\x});
\path[fill=gray, fill opacity=0.5]
plot (0.223, -0.7/0.223)--(3.2,-0.7/3.2)--(3.2,-0.7/3.2-0.15)--(0.223+0.15, -0.7/0.223);
\path[draw, thick, fill=white] plot[domain=0.2355:3.0245, samples=50, xshift=5, yshift=-5] ({\x},{-0.7/\x});
\path[draw,thick, fill=gray, fill opacity=0.5] plot[domain=-3.2:-0.223, samples=50] ({\x},{-0.7/\x});
\path[fill=gray, fill opacity=0.5]
plot (-0.223, 0.7/0.223)--(-3.2,0.7/3.2)--(-3.2,0.7/3.2+0.15)--(-0.223-0.15, 0.7/0.223);
\path[draw, thick, fill=white] plot[domain=-3.0245:-0.2355, samples=50, xshift=-5, yshift=5] ({\x},{-0.7/\x});
\end{scope}
\begin{scope}[xshift=3.75cm]
  \useasboundingbox (-3.5,-3.7) rectangle (3.5,3.5);
   \draw [thick, -stealth](-3.2,0)--(3.2,0) node [anchor=west, font=\scriptsize]{$\xi_2'$};
   \draw [thick, -stealth](0,-3.2)--(0,3.2) node [anchor=west, font=\scriptsize]{};
\node [anchor=west, font=\scriptsize] at(-0.55,3.2){$\eta_2'$};
\path[draw,thick, fill=gray, fill opacity=0.5] plot[domain=0.223:3.2, samples=50] ({\x},{0.7/\x});
\path[fill=gray, fill opacity=0.5]
plot (0.223, 0.7/0.223)--(3.2,0.7/3.2)--(3.2,0.7/3.2+0.15)--(0.223+0.15, 0.7/0.223);
\path[draw, thick, fill=white] plot[domain=0.2355:3.0245, samples=50, xshift=5, yshift=5] ({\x},{0.7/\x});
\path[draw,thick, fill=gray, fill opacity=0.5] plot[domain=-3.2:-0.223, samples=50] ({\x},{0.7/\x});
\path[fill=gray, fill opacity=0.5]
plot (-0.223, -0.7/0.223)--(-3.2,-0.7/3.2)--(-3.2,-0.7/3.2-0.15)--(-0.223-0.15, -0.7/0.223);
\path[draw, thick, fill=white] plot[domain=-3.0245:-0.2355, samples=50, xshift=-5, yshift=-5] ({\x},{0.7/\x});
\end{scope}
    \end{tikzpicture}
\end{figure}
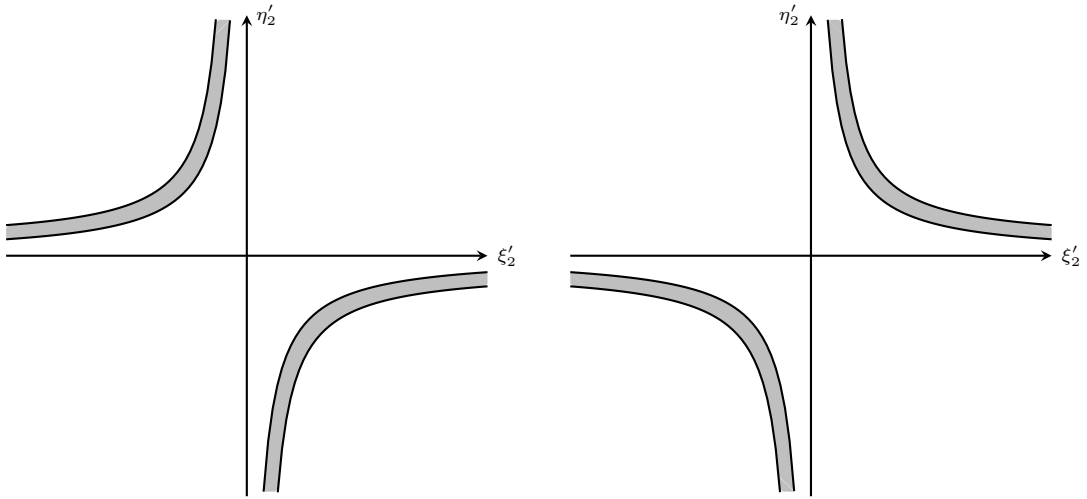
By seeing the above figures, for example, we can find that the almost orthogonality does not hold when there exists a point $a \in \R^2$ such that $\tilde{\Phi} (a)=\tilde{F}(a) = 0$ and the tangent line of  $\tilde{\Phi} (\xi_2' ,\eta_2') = 0$ at $a$ corresponds to that of $\tilde{F} (\xi_2' ,\eta_2')=0$ at $a$.
See Figure \ref{fig:both} below. To remove these interactions, we excluded $\mathcal{K} \cup \mathcal{K}'$ from
$\R^2 \times \R^2$.
\begin{figure}[H]
\caption{Failure of the almost orthogonality.}
\label{fig:both}
\centering
    \begin{tikzpicture}
  \useasboundingbox (-3.5,-3.5) rectangle (10.5,3.4);
   \draw [thick, -stealth](-3.2,0)--(3.2,0) node [anchor=west, font=\scriptsize, xshift=-8pt, yshift=10pt]{$\xi_2'$};
   \draw [thick, -stealth](0,-3.2)--(0,3.2) node [anchor=west, font=\scriptsize]{};
\node [anchor=west, font=\scriptsize] at(0,3.2){$\eta_2'$};
\path[draw,thick] plot[domain=0.223:3.2, samples=50] ({\x},{-0.7/\x});
\path[draw, thick] plot[domain=0.2355:3.0245, samples=50, xshift=5, yshift=-5] ({\x},{-0.7/\x});
\path[draw,thick] plot[domain=-3.2:-0.223, samples=50] ({\x},{-0.7/\x});
\path[draw, thick] plot[domain=-3.0245:-0.2355, samples=50, xshift=-5, yshift=5] ({\x},{-0.7/\x});
\path[draw,thick] plot[domain=0:6.28,variable=\t, samples=50] ({0.88*cos(\t r)},{2.35*sin(\t r)});
\path[draw,thick] plot[domain=0:6.28,variable=\t, samples=70] ({0.69*cos(\t r)},{0.83*2.5*sin(\t r)});
\path[draw,thick,fill=gray]  plot[domain=0.32:0.82, samples=50] ({\x},{-0.7/\x});
\path[draw,thick,fill=gray]
plot[domain=5.08:5.92,variable=\t, samples=50] ({0.88*cos(\t r)},{2.35*sin(\t r)});
\path[draw,thick,fill=gray]  plot[domain=-0.82:-0.32, samples=50] ({\x},{-0.7/\x});
\path[draw,thick,fill=gray]
plot[domain=5.08-3.1415:5.92-3.1415,variable=\t, samples=50] ({0.88*cos(\t r)},{2.35*sin(\t r)});
\path[draw,thick,fill=gray]  plot[domain=0.32:0.82, samples=50, xshift=150, yshift=120] ({3*\x},{3*-0.7/\x});
\path[draw,thick,fill=gray]
plot[domain=5.08:5.92,variable=\t, samples=50, xshift=150, yshift=120] ({3*0.88*cos(\t r)},{3*2.35*sin(\t r)});
   \draw [thick, dashed](0.82,-0.7/0.82)--(7.75,1.73) node [anchor=west, font=\scriptsize]{};
   \draw [thick, dashed](0.32,-0.7/0.32)--(6.28,-2.34) node [anchor=west, font=\scriptsize]{};
   \draw [thin](7.75,1.73)--(6.28,-2.34) node [anchor=west, font=\scriptsize]{};
\draw (7.75/2+6.28/2,1.73/2-2.34/2) to [out=350,in=195] (9,1)
node [anchor=west]{$\sim A^{-\frac{1}{2}}N_1$};
    \end{tikzpicture}
\end{figure}
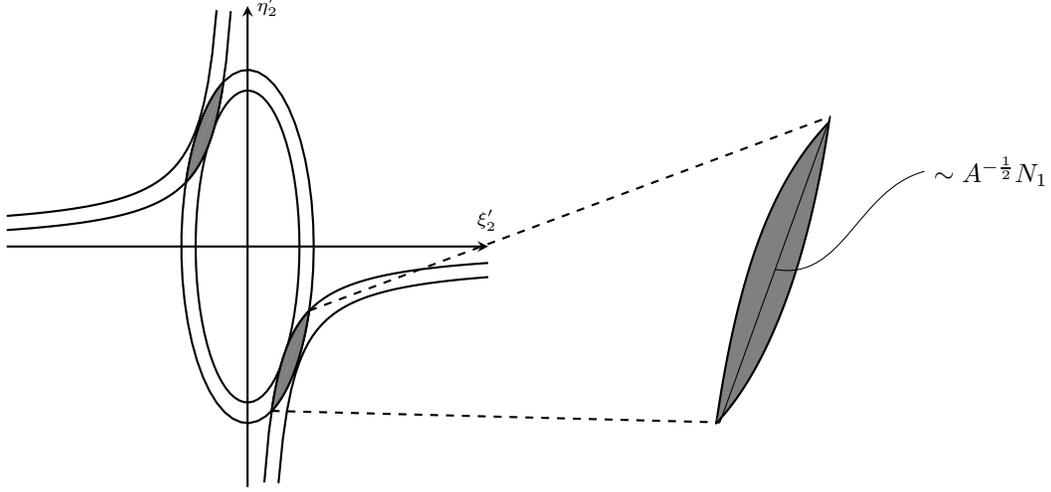
\end{rem}
By using Propositions \ref{prop3.3}, \ref{prop3.4} and Lemma \ref{lemma3.6}, we get the desired estimate \eqref{desired-est-11-17} under the assumption
$(\xi_1, \eta_1) \times (\xi_2, \eta_2) \in (\mathcal{K})^c \cap (\mathcal{K}')^c$.
\begin{prop}\label{prop3.7}
Assume \textnormal{(i)-(iii)} in \textit{Case} \textnormal{\ref{ass-1}}. Then we have
\begin{align}
&
 \Bigl|\int_{*}{  \ha{w}_{{N_0, L_0}}(\tau, \xi, \eta) \chi_{\{(\mathcal{K})^c \cap (\mathcal{K}')^c\}}((\xi_1, \eta_1),
(\xi_2, \eta_2))
\ha{u}_{N_1, L_1}(\tau_1, \xi_1, \eta_1)  \ha{v}_{N_2, L_2}(\tau_2, \xi_2, \eta_2)
}
d\sigma_1 d\sigma_2 \Bigr| \notag\\
& \qquad \qquad \qquad  \lesssim   (N_1)^{-\frac{5}{4}}  (L_0 L_1 L_2)^{\frac{5}{12}} \|\ha{u}_{N_1, L_1} \|_{L^2} \| \ha{v}_{N_2, L_2} \|_{L^2}
\|\ha{w}_{{N_0, L_0}} \|_{L^2},\label{desired-est-prop3.8}
\end{align}
where $d \sigma_j = d\tau_j d \xi_j d \eta_j$ and $*$ denotes $(\tau, \xi, \eta) = (\tau_1 + \tau_2, \xi_1+ \xi_2, \eta_1 + \eta_2).$
\end{prop}
\begin{proof}
Clearly, it suffices to show the estimate \eqref{desired-est-prop3.8} for non-negative
$\ha{w}_{{N_0, L_0}}$, $\ha{u}_{N_1, L_1}$,
$\ha{v}_{N_2, L_2} $. Thus, let $\ha{w}_{{N_0, L_0}}$, $\ha{u}_{N_1, L_1}$,
$\ha{v}_{N_2, L_2} $ be non-negative.
Let $A_0 \geq 2^{100}$ be dyadic.
By the definitions of $\widehat{Z}_A$ and $ \overline{Z}_{A_0}$, we see that
$\left(
G_{N_1, L_1} \times G_{N_2, L_2} \right) \cap \mathcal{A}  \cap
(\tilde{\mathcal{K}})^c \cap (\tilde{\mathcal{K}}')^c $ are contained in
\begin{equation*}
\bigcup_{2^{100} \leq A \leq A_0}
\bigcup_{(k_1, k_2) \in \widehat{Z}_A}
 \bigl( \tilde{\mathcal{T}}_{k_1}^{A} \times \tilde{\mathcal{T}}_{k_2}^{A} \bigr)
 \cup \bigcup_{(k_1, k_2) \in \overline{Z}_{A_0}}
 \bigl( \tilde{\mathcal{T}}_{k_1}^{A_0} \times \tilde{\mathcal{T}}_{k_2}^{A_0} \bigr).
\end{equation*}
Thus, we calculate that
\begin{align*}
& \textnormal{(LHS) of \eqref{desired-est-prop3.8}} \\
\leq & \sum_{2^{100} \leq A \leq A_0} \sum_{(k_1, k_2) \in \widehat{Z}_A}
\Bigl|\int_{*}{  \ha{w}_{{N_0, L_0}}(\tau, \xi, \eta)
\ha{u}_{N_1, L_1}|_{\tilde{\mathcal{T}}_{k_1}^A}(\tau_1, \xi_1, \eta_1)  \ha{v}_{N_2, L_2}|_{\tilde{\mathcal{T}}_{k_2}^A}(\tau_2, \xi_2, \eta_2)
}
d\sigma_1 d\sigma_2 \Bigr|\\
& \quad + \sum_{(k_1, k_2) \in \overline{Z}_{A_0}}
\Bigl|\int_{*}{  \ha{w}_{{N_0, L_0}}(\tau, \xi, \eta)
\ha{u}_{N_1, L_1}|_{\tilde{\mathcal{T}}_{k_1}^{A_0}}(\tau_1, \xi_1, \eta_1)  \ha{v}_{N_2, L_2}|_{\tilde{\mathcal{T}}_{k_2}^{A_0}}(\tau_2, \xi_2, \eta_2)
}
d\sigma_1 d\sigma_2 \Bigr|\\
& =: \sum_{2^{100} \leq A \leq A_0} \sum_{(k_1, k_2) \in \widehat{Z}_A}  I_1 + \sum_{(k_1, k_2) \in \overline{Z}_{A_0}}  I_2.
\end{align*}
We deduce from Propositions \ref{prop3.3}, \ref{prop3.4} and Lemma \ref{lemma3.6} that
\begin{align*}
& \sum_{(k_1, k_2) \in \widehat{Z}_A}  I_1 \\ \lesssim &
\sum_{(k_1, k_2) \in \widehat{Z}_A}
A^{\frac{1}{2}} N_1^{-2} (L_0 L_1 L_2)^{\frac{1}{2}} \|\ha{u}_{N_1, L_1}|_{\tilde{\mathcal{T}}_{k_1}^A} \|_{L^2}
\| \ha{v}_{N_2, L_2}|_{\tilde{\mathcal{T}}_{k_2}^A} \|_{L^2}
\|\ha{w}_{{N_0, L_0}} \|_{L^2}\\
\lesssim & \
A^{\frac{1}{2}} N_1^{-2} (L_0 L_1 L_2)^{\frac{1}{2}} \|\ha{u}_{N_1, L_1} \|_{L^2}
\| \ha{v}_{N_2, L_2} \|_{L^2}
\|\ha{w}_{{N_0, L_0}} \|_{L^2}.
\end{align*}
Let $A_0$ be the minimal dyadic number which satisfies
$A_0 \geq (L_{012}^{\max})^{-1/2} N_1^{3/2}$.
Then, from the above, we obtain
\begin{equation*}
\sum_{2^{100} \leq A \leq A_0} \sum_{(k_1, k_2) \in \widehat{Z}_A}   I_1
 \lesssim
(N_1)^{-\frac{5}{4}}  (L_0 L_1 L_2)^{\frac{5}{12}} \|\ha{u}_{N_1, L_1} \|_{L^2}
\| \ha{v}_{N_2, L_2} \|_{L^2}
\|\ha{w}_{{N_0, L_0}} \|_{L^2}.
\end{equation*}
Next we handle the latter term. For simplicity, we assume $L_0 =L_{012}^{\max} $.
The other cases $L_1 =L_{012}^{\max} $ and $L_2 =L_{012}^{\max} $ can be treated similarly.
It follows from \eqref{bilinearStrichartz-1} in Proposition \ref{prop3.2} and Lemma \ref{lemma3.6} that
\begin{align*}
\sum_{(k_1, k_2) \in \overline{Z}_{A_0}}  I_2 & \lesssim
 (N_1 A_0)^{-\frac{1}{2}} (L_1 L_2)^{\frac{1}{2}} \sum_{(k_1, k_2) \in \overline{Z}_{A_0}} \|\ha{u}_{N_1, L_1}|_{\tilde{\mathcal{T}}_{k_1}^{A_0}} \|_{L^2}
\| \ha{v}_{N_2, L_2}|_{\tilde{\mathcal{T}}_{k_2}^{A_0}} \|_{L^2} \|\ha{w}_{{N_0, L_0}} \|_{L^2}\\
& \lesssim (N_1)^{-\frac{5}{4}}  (L_0 L_1 L_2)^{\frac{5}{12}} \|\ha{u}_{N_1, L_1} \|_{L^2}
\| \ha{v}_{N_2, L_2} \|_{L^2}
\|\ha{w}_{{N_0, L_0}} \|_{L^2}.
\end{align*}
This completes the proof.
\end{proof}
It remains to prove the estimate \eqref{desired-est-11-17} for $(\xi_1 , \eta_1) \times (\xi_2, \eta_2) \in \left( \mathcal{K} \cup \mathcal{K}' \right)$. In this case, as we saw in the proof of Lemma \ref{lemma3.6} and \textit{Remark} \ref{pairs-of-tiles}, the almost one-to-one correspondence of
$(k_1, k_2) \in \tilde{Z}_{A}$ does not hold. Therefore we introduce another decomposition.
We note that, by symmetry of $\xi_i$ and $\eta_i$ with $i=1,2$, once the estimate
\eqref{desired-est-11-17} is verified for the case
$ (\xi_1 , \eta_1) \times (\xi_2, \eta_2) \in \mathcal{K}$, one can obtain the same estimate for $(\xi_1 , \eta_1) \times (\xi_2, \eta_2) \in \mathcal{K}'$. Similarly, by symmetry of $(\xi_1, \eta_1)$ and $(\xi_2, \eta_2)$, it suffices to show the estimate \eqref{desired-est-11-17} for the case $(\xi_1 , \eta_1) \times (\xi_2, \eta_2) \in (\mathcal{K}_1 \cup \mathcal{K}_2 ) \times \mathcal{K}_0$. For simplicity, we use
$\widehat{\mathcal{K}} :=  (\mathcal{K}_1 \cup \mathcal{K}_2 ) \times \mathcal{K}_0$.
\begin{defn}
Let $m=(n, z) \in \N  \times \Z$. We define the monotone increasing sequence $\{ a_{A,n} \}_{n \in \N} $ as
\begin{equation*}
a_{A,1} = 0, \qquad a_{A,n+1} = a_{A,n} + \frac{N_1}{\sqrt{(n+1)A}}.
\end{equation*}
and sets $\mathcal{R}_{A,m,1}$,
$\mathcal{R}_{A,m,2}$ as follows:
\begin{align*}
\mathcal{R}_{A,m,1} = &
\left\{ (\xi, \eta)  \in \R^2  \, \left| \,
\begin{aligned} & a_{A,n} \leq | \eta-   ( \sqrt{2}+ 1 )^{\frac{2}{3}} (\sqrt{2} + \sqrt{3} ) \xi |
< a_{A,n+1}, \\
& z A^{-1}  N_1 \leq \eta-( \sqrt{2}+ 1 )^{\frac{2}{3}}\xi < (z+1) A^{-1}N_1
 \end{aligned} \right.
\right\},\\
\mathcal{R}_{A,m,2} = &
\left\{ (\xi, \eta)  \in \R^2  \, \left| \,
\begin{aligned} & a_{A,n} \leq | \eta +  ( \sqrt{2}+ 1 )^{\frac{2}{3}} (\sqrt{3} - \sqrt{2} ) \xi |
< a_{A,n+1}, \\
& z A^{-1}  N_1 \leq \eta-( \sqrt{2}+ 1 )^{\frac{2}{3}}\xi < (z+1) A^{-1}N_1
 \end{aligned} \right.
\right\},\\
\tilde{\mathcal{R}}_{A,m,1} = & \R \times \mathcal{R}_{A,m,1}, \quad
\tilde{\mathcal{R}}_{A,m,2} =  \R \times \mathcal{R}_{A,m,2}.
\end{align*}
We will perform the Whitney type decomposition by using the above sets
instead of simple square tiles.
We define for $i=1,2$ that
\begin{align*}
M_{A,i}^1 & = \left\{ (m, k) \in (\N \times \Z) \times \Z^2 \ \left| \
\begin{aligned} & |\Phi(\xi_1, \eta_1, \xi_2, \eta_2)| \geq A^{-1} N_1^3 \ \\
&\textnormal{for any} \ (\xi_1, \eta_1) \in \mathcal{R}_{A,m,i} \  \textnormal{and} \
(\xi_2, \eta_2) \in \mathcal{T}_{k}^A
 \end{aligned} \right.
\right\},\\
M_{A,i}^2 & = \left\{ (m, k) \in (\N \times \Z) \times \Z^2 \ \left| \
\begin{aligned} & |F(\xi_1, \eta_1, \xi_2, \eta_2)| \geq A^{-1} N_1^3 \ \\
&\textnormal{for any} \ (\xi_1, \eta_1) \in \mathcal{R}_{A,m,i} \  \textnormal{and} \
(\xi_2, \eta_2) \in \mathcal{T}_{k}^A
 \end{aligned} \right.
\right\},\\
M_{A,i} & = M_{A,i}^1 \cup M_{A,i}^2 \subset  (\N \times \Z) \times \Z^2, \\
R_{A,i} & = \bigcup_{(m, k) \in M_{A,i}} \mathcal{R}_{A,m,i} \times
\mathcal{T}_{k}^A \subset \R^2 \times \R^2.
\end{align*}
Furthermore, we define ${M}_{A,i}' \subset M_{A,i}$ as the collection of $(m,k) \in \N \times \Z$ such that
\begin{equation*}
\mathcal{R}_{A,m,i} \times
\mathcal{T}_{k}^A \subset \bigcup_{2^{100} \leq A' <A} R_{A', i}.
\end{equation*}
By using ${M}_{A,i}'$, we define
\begin{align*}
Q_{A,i} =
\begin{cases}
R_{A,i} \setminus {\displaystyle \bigcup_{(m,k) \in {M}_{A,i}'}} (\mathcal{R}_{A,m,i} \times
\mathcal{T}_{k}^A) \ \textnormal{for} \   A \geq 2^{101},\\
 \quad \  R_{2^{100},i}  \, \quad \qquad \qquad \qquad \qquad \textnormal{for} \  A = 2^{100},
\end{cases}
\end{align*}
and $\tilde{M}_{A,i} = M_{A,i} \setminus {M}_{A,i}'$. Clearly, the followings hold.
\begin{equation*}
\bigcup_{(m,k) \in \tilde{M}_{A,i}} \mathcal{R}_{A,m,i} \times
\mathcal{T}_{k}^A = Q_{A,i}, \quad
\bigcup_{2^{100} \leq A \leq A_0} Q_{A,i} = R_{A_0,i},
\end{equation*}
where $A_0 \geq 2^{100}$ is dyadic.
Lastly, we define
\begin{align*}
\widehat{Z}_{A,i} & = \{ (m, k) \in \tilde{M}_{A,i} \, | \,
(  \tilde{\mathcal{R}}_{A,m,i} \times
\tilde{\mathcal{T}}_{k}^A ) \cap \left(
G_{N_1, L_1} \times G_{N_2, L_2} \right) \cap ( \tilde{\mathcal{K}}_i \times
\tilde{\mathcal{K}}_0 ) \not=
\emptyset \},\\
\overline{Z}_{A,i} & = \{ (m, k) \in  M_{A,i}^c \, | \,
(  \tilde{\mathcal{R}}_{A,m,i} \times
\tilde{\mathcal{T}}_{k}^A ) \cap \left(
G_{N_1, L_1} \times G_{N_2, L_2} \right) \cap ( \tilde{\mathcal{K}}_i \times
\tilde{\mathcal{K}}_0 )  \not=
\emptyset \},
\end{align*}
where $M_{A,i}^c = (\N \times \Z) \setminus M_{A,i}$ and $N_1$, $L_1$, $N_2$, $L_2$ satisfy (i) and (ii) in \textit{Case} \textnormal{\ref{ass-1}}. We easily see that
\begin{equation*}
 \left(
G_{N_1, L_1} \times G_{N_2, L_2} \right) \cup ( \tilde{\mathcal{K}}_i \times
\tilde{\mathcal{K}}_0  ) \subset \bigcup_{(m,k) \in \widehat{Z}_{A,i}} (  \tilde{\mathcal{R}}_{A,m,i} \times
\tilde{\mathcal{T}}_{k}^A ) \cup \bigcup_{(m,k) \in \overline{Z}_{A,i}} (  \tilde{\mathcal{R}}_{A,m,i} \times
\tilde{\mathcal{T}}_{k}^A ).
\end{equation*}
\end{defn}
\begin{lem}\label{lemma3.8}
Let $i=1,2$. For fixed $m \in \N \times \Z$, the number of $k \in \Z^2$ such that
$(m, k) \in \widehat{Z}_{A,i}$ is less than $2^{1000}$. On the other hand, for
fixed $k \in \Z^2$, the number of $m \in \N \times \Z$ such that
$(m, k) \in \widehat{Z}_{A,i}$ is less than $2^{1000}$.
Furthermore, the claim holds true whether we replace $\widehat{Z}_{A,i}$ by $\overline{Z}_{A,i}$ in
the above statements.
\end{lem}
The proof of Lemma \ref{lemma3.8} will be long and complicated. For the sake of convenience, we skip the proof of Lemma \ref{lemma3.8} here and give it in Appendix.
\begin{prop}
Assume \textnormal{(i)-(iii)} in \textit{Case} \textnormal{\ref{ass-1}}. Then we have
\begin{align}
&
 \Bigl|\int_{*}{  \ha{w}_{{N_0, L_0}}(\tau, \xi, \eta) \chi_{\widehat{\mathcal{K}}}((\xi_1, \eta_1),
(\xi_2, \eta_2))
\ha{u}_{N_1, L_1}(\tau_1, \xi_1, \eta_1)  \ha{v}_{N_2, L_2}(\tau_2, \xi_2, \eta_2)
}
d\sigma_1 d\sigma_2 \Bigr| \notag\\
& \qquad \qquad \qquad  \lesssim   (N_1)^{-\frac{5}{4}}  (L_0 L_1 L_2)^{\frac{5}{12}} \|\ha{u}_{N_1, L_1} \|_{L^2} \| \ha{v}_{N_2, L_2} \|_{L^2}
\|\ha{w}_{{N_0, L_0}} \|_{L^2},\label{desired-est-prop3.9}
\end{align}
where $d \sigma_j = d\tau_j d \xi_j d \eta_j$ and $*$ denotes $(\tau, \xi, \eta) = (\tau_1 + \tau_2, \xi_1+ \xi_2, \eta_1 + \eta_2).$
\end{prop}
\begin{proof}
For simplicity, we here only consider the case $(\xi_1 , \eta_1) \times (\xi_2, \eta_2) \in \mathcal{K}_1 \times \mathcal{K}_0 =:\widehat{\mathcal{K}}_1$.
The case $(\xi_1 , \eta_1) \times (\xi_2, \eta_2) \in \mathcal{K}_2 \times \mathcal{K}_0$ can be treated in a similar way.
The strategy of the proof is completely the same as that for Proposition \ref{prop3.7}.
By the relation
\begin{equation*}
 \left(
G_{N_1, L_1} \times G_{N_2, L_2} \right) \cup ( \tilde{\mathcal{K}}_i \times
\tilde{\mathcal{K}}_0  ) \subset \bigcup_{(m,k) \in \widehat{Z}_{A,i}} (  \tilde{\mathcal{R}}_{A,m,i} \times
\tilde{\mathcal{T}}_{k}^A ) \cup \bigcup_{(m,k) \in \overline{Z}_{A,i}} (  \tilde{\mathcal{R}}_{A,m,i} \times
\tilde{\mathcal{T}}_{k}^A ),
\end{equation*}
we have
\begin{align*}
&  \Bigl|\int_{*}{  \ha{w}_{{N_0, L_0}}(\tau, \xi, \eta) \chi_{\widehat{\mathcal{K}}_1}((\xi_1, \eta_1),
(\xi_2, \eta_2))
\ha{u}_{N_1, L_1}(\tau_1, \xi_1, \eta_1)  \ha{v}_{N_2, L_2}(\tau_2, \xi_2, \eta_2)
}
d\sigma_1 d\sigma_2 \Bigr| \\
\leq & \sum_{2^{100} \leq A \leq A_0} \sum_{(m,k) \in \widehat{Z}_{A,1}}
\Bigl|\int_{*}{  \ha{w}_{{N_0, L_0}}(\tau, \xi, \eta)
\ha{u}_{N_1, L_1}|_{\tilde{\mathcal{R}}_{A,m,1}}(\tau_1, \xi_1, \eta_1)  \ha{v}_{N_2, L_2}|_{\tilde{\mathcal{T}}_{k}^A}(\tau_2, \xi_2, \eta_2)
}
d\sigma_1 d\sigma_2 \Bigr|\\
& \quad + \sum_{(m, k) \in \overline{Z}_{A_0,1}}
\Bigl|\int_{*}{  \ha{w}_{{N_0, L_0}}(\tau, \xi, \eta)
\ha{u}_{N_1, L_1}|_{\tilde{\mathcal{R}}_{A_0,m,1}}(\tau_1, \xi_1, \eta_1)  \ha{v}_{N_2, L_2}|_{\tilde{\mathcal{T}}_{k}^{A_0}}(\tau_2, \xi_2, \eta_2)
}
d\sigma_1 d\sigma_2 \Bigr|\\
& =: \sum_{2^{100} \leq A \leq A_0} \sum_{(m,k) \in \widehat{Z}_{A,1}}   I_1 + \sum_{(m, k) \in \overline{Z}_{A_0,1}}   I_2.
\end{align*}
We deduce from Propositions \ref{prop3.3}, \ref{prop3.4},Lemma \ref{lemma3.8} and the almost
orthogonality that
\begin{align*}
& \sum_{(m,k) \in \widehat{Z}_{A,1}}   I_1 \\ \lesssim &
\sum_{(m,k) \in \widehat{Z}_{A,1}}
A^{\frac{1}{2}} N_1^{-2} (L_0 L_1 L_2)^{\frac{1}{2}} \|\ha{u}_{N_1, L_1}|_{\tilde{\mathcal{R}}_{A,m,1}} \|_{L^2}
\| \ha{v}_{N_2, L_2}|_{\tilde{\mathcal{T}}_{k}^A} \|_{L^2}
\|\ha{w}_{{N_0, L_0}} \|_{L^2}\\
\lesssim & \
A^{\frac{1}{2}} N_1^{-2} (L_0 L_1 L_2)^{\frac{1}{2}} \|\ha{u}_{N_1, L_1} \|_{L^2}
\| \ha{v}_{N_2, L_2} \|_{L^2}
\|\ha{w}_{{N_0, L_0}} \|_{L^2}.
\end{align*}
Let $A_0$ be the minimal dyadic number which satisfies
$A_0 \geq (L_{012}^{\max})^{-1/2} N_1^{3/2}$.
We obtain
\begin{equation*}
\sum_{2^{100} \leq A \leq A_0} \sum_{(m,k) \in \widehat{Z}_{A,1}}    I_1
 \lesssim
(N_1)^{-\frac{5}{4}}  (L_0 L_1 L_2)^{\frac{5}{12}} \|\ha{u}_{N_1, L_1} \|_{L^2}
\| \ha{v}_{N_2, L_2} \|_{L^2}
\|\ha{w}_{{N_0, L_0}} \|_{L^2}.
\end{equation*}
Next we deal with the latter term. For simplicity, we assume $L_0 =L_{012}^{\max} $.
It follows from \eqref{bilinearStrichartz-1} in Proposition \ref{prop3.2}, Lemma \ref{lemma3.8} and the almost orthogonality that
\begin{align*}
\sum_{(m, k) \in \overline{Z}_{A_0,1}}  I_2 & \lesssim
 (N_1 A_0)^{-\frac{1}{2}} (L_1 L_2)^{\frac{1}{2}} \sum_{(m, k) \in \overline{Z}_{A_0,1}}  \|\ha{u}_{N_1, L_1}|_{\tilde{\mathcal{R}}_{A,m,1}} \|_{L^2}
\| \ha{v}_{N_2, L_2}|_{\tilde{\mathcal{T}}_{k}^{A_0}} \|_{L^2} \|\ha{w}_{{N_0, L_0}} \|_{L^2}\\
& \lesssim (N_1)^{-\frac{5}{4}}  (L_0 L_1 L_2)^{\frac{5}{12}} \|\ha{u}_{N_1, L_1} \|_{L^2}
\| \ha{v}_{N_2, L_2} \|_{L^2}
\|\ha{w}_{{N_0, L_0}} \|_{L^2}.
\end{align*}
This completes the proof.
\end{proof}
\subsection{Case 3: low modulation, parallel interactions}\label{Case 3}
It remains to show \eqref{2018-06-05-01} for Case 3. We recall that,
by Plancherel's theorem, \eqref{2018-06-05-01} can be written
as
\begin{equation}
\begin{split}
&
 \Bigl|\int_{*}{  (\xi+\eta)\ha{w}_{{N_0, L_0}}(\tau, \xi, \eta)
\ha{u}_{N_1, L_1}(\tau_1, \xi_1, \eta_1)  \ha{v}_{N_2, L_2}(\tau_2, \xi_2, \eta_2)
}
d\sigma_1 d\sigma_2 \Bigr| \\
& \qquad \qquad \qquad \qquad \qquad  \lesssim   \|u \|_{X^{s,\,b}} \|v \|_{X^{s,\,b}}
\|w \|_{X^{-s,\,1-b-\e}}.\label{desired-est-12-23}
\end{split}
\end{equation}
In this subsection, we will show \eqref{desired-est-12-23} under the following assumptions.
\begin{case}[low modulation, parallel interactions]\label{case-2} $\quad$\\
(i) $ \, \, $ $L_{012}^{\max} \leq 2^{-100} (N_{012}^{\max})^3$
\begin{equation*}
\textnormal{(ii)'} \
\begin{cases}
\textnormal{If min}(N_0, N_1, N_2) = N_0 \quad
|\sin \angle \left( (\xi_1, \eta_1), (\xi_2, \eta_2) \right)|  \leq 2^{-20},
\\
\textnormal{If min}(N_0, N_1, N_2) = N_1 \quad
|\sin \angle \left( (\xi, \, \eta), \, (\xi_2, \eta_2) \right)| \leq 2^{-20},\\
\textnormal{If min}(N_0, N_1, N_2) = N_2 \quad
|\sin \angle \left( (\xi, \, \eta), \, (\xi_1, \eta_1) \right)| \leq 2^{-20}.
\end{cases}
\end{equation*}
\end{case}
\begin{rem}
Note that \textit{Case} 1, \textit{Case} 2 and \textit{Case} 3 cover all cases. To see this, we show
that if we assume that
$N_0$, $N_1$, $N_2$, $(\xi,\eta)$, $(\xi_1,\eta_1)$, $(\xi_2,\eta_2)$ do not satisfy (ii)' in \textit{Case} 3 then
these satisfy (ii) and (iii) in \textit{Case} 2. For simplicity, we assume $N_2 = \min (N_0, N_1,N_2)$ and
$|\sin \angle \left( (\xi, \, \eta), \, (\xi_1, \eta_1) \right)| > 2^{-20}$. Let $(\xi_1, \eta_1) = (r_1 \cos \theta_1,
r_1 \sin \theta_1)$ and $(\xi, \eta) = (r \cos \theta,
r \sin \theta)$. We have
\begin{align*}
|(\xi_2, \eta_2)| & = |(\xi-\xi_1, \eta - \eta_1)|\\
& = \sqrt{(r \cos \theta - r_1 \cos \theta_1)^2 + (r \sin \theta - r_1 \sin \theta_1)^2}\\
& = \sqrt{r^2 + r_1^2 - 2 r r_1 (\cos \theta \cos \theta_1 + \sin \theta \sin \theta_1)}\\
& = \sqrt{ (r - r_1 )^2 + 2 r r_1 ( 1- \cos (\theta - \theta_1))}\\
& \geq \sqrt{r r_1} |\sin (\theta-\theta_1)|\\
& \geq \frac{N_1}{2} |\sin \angle \left( (\xi, \, \eta), \, (\xi_2, \eta_2) \right)| > 2^{-21} N_1,
\end{align*}
which implies (ii) in \textit{Case} 2. In addition, we see
\begin{align*}
|\sin \angle \left( (\xi_1, \, \eta_1), \, (\xi_2, \eta_2) \right)| & =
\frac{|\xi_1 \eta_2 - \xi_2 \eta_1|}{|(\xi_1,\eta_1)| | (\xi_2, \eta_2)|} \\
 & \geq  \frac{|\xi_1 \eta - \xi \eta_1|}{4 |(\xi_1,\eta_1)| |(\xi, \eta)|}\\
 & = 2^{-2} |\sin \angle \left( (\xi, \, \eta), \, (\xi_1, \eta_1) \right)| > 2^{-22},
\end{align*}
which implies (iii) in \textit{Case} 2.
\end{rem}
Next, we introduce an angular decomposition.
\begin{defn}
We define a partition of unity in $\R$,
\begin{equation*}
1 = \sum_{j \in \Z} \omega_j, \qquad \omega_j (s) = \psi(s-j) \Bigl( \sum_{k \in \Z} \psi (s-k) \Bigr)^{-1}.
\end{equation*}
For a dyadic number $A \geq 64$, we also define a partition of unity on the unit circle,
\begin{equation*}
1 = \sum_{j =0}^{A-1} \omega_j^A, \qquad \omega_j^A (\theta) =
\omega_j \Bigl( \frac{A\theta}{\pi} \Bigr) + \omega_{j-A} \Bigl( \frac{A\theta}{\pi} \Bigr).
\end{equation*}
We observe that $\omega_j^A$ is supported in
\begin{equation*}
\Theta_j^A = \Bigl[\frac{\pi}{A} \, (j-2), \ \frac{\pi}{A} \, (j+2) \Bigr]
\cup \Bigl[-\pi + \frac{\pi}{A} \, (j-2), \ - \pi +\frac{\pi}{A} \, (j+2) \Bigr].
\end{equation*}
We now define the angular frequency localization operators $R_j^A$,
\begin{equation*}
\F_{x,y} (R_j^A f)(\xi,\eta) = \omega_j^A(\theta) \F_{x,y} f(\xi, \eta),
\quad \textnormal{where} \ (\xi, \eta) = |(\xi,\eta)|
(\cos \theta, \sin \theta).
\end{equation*}
For any function $u  : \, \R \, \times \, \R^2 \, \to \C$, $(t,x,y) \mapsto u(t,x,y)$ we set
$(R_j^A u ) (t, x,y) = (R_j^Au( t, \cdot)) (x,y)$. These operators localize function in frequency to the sets
\begin{equation*}
\tilde{\mathfrak{D}}_j^A = \R \times {\mathfrak{D}}_j^A,
\end{equation*}
where ${\mathfrak{D}}_j^A  = \{ ( |(\xi,\eta)| \cos \theta, |(\xi, \eta)| \sin \theta) \in  \R^2
\, | \, \theta \in \Theta_j^A  \} $.

Immediately, we can see
\begin{equation*}
u = \sum_{j=0}^{A-1} R_j^A u.
\end{equation*}
Let $\mathcal{I}_1$, $\mathcal{I}_2$, $\mathcal{I}_3 \subset \R^2 \times \R^2$ be defined as follows:
\begin{align*}
\mathcal{I}_1 & = \bigl( {\mathfrak{D}}_{0}^{2^{11}} \times {\mathfrak{D}}_{0}^{2^{11}} \bigr)
\cup \bigl( {\mathfrak{D}}_{2^{10}}^{2^{11}} \times {\mathfrak{D}}_{2^{10}}^{2^{11}} \bigr) , &
\tilde{\mathcal{I}}_1 & = \bigl( \tilde{{\mathfrak{D}}}_{0}^{2^{11}} \times \tilde{{\mathfrak{D}}}_{0}^{2^{11}} \bigr)
\cup \bigl( \tilde{{\mathfrak{D}}}_{2^{10}}^{2^{11}} \times \tilde{{\mathfrak{D}}}_{2^{10}}^{2^{11}} \bigr),\\
\mathcal{I}_2 & = \bigl( {\mathfrak{D}}_{2^9 \times 3}^{2^{11}} \times {\mathfrak{D}}_{2^9 \times 3}^{2^{11}} \bigr)  ,&
\tilde{\mathcal{I}}_2 & = \bigl( \tilde{{\mathfrak{D}}}_{2^9 \times 3}^{2^{11}} \times \tilde{{\mathfrak{D}}}_{2^9 \times 3}^{2^{11}} \bigr),\\
\mathcal{I}_3 & = ( \R^2 \times \R^2 ) \setminus  ( \mathcal{I}_1 \cup \mathcal{I}_2 ), &
\tilde{\mathcal{I}}_3 & = ( \R^3 \times \R^3 ) \setminus
\bigl( \tilde{\mathcal{I}}_1 \cup \tilde{\mathcal{I}}_2 \bigr).
\end{align*}
Note that
\begin{align*}
 {\mathfrak{D}}_{0}^{2^{11}} & =\left\{ ( |(\xi,\eta)| \cos \theta, |(\xi, \eta)| \sin \theta) \in  \R^2
\, | \, \textnormal{min} \left( |\theta|, |\theta - \pi| \right)  \leq 2^{-10}\pi \right\},\\
 {\mathfrak{D}}_{2^{10}}^{2^{11}} & =\Bigl\{ ( |(\xi,\eta)| \cos \theta, |(\xi, \eta)| \sin \theta) \in  \R^2
\, | \, \textnormal{min} \Bigl( \Bigl| \theta-\frac{\pi}{2}\Bigr|,
\Bigl| \theta + \frac{\pi}{2} \Bigr| \Bigr) \leq 2^{-10}  \pi
 \Bigr\},\\
 {\mathfrak{D}}_{2^{9} \times 3}^{2^{11}} & =\Bigl\{ ( |(\xi,\eta)| \cos \theta, |(\xi, \eta)| \sin \theta) \in  \R^2
\, | \, \textnormal{min} \Bigl( \Bigl| \theta-\frac{3 \pi}{4}\Bigr|,
\Bigl| \theta + \frac{\pi}{4} \Bigr| \Bigr) \leq 2^{-10} \pi
 \Bigr\}.
\end{align*}
\end{defn}
We turn to show \eqref{desired-est-12-23}.
We only consider the cases $\min(N_0, N_1, N_2) = N_2$ and $\min(N_0, N_1, N_2) = N_0$.
By symmetry, the same argument for $\min(N_0, N_1, N_2) = N_2$
can be applied to the case $\min(N_0, N_1, N_2) = N_1$.
In addition, we mainly treat the case $\min(N_0, N_1, N_2) = N_2$ and skip the most of the proof for the case $\min(N_0, N_1, N_2) = N_0$ to avoid redundancy.

Suppose $\min(N_0, N_1, N_2) = N_2$.
We divide the proof of \eqref{desired-est-12-23} into the three cases.\\
(I) $(\xi_1,\eta_1)\times (\xi,\eta) \in \mathcal{I}_1 $, $ \ $
\textnormal{(I \hspace{-0.15cm}I)} $(\xi_1,\eta_1)\times (\xi,\eta) \in \mathcal{I}_2$, $ \ $
\textnormal{(I \hspace{-0.16cm}I \hspace{-0.16cm}I)}
$(\xi_1,\eta_1)\times (\xi,\eta) \in \mathcal{I}_3$.

In the first case, both of $(\xi_1, \eta_1)$ and $(\xi,\eta)$ are close to
the $\xi$-axis or close to the $\eta$-axis.
We recall that the transversality is determined by the following product.
\begin{equation*}
|\xi_1 \eta - \xi \eta_1| \, |\xi_1 \eta +  \xi \eta_1 + 2 (\xi_1 \eta_1 + \xi \eta)|.
\end{equation*}
In the case (I), the both sizes of former function $|\xi_1 \eta - \xi \eta_1|$ and the latter function $|\xi_1 \eta +  \xi \eta_1 + 2 (\xi_1 \eta_1 + \xi \eta)|$ are small.
Therefore, we will see that this case is the most difficult to show \eqref{desired-est-12-23} and the proof  become complicated.
Next, the interactions which arise near the line
$\xi + \eta = 0$ will be treated in the case \textnormal{(I \hspace{-0.17cm}I)}.
The interesting point is that if
$\xi_1+ \eta_1 = \xi + \eta =0$ then the resonance function $\Phi(\xi_1,\eta_1, \xi,\eta)$ is always equal to $0$ and the derivative loss $\xi+\eta$ in \eqref{desired-est-12-23} also
become $0$. Thus, we will deal with these interactions by a similar argument which was utilized to show the estimates of a null-form
nonlinearity. In the last case \textnormal{(I \hspace{-0.16cm}I \hspace{-0.16cm}I)}, \eqref{desired-est-12-23} can be verified by the same strategy as that for the Zakharov system in \cite{BHHT09}. Therefore, we first consider the case \textnormal{(I \hspace{-0.16cm}I \hspace{-0.16cm}I)}. Next, we treat the case \textnormal{(I \hspace{-0.15cm}I)}, and lastly we prove \eqref{desired-est-12-23} in the case (I).
The following proposition immediately gives \eqref{desired-est-12-23} in the case \textnormal{(I \hspace{-0.17cm}I \hspace{-0.17cm}I)}.
\begin{prop}\label{prop3.10}
Assume \textnormal{(i)}, \textnormal{(ii)'} in \textit{Case} \textnormal{\ref{case-2}}
and $\min(N_0, N_1, N_2) = N_2$. Then we have
\begin{equation}
\begin{split}
&
\Bigl|\int_{**}{  \ha{v}_{{N_2, L_2}}(\tau_2, \xi_2, \eta_2) \chi_{\mathcal{I}_3}((\xi_1,\eta_1), (\xi,\eta))
\ha{u}_{N_1, L_1}(\tau_1, \xi_1, \eta_1)
\ha{w}_{N_0, L_0}(\tau, \xi, \eta)
}
d\sigma_1 d\sigma \Bigr|\\
& \qquad \qquad \qquad \qquad
\lesssim  N_1^{-\frac{5}{4}} L_0^{\frac{1}{4}} (L_1 L_2)^{\frac{1}{2}} \|\ha{u}_{N_1, L_1}\|_{L^2}
\| \ha{v}_{N_2, L_2}\|_{L^2}
\|\ha{w}_{{N_0, L_0}}\|_{L^2},\label{prop3.10-1}
\end{split}
\end{equation}
where $d \sigma_1 = d\tau_1 d \xi_1 d \eta_1$, $d \sigma = d\tau d \xi d \eta$
 and $**$ denotes $(\tau_2, \xi_2, \eta_2) = (\tau_1 + \tau, \xi_1+ \xi, \eta_1 + \eta).$
\end{prop}
Similarly to the argument in Case 2, by employing the bilinear Strichartz estimates and the nonlinear
Loomis-Whitney inequality, we show Proposition \ref{prop3.10}.
First we observe that low modulation condition $L_{012}^{\max} \leq  2^{-20}A^{-1} (N_{012}^{\max})^3$
provides the smallness of the low frequency if $(\xi_1, \eta_1)$ and $(\xi, \eta)$ are almost parallel.
\begin{lem}\label{lemma3.11}
Let $A \geq 2^{25}$ and
\begin{equation*}
(\tau_1, \xi_1, \eta_1) \in G_{N_1, L_1}\cap  \tilde{{\mathfrak{D}}}_{j_1}^A, \
(\tau,\xi, \eta) \in G_{N_0,L_0} \cap  \tilde{\mathfrak{D}}_{j}^A, \ (\tau_1+\tau, \xi_1+\xi, \eta_1+\eta) \in G_{N_2, L_2}.
\end{equation*}
Assume $L_{012}^{\max} \leq  2^{-20}A^{-1} N_1^3$, $|j_1-j| \leq 32$,
$\min(N_0, N_1, N_2) = N_2$ and
\begin{equation*}
(\xi_1,\eta_1) \times (\xi,\eta) \notin \mathcal{I}_2.
\end{equation*}
Then we have $N_2 \leq 2^{11}A^{-1} N_1$.
\end{lem}
\begin{proof}
Put $r_1=|(\xi_1,\eta_1)|$, $r =|(\xi,\eta)|$. $\theta_1$, $\theta \in [0,2\pi)$ denote angular variables defined by
\begin{equation*}
(\xi_1,\eta_1) = r_1 (\cos \theta_1, \sin \theta_1), \quad
(\xi,\eta) = r (\cos \theta, \sin \theta).
\end{equation*}
Since $(\xi_1,\eta_1) \times (\xi,\eta) \notin \mathcal{I}_2$, without loss of generality, we may
assume that $(\xi_1,\eta_1) \notin  {\mathfrak{D}}_{2^{9} \times 3}^{2^{11}}$ which gives $|\cos \theta_1 + \sin \theta_1| = \sqrt{2} |\sin (\theta_1 + \pi/4)| > 2^{-11}\pi$.
We deduce from the assumption $|j_1-j| \leq 32$ that
$ |(\cos \theta_1,\sin \theta_1 ) - (\cos \theta, \sin \theta)| \leq 2^{7} A^{-1}$ or $ |(\cos \theta_1,\sin \theta_1 ) + (\cos \theta, \sin \theta)| \leq 2^{7} A^{-1}$.
For the former case, we observe that
\begin{align*}
|\Phi(\xi_1,\eta_1,\xi,\eta) |= & |\xi_1\xi (\xi_1 +\xi) + \eta_1 \eta(\eta_1+\eta)|\\
\geq  & r_1 r (r_1 + r)|\cos^3 \theta_1+\sin^3 \theta_1 |- 2^{9} A^{-1}r_1 r (r_1 + r)\\
= &  r_1 r (r_1 + r) (1- 2^{-1}\sin 2\theta_1) | \cos \theta_1 + \sin \theta_1| -2^{9} A^{-1}r_1 r (r_1 + r)\\
\geq & 2^{-13} r_1 r (r_1 + r),
\end{align*}
which contradicts the assumption $L_{012}^{\max} \leq  2^{-20}A^{-1} N_1^3$. Similarly, for the latter case, we get $|\Phi(\xi_1,\eta_1,\xi,\eta)| \geq 2^{-13} r_1 r (r_1 - r)$. This and low modulations assumption $L_{012}^{\max} \leq  2^{-20}A^{-1} N_1^3$ yield $|r_1-r| \leq A^{-1}N_1$. We conclude that
\begin{align*}
|(\xi_1+ \xi, \eta_1+\eta)| & \leq | (r_1 \cos \theta_1 + r \cos \theta, r_1\sin \theta_1 + r \sin \theta)| \\
& \leq |r_1-r|+ 2^{8} A^{-1}r \leq 2^{10} A^{-1} N_1.
\end{align*}
\end{proof}
\begin{rem}\label{rem3.5}
For any $(\xi_1, \eta_1)$, $(\xi,\eta)$ which satisfy $|\sin \angle \left( (\xi, \, \eta), \, (\xi_1, \eta_1) \right)| \leq 2^{-20}$, there exist $j_1$ and $j$ such that $|j_1-j| \leq 32$ and $(\xi_1,\eta_1) \times (\xi,\eta) \in
{\mathfrak{D}}_{j_1}^{25} \times {\mathfrak{D}}_{j}^{25}$.
Therefore, Lemma \ref{lemma3.11} implies that if we assume \textnormal{(i)}, \textnormal{(ii)'} in \textit{Case} \textnormal{\ref{case-2}}
, $\min(N_0, N_1, N_2) = N_2$ and $
(\xi_1,\eta_1) \times (\xi,\eta) \notin \mathcal{I}_2$,
then we have $N_2 \leq 2^{-14} N_1$.
\end{rem}
We now show the bilinear Strichartz estimates.
Here we use the angular decomposition
${\mathfrak{D}}_{j}^A$ instead of the square-tile decomposition $\mathcal{T}_{k}^{A}$.
\begin{prop}\label{prop3.12}
Assume \textnormal{(i)}, \textnormal{(ii)'} in \textit{Case} \textnormal{\ref{case-2}}
and $\min(N_0, N_1, N_2) = N_2.$ Let $A \geq 2^{25}$ be dyadic, $|j_1-j| \leq 32$ and
\begin{equation*}
\left( {\mathfrak{D}}_{j_1}^A \times {\mathfrak{D}}_{j}^A \right) \subset \mathcal{I}_3.
\end{equation*}
Then we have
\begin{align}
& \Bigl\| \chi_{G_{N_2, L_2}} \int \ha{u}_{N_1, L_1}|_{\tilde{\mathfrak{D}}_{j_1}^A} (\tau_1, \xi_1, \eta_1) \ha{w}_{N_0, L_0}|_{\tilde{\mathfrak{D}}_{j}^A} (\tau_2 - \tau_1, \xi_2-\xi_1, \eta_2- \eta_1) d\sigma_1
\Bigr\|_{L_{\xi_2, \eta_2, \tau_2}^2} \notag \\
& \qquad \qquad \qquad \qquad \qquad
\lesssim N_1^{-\frac{1}{2}} (L_0 L_1)^{\frac{1}{2}} \|\ha{u}_{N_1, L_1}|_{\tilde{\mathfrak{D}}_{j_1}^A}\|_{L^2}
\|\ha{w}_{N_0, L_0}|_{\tilde{\mathfrak{D}}_{j}^A} \|_{L^2},\label{bilinearStrichartz-4}\\
& \Bigl\| \chi_{G_{N_1, L_1} \cap \tilde{\mathfrak{D}}_{j_1}^A} \int \ha{w}_{N_0, L_0}|_{\tilde{\mathfrak{D}}_{j}^A}
(\tau, \xi, \eta)
\ha{v}_{N_2, L_2} (\tau_1+ \tau, \xi_1+\xi, \eta_1+ \eta) d\sigma \Bigr\|_{L_{\xi_1, \eta_1, \tau_1}^2}
\notag \\
& \qquad \qquad \qquad \qquad \qquad
\lesssim (A N_1)^{-\frac{1}{2}} (L_0 L_2)^{\frac{1}{2}}
\|\ha{w}_{N_0, L_0}|_{\tilde{\mathfrak{D}}_{j}^A} \|_{L^2}
\|\ha{v}_{N_2, L_2}\|_{L^2}, \label{bilinearStrichartz-5}\\
& \Bigl\| \chi_{G_{N_0,L_0} \cap \tilde{\mathfrak{D}}_{j}^A} \int \ha{v}_{N_2, L_2} (\tau_1+ \tau, \xi_1+\xi, \eta_1+ \eta)  \ha{u}_{N_1, L_1}|_{\tilde{\mathfrak{D}}_{j_1}^A} (\tau_1, \xi_1, \eta_1) d \sigma_1 \Bigr\|_{L_{\xi, \eta, \tau}^2} \notag \\
& \qquad \qquad \qquad \qquad \qquad
\lesssim ( A N_1)^{-\frac{1}{2}} (L_1 L_2)^{\frac{1}{2}} \|\ha{v}_{N_2, L_2} \|_{L^2}
\|\ha{u}_{N_1, L_1}|_{\tilde{\mathfrak{D}}_{j_1}^A} \|_{L^2} .\label{bilinearStrichartz-6}
\end{align}
\end{prop}
\begin{proof}
As mentioned in \textit{Remark} \ref{rem3.5}, the assumption $\left( {\mathfrak{D}}_{j_1}^A \times {\mathfrak{D}}_{j}^A \right) \subset \mathcal{I}_3$
and Lemma \ref{lemma3.11} imply $N_2 \leq 2^{-14} N_1$. Since $\left( {\mathfrak{D}}_{j_1}^A \times {\mathfrak{D}}_{j}^A \right) \subset \mathcal{I}_3$, we may assume
$\min(|\xi_1|, |\eta_1|, |\xi_2-\xi_1|, |\eta_2-\eta_1|) \gtrsim N_1$ in (LHS) of
\eqref{bilinearStrichartz-4}. Therefore,
\eqref{bilinearStrichartz-4} is immediately obtained by H\"{o}lder's inequality and the Strichartz estimates
\eqref{Strichartz-3} with $p=q=4$.
For \eqref{bilinearStrichartz-6}, by the same argument as in the proof of
Proposition \ref{prop3.2}, it suffices to show
\begin{equation}
\sup_{(\tau, \xi, \eta) \in G_{N_0, L_0}\cap \tilde{\mathfrak{D}}_{j}^A}
|E(\tau, \xi, \eta)| \lesssim (A N_1)^{-1} L_1 L_2, \label{est01-prop3.12}
\end{equation}
where
\begin{equation*}
E(\tau, \xi, \eta) := \{ (\tau_1, \xi_1, \eta_1) \in G_{N_1, L_1} \cap \tilde{\mathfrak{D}}_{j_1}^A
\, | \, (\tau_1+\tau, \xi_1+ \xi, \eta_1+\eta) \in G_{N_2,L_2} \}.
\end{equation*}
First, for fixed $(\xi, \eta)$, we deduce from $(\tau_1, \xi_1, \eta_1) \in G_{N_1, L_1}$ and $(\tau_1+\tau, \xi_1+ \xi, \eta_1+\eta) \in G_{N_2,L_2}$ that
\begin{equation}
\sup_{(\tau, \xi, \eta) \in G_{N_0, L_0} \cap \tilde{\mathfrak{D}}_{j}^A}
| \{ \tau_1 \, | \, (\tau_1, \xi_1, \eta_1) \in E(\tau, \xi, \eta) \}|
\lesssim \min(L_1, L_2).\label{est02-prop3.12}
\end{equation}
We observe that
\begin{align*}
 |3 \Phi(\xi_1,\eta_1,\xi,\eta) -\tau + \xi^3 +\eta^3| & =
|3 \xi_1 \xi (\xi_1+\xi) + 3 \eta_1 \eta (\eta_1+\eta)-\tau + \xi^3 +\eta^3|\\
& = |  (\tau_1 - \xi_1^3 -\eta_1^3) - (\tau_1+\tau - (\xi_1+ \xi)^3 - (\eta_1 + \eta)^3)|\\
& \lesssim \max(L_1, L_2).
\end{align*}
We write $(\xi_1,\eta_1) = (r_1 \cos \theta_1, r_1 \sin \theta_1)$ and
$(\xi,\eta) = r (\cos \theta, \sin \theta)$. We calculate that
\begin{align*}
|\partial_{r_1}  \Phi(\xi_1,\eta_1,\xi,\eta)| & = |(\cos \theta_1 \partial_{\xi_1} + \sin \theta_1
\partial_{\eta_1})\Phi(\xi_1,\eta_1,\xi,\eta)|\\
& = \, r \bigl|\cos \theta_1 \cos \theta (r_1 \cos \theta_1 + \xi_1 + \xi)
+ \sin \theta_1 \sin \theta (r_1 \sin \theta_1 +
\eta_1 + \eta) \bigr|\\
& \geq \, r_1 r |\cos^2 \theta_1 \cos \theta + \sin^2 \theta_1 \sin \theta|-r(|\xi_1+ \xi|+|\eta_1 + \eta|)\\
& \geq \,  r_1 r (1- 2^{-1}\sin 2\theta_1)  |\cos \theta_1 + \sin \theta_1| - 2^{-11}r N_1\\
& \geq \, 2^{-12} N_1^2.
\end{align*}
Here we used the assumptions $A \geq 2^{25}$ and $|j_1-j| \leq 32$ which provide
$ |(\cos \theta_1,\sin \theta_1 ) - (\cos \theta, \sin \theta)| \leq 2^{-18}$ or $ |(\cos \theta_1,\sin \theta_1 ) + (\cos \theta, \sin \theta)| \leq 2^{-18}$.
The above two estimates imply that $r_1$ is confined to a set of measure
$\max(L_1,L_2)/N_1^2$ for fixed $\theta_1$. In addition, it follows from $(\xi_1, \eta_1)  \in  {\mathfrak{D}}_{j_1}^{A}$ that $\theta_1$ is confined to a set of measure $\sim A^{-1}$.
We observe that
\begin{align*}
  | \{ (\xi_1, \eta_1) \ | \ (\tau_1, \xi_1, \eta_1) \in E(\tau, \xi,\eta) \} |
= &  \int_{\theta_1} \int_{r_1} {\chi}_{E(\tau,\xi,\eta)} (|\xi_1|, \theta_1) r_1 d r_1 d \theta_1 \\
\lesssim & (N_1A)^{-1}\max (L_1, L_2).
\end{align*}
This and \eqref{est02-prop3.12} yield \eqref{est01-prop3.12}. \eqref{bilinearStrichartz-5} is verified by taking duality of \eqref{bilinearStrichartz-6}.
\end{proof}
\begin{prop}\label{prop3.13}
Assume \textnormal{(i)}, \textnormal{(ii)'} in \textit{Case} \textnormal{\ref{case-2}}
and $\min(N_0, N_1, N_2) = N_2.$ Let $A \geq 2^{25}$ be dyadic,
$j_1$, $j$ satisfy $16 \leq |j_1-j| \leq 32$ and
\begin{equation*}
\left( {\mathfrak{D}}_{j_1}^A \times {\mathfrak{D}}_{j}^A \right) \subset \mathcal{I}_3.
\end{equation*}
Then we have
\begin{equation}
\begin{split}
&
\Bigl|\int_{**}{  \ha{v}_{{N_2, L_2}}(\tau_2, \xi_2, \eta_2)
\ha{u}_{N_1, L_1}|_{\tilde{\mathfrak{D}}_{j_1}^A}(\tau_1, \xi_1, \eta_1)
\ha{w}_{N_0, L_0}|_{\tilde{\mathfrak{D}}_{j}^A}(\tau, \xi, \eta)
}
d\sigma_1 d\sigma \Bigr|\\
& \qquad \qquad \qquad  \lesssim  A^{\frac{1}{2}} N_1^{-2} (L_0 L_1 L_2)^{\frac{1}{2}} \|\ha{u}_{N_1, L_1}|_{\tilde{\mathfrak{D}}_{j_1}^A} \|_{L^2}
\| \ha{v}_{N_2, L_2}\|_{L^2}
\|\ha{w}_{{N_0, L_0}}|_{\tilde{\mathfrak{D}}_{j}^A}  \|_{L^2},\label{prop3.11-1}
\end{split}
\end{equation}
where $d \sigma_1 = d\tau_1 d \xi_1 d \eta_1$, $d \sigma = d\tau d \xi d \eta$
 and $**$ denotes $(\tau_2, \xi_2, \eta_2) = (\tau_1 + \tau, \xi_1+ \xi, \eta_1 + \eta).$
\end{prop}
\begin{proof}
We divide the proof into the two cases, $N_2 \geq 2^{100}N_1 A^{-1}$ and $N_2 \leq 2^{100}N_1 A^{-1}$.\\
\underline{Case $N_2  \geq 2^{100}N_1 A^{-1}$}

In this case, by Lemma \ref{lemma3.11}, we get $|\Phi(\xi_1,\eta_1,\xi,\eta)| \gtrsim N_1^3 A^{-1}$.
From this and Proposition \ref{prop3.12}, we easily obtain \eqref{prop3.11-1}.\\
\underline{Case $N_2  \leq 2^{100}N_1 A^{-1}$}

The strategy of the proof is the same as for Proposition \ref{prop3.4}. Clearly, we may assume that $A \geq 2^{100}$. Since $N_2 \leq  2^{100} N_1 A^{-1}$,
by applying harmless decompositions, we may assume that $\ha{u}_{N_1, L_1}$, $\ha{v}_{N_2, L_2}$ and
$\ha{w}_{{N_0, L_0}}$ are supported in square prisms whose side length is $2^{-100} N_1 A^{-1}$. Let $A' = 2^{100} A$. Put
$f$, $g$, $h$ to satisfy
\begin{equation*}
\operatorname{supp} f \subset G_{N_1, L_1} \cap \tilde{\mathcal{T}}_{k_1}^{A'}, \quad
\operatorname{supp} g \subset G_{N_0, L_0} \cap \tilde{\mathcal{T}}_{k_2}^{A'}, \quad
\operatorname{supp} h \subset G_{N_2, L_2} \cap \tilde{\mathcal{T}}_{k_3}^{A'},
\end{equation*}
where $\left( \mathcal{T}_{k_1}^{A'} \times \mathcal{T}_{k_2}^{A'}\right) \cap \left( {\mathfrak{D}}_{j_1}^A \times {\mathfrak{D}}_{j}^A \right) \not= \emptyset$. The desired estimate \eqref{prop3.11-1} is reduced to the estimate
\begin{equation}
\begin{split}
& \Bigl|\int_{\R^3 \times \R^3} { h (\tau_1+\tau_2, \xi_1+\xi_2, \eta_1+\eta_2)
f(\tau_1, \xi_1, \eta_1)  g(\tau_2, \xi_2, \eta_2)
}
d\sigma_1 d\sigma_2 \Bigr|\\
& \qquad \qquad \qquad \qquad \lesssim A^{\frac{1}{2}} N_1^{-2} (L_0 L_1 L_2)^{\frac{1}{2}} \|f \|_{L^2}
\|g\|_{L^2}
\| h \|_{L^2}.
\end{split}\label{est01-prop3.11}
\end{equation}
As we saw in the proof of Proposition \ref{prop3.4}, it suffices to show that
\begin{equation}
\| \tilde{f} |_{S_1} * \tilde{g} |_{S_2} \|_{L^2(S_3)} \lesssim A^{\frac{1}{2}}
\| \tilde{f} \|_{L^2(S_1)} \| \tilde{g} \|_{L^2(S_2)}.\label{est02-prop3.12}
\end{equation}
Here we used the same notations $\tilde{f}$, $\tilde{g}$, $S_1$, $S_2$, $S_3$ as in the proof of Proposition \ref{prop3.4}:
\begin{align*}
& \tilde{f}  (\tau_1, \xi_1 , \eta_1)  = f (N_1^3 \tau_1 , N_1 \xi_1, N_1 \eta_1), \\
& \tilde{g} (\tau_2, \xi_2, \eta_2)  = g (N_1^3 \tau_2, N_1 \xi_2, N_1 \eta_2), \\
& S_1  = \{ \phi_{\tilde{c_1}} (\xi, \eta) = (\xi^3 + \eta^3 + \tilde{c_1}, \, \xi, \, \eta) \in \R^3 \ | \
(\xi, \eta) \in \mathcal{T}_{\tilde{k_1}}^{N_1^{-1}A'} \}, \\
& S_2  = \{ \phi_{\tilde{c_2}} (\xi, \eta) = (\xi^3 + \eta^3 + \tilde{c_2}, \, \xi, \, \eta) \in \R^3 \ | \ (\xi, \eta) \in \mathcal{T}_{\tilde{k_2}}^{N_1^{-1}A'} \},\\
& S_3  = \Bigl\{ (\psi (\xi,\eta), \xi, \eta) \in \R^3  \ | \  (\xi, \eta) \in \mathcal{T}_{\tilde{k_3}}^{N_1^{-1}A'},
\  \psi (\xi,\eta) =  \xi^3 + \eta^3 + \frac{c_0'}{N_1^{3}} \Bigr\},
\end{align*}
where $\tilde{k_i} = k_i/N_1$ with $i = 1,2,3$.
We easily see
\begin{equation}
\textnormal{diam} (S_i) \leq 2^{-80} A^{-1} \qquad  \textnormal{for} \ i=1,2,3.\label{diam-prop3.12}
\end{equation}
For any $\lambda_i \in S_i$, we can find $(\xi_1, \eta_1)$, $(\xi_2, \eta_2)$, $(\xi, \eta)$ which satisfy
\begin{equation*}
\lambda_1=\phi_{\tilde{c_1}} (\xi_1, \eta_1), \quad \lambda_2 =  \phi_{\tilde{c_2}} (\xi_2,\eta_2), \quad
\lambda_3 = (\psi (\xi,\eta), \xi,\eta).
\end{equation*}
Let $(\xi_1', \eta_1')$, $(\xi_2', \eta_2')$, $(\xi', \eta')$ satisfy
\begin{align*}
& (\xi_1', \eta_1') + (\xi_2', \eta_2') = (\xi', \eta'), \\
\lambda_1' = \phi_{\tilde{c_1}} (\xi_1', \eta_1')  \in S_1, \ \ & \lambda_2' = \phi_{\tilde{c_2}} (\xi_2', \eta_2')\in S_2, \ \
\lambda_3' = (\psi (\xi',\eta'), \xi',\eta') \in S_3.
\end{align*}
Similarly to the proof of Proposition \ref{prop3.4}, the hypersurfaces $S_1$, $S_2$, $S_3$ satisfy the following estimates:
\begin{align*}
\sup_{\lambda_i, \lambda_i' \in S_i} & \frac{|\mathfrak{n}_i(\lambda_i) -
\mathfrak{n}_i(\lambda_i')|}{|\lambda_i - \lambda_i'|}
+ \frac{|\mathfrak{n}_i(\lambda_i) (\lambda_i - \lambda_i')|}{|\lambda_i - \lambda_i'|^2} \leq 2^3,\\
&|{\mathfrak{n}}_1(\lambda_1) - {\mathfrak{n}}_1(\lambda_1')| \leq 2^{-70} A^{-1},\\
& |{\mathfrak{n}}_2(\lambda_2) - {\mathfrak{n}}_2(\lambda_2')| \leq 2^{-70} A^{-1},\\
& |{\mathfrak{n}}_3(\lambda_3) - {\mathfrak{n}}_3(\lambda_3')| \leq 2^{-70} A^{-1}.
\end{align*}
Thus it suffices to show
\begin{equation*}
2^{-30} A^{-1} \leq |\textnormal{det} N(\lambda_1', \lambda_2', \lambda_3')| .
\end{equation*}
Since $\lambda_1' =  \phi_{\tilde{c_1}} (\xi_1', \eta_1')$, $\lambda_2' = \phi_{\tilde{c_2}} (\xi_2', \eta_2')$, $\lambda_3' =  (\psi (\xi',\eta'), \xi',\eta') $ and
$(\xi_1', \eta_1') + (\xi_2', \eta_2') = (\xi', \eta')$, we get
\begin{align*}
|\textnormal{det} N(\lambda_1', \lambda_2', \lambda_3')| \geq &
2^{-10} \frac{1}{\LR{(\xi_1, \eta_1) }^2 \LR{(\xi_2, \eta_2)}^2} \left|\textnormal{det}
\begin{pmatrix}
-1 & -1 & - 1 \\
 3 (\xi_1')^2  &  3 (\xi_2')^2  & 3 (\xi')^2 \\
 3 (\eta_1')^2   & 3 (\eta_2')^2  & 3 (\eta')^2
\end{pmatrix} \right| \notag \\
\geq & 2^{-10}\frac{|\xi_1' \eta_2' - \xi_2' \eta_1' |}{\LR{(\xi_1, \eta_1) }^2 \LR{(\xi_2, \eta_2)}^2}
| \xi_1' \eta_2' +  \xi_2' \eta_1' + 2 (\xi_1' \eta_1' + \xi_2' \eta_2')|
\\
\geq & 2^{-20}
\frac{|\xi_1' \eta_2' - \xi_2' \eta_1' |}{|(\xi_1', \eta_1')| \, |(\xi_2', \eta_2')|}| \xi_1' \eta_2' +  \xi_2' \eta_1' + 2 (\xi_1' \eta_1' + \xi_2' \eta_2')| \\
\geq & 2^{-20}A^{-1}| \xi_1' \eta_2' +  \xi_2' \eta_1' + 2 (\xi_1' \eta_1' + \xi_2' \eta_2')|.
\end{align*}
Here we used the relation $\bigl( \mathcal{T}_{\tilde{k_1}}^{N_1^{-1} A'} \times
\mathcal{T}_{\tilde{k_2}}^{N_1^{-1}A'}\bigr) \cap ( {\mathfrak{D}}_{j_1}^A \times {\mathfrak{D}}_{j}^A ) \not= \emptyset$ which gives
\begin{equation*}
\frac{|\xi_1' \eta_2' - \xi_2' \eta_1' |}{|(\xi_1', \eta_1')| \, |(\xi_2', \eta_2')|} \geq A^{-1}.
\end{equation*}
We only need to show
\begin{equation*}
| \xi_1' \eta_2' +  \xi_2' \eta_1' + 2 (\xi_1' \eta_1' + \xi_2' \eta_2')| \geq 2^{-10}.
\end{equation*}
We deduce from $\bigl( \mathcal{T}_{\tilde{k_1}}^{N_1^{-1} A'} \times
\mathcal{T}_{\tilde{k_2}}^{N_1^{-1}A'}\bigr) \cap ( {\mathfrak{D}}_{j_1}^A \times {\mathfrak{D}}_{j}^A ) \not= \emptyset$ and $( {\mathfrak{D}}_{j_1}^A \times {\mathfrak{D}}_{j}^A )
\subset \mathcal{I}_3$ that $|\xi_1' \eta_1'| \geq 2^{-10}$. Then we have
\begin{align*}
| \xi_1' \eta_2' +  \xi_2' \eta_1' + 2 (\xi_1' \eta_1' + \xi_2' \eta_2')| &
\geq 2 | \xi_1' \eta_1' + \xi_2'(\eta_1'+\eta_2')| - |\xi_1' \eta_2' - \xi_2' \eta_1'|\\
& \geq 2 |\xi_1' \eta_1'| - 2^2 \frac{N_2}{N_1} - 2^2 A^{-1}\\
& \geq 2^{-10}.
\end{align*}
\end{proof}
\begin{proof}[Proof of Proposition \ref{prop3.10}]
It suffices to show the estimate \eqref{prop3.10-1} for non-negative
$\ha{w}_{{N_0, L_0}}$, $\ha{u}_{N_1, L_1}$,
$\ha{v}_{N_2, L_2} $.
We define that
\begin{equation*}
J_{A}^{\mathcal{I}_3} = \{ (j_1, j) \, | \, 0 \leq j_1,j \leq A -1, \ \left( {\mathfrak{D}}_{j_1}^A \times {\mathfrak{D}}_{j}^A \right) \subset \mathcal{I}_3\}.
\end{equation*}
Let $A_0 \geq 2^{25}$ be dyadic which will be chosen later. We apply
the Whitney type decomposition of angular variables to  $\mathcal{I}_3$ as follows:
\begin{equation*}
\mathcal{I}_3 = \bigcup_{64 \leq A \leq A_0} \ \bigcup_{\tiny{\substack{(j_1,j) \in J_{A}^{\mathcal{I}_3}\\ 16 \leq |j_1 - j|\leq 32}}}
{\mathfrak{D}}_{j_1}^A \cross {\mathfrak{D}}_{j}^A
\cup \bigcup_{\tiny{\substack{(j_1,j) \in J_{A_0}^{\mathcal{I}_3}\\|j_1 - j|\leq 16}}}
{\mathfrak{D}}_{j_1}^{A_0} \cross {\mathfrak{D}}_{j}^{A_0}.
\end{equation*}
In addition, by the assumption \textnormal{(ii)'} in \textit{Case} \textnormal{\ref{case-2}}, we can assume that $A \geq 2^{25}$.
We calculate that
\begin{align*}
& \textnormal{(LHS) of \eqref{prop3.10-1}} \\
\leq & \sum_{2^{25} \leq A \leq A_0} \sum_{\tiny{\substack{(j_1,j) \in J_{A}^{\mathcal{I}_3}\\ 16 \leq  |j_1 - j|\leq 32}}}
\Bigl|\int_{**}{  \ha{v}_{{N_2, L_2}}(\tau_2, \xi_2, \eta_2)
\ha{u}_{N_1, L_1}|_{\tilde{\mathfrak{D}}_{j_1}^A}(\tau_1, \xi_1, \eta_1)
\ha{w}_{N_0, L_0}|_{\tilde{\mathfrak{D}}_{j}^A}(\tau, \xi, \eta)
}
d\sigma_1 d\sigma \Bigr|\\
& \quad + \sum_{\tiny{\substack{(j_1,j) \in J_{A_0}^{\mathcal{I}_3}\\ |j_1 - j|\leq 16}}}
\Bigl|\int_{**}{  \ha{v}_{{N_2, L_2}}(\tau_2, \xi_2, \eta_2)
\ha{u}_{N_1, L_1}|_{\tilde{\mathfrak{D}}_{j_1}^{A_0}}(\tau_1, \xi_1, \eta_1)
\ha{w}_{N_0, L_0}|_{\tilde{\mathfrak{D}}_{j}^{A_0}}(\tau, \xi, \eta)
}
d\sigma_1 d\sigma \Bigr|\\
& =: \sum_{2^{25} \leq A \leq A_0} \sum_{\tiny{\substack{(j_1,j) \in J_{A}^{\mathcal{I}_3}\\ 16 \leq  |j_1 - j|\leq 32}}}   I_1 + \sum_{\tiny{\substack{(j_1,j) \in J_{A_0}^{\mathcal{I}_3}\\|j_1 - j|\leq 16}}}   I_2.
\end{align*}
By Proposition \ref{prop3.13}, we get
\begin{align*}
& \sum_{\tiny{\substack{(j_1,j) \in J_{A}^{\mathcal{I}_3}\\ 16 \leq  |j_1 - j|\leq 32}}}   I_1 \\ \lesssim &
\sum_{\tiny{\substack{(j_1,j) \in J_{A}^{\mathcal{I}_3}\\ 16 \leq  |j_1 - j|\leq 32}}}
A^{\frac{1}{2}} N_1^{-2} (L_0 L_1 L_2)^{\frac{1}{2}} \|\ha{u}_{N_1, L_1}|_{\tilde{\mathfrak{D}}_{j_1}^A} \|_{L^2}
\| \ha{v}_{N_2, L_2}\|_{L^2}
\|\ha{w}_{{N_0, L_0}}|_{\tilde{\mathfrak{D}}_{j}^A}  \|_{L^2}\\
\lesssim & \
A^{\frac{1}{2}} N_1^{-2} (L_0 L_1 L_2)^{\frac{1}{2}} \|\ha{u}_{N_1, L_1} \|_{L^2}
\| \ha{v}_{N_2, L_2} \|_{L^2}
\|\ha{w}_{{N_0, L_0}} \|_{L^2}.
\end{align*}
Letting $A_0$ be the minimal dyadic number which satisfies
$A_0 \geq L_{0}^{-1/2} N_1^{3/2}$, we obtain
\begin{equation*}
\sum_{2^{25} \leq A \leq A_0} \sum_{\tiny{\substack{(j_1,j) \in J_{A}^{\mathcal{I}_3}\\ 16 \leq  |j_1 - j|\leq 32}}}    I_1
 \lesssim
(N_1)^{-\frac{5}{4}}  L_0^{\frac{1}{4}} (L_1 L_2)^{\frac{1}{2}} \|\ha{u}_{N_1, L_1} \|_{L^2}
\| \ha{v}_{N_2, L_2} \|_{L^2}
\|\ha{w}_{{N_0, L_0}} \|_{L^2}.
\end{equation*}
For the latter term, we utilize the estimate \eqref{bilinearStrichartz-6} in Proposition \ref{prop3.12} as
\begin{align*}
\sum_{\tiny{\substack{(j_1,j) \in J_{A_0}^{\mathcal{I}_3}\\ |j_1 - j|\leq 16}}}  I_2 & \lesssim
 (N_1 A_0)^{-\frac{1}{2}} (L_1 L_2)^{\frac{1}{2}}
\sum_{\tiny{\substack{(j_1,j) \in J_{A_0}^{\mathcal{I}_3}\\ |j_1 - j|\leq 16}}}
\|\ha{u}_{N_1, L_1}|_{\tilde{\mathfrak{D}}_{j_1}^A} \|_{L^2}
\| \ha{v}_{N_2, L_2}\|_{L^2}
\|\ha{w}_{{N_0, L_0}}|_{\tilde{\mathfrak{D}}_{j}^A}  \|_{L^2}\\
& \lesssim (N_1)^{-\frac{5}{4}}  L_0^{\frac{1}{4}} (L_1 L_2)^{\frac{1}{2}} \|\ha{u}_{N_1, L_1} \|_{L^2}
\| \ha{v}_{N_2, L_2} \|_{L^2}
\|\ha{w}_{{N_0, L_0}} \|_{L^2}.
\end{align*}
This completes the proof.
\end{proof}
Next, we consider Case \textnormal{(I \hspace{-0.15cm}I)}, $(\xi_1,\eta_1)\times (\xi,\eta) \subset \mathcal{I}_2$. The following Proposition gives the desired estimate \eqref{desired-est-12-23}.
\begin{prop}\label{prop3.14}
Assume \textnormal{(i)}, \textnormal{(ii)'} in \textit{Case} \textnormal{\ref{case-2}}
and $\min(N_0, N_1, N_2) = N_2$. Then we have
\begin{align}
&
\Bigl|\int_{**}{ |\xi+\eta| \ha{v}_{{N_2, L_2}}(\tau_2, \xi_2, \eta_2) \chi_{\mathcal{I}_2}((\xi_1,\eta_1), (\xi,\eta))
\ha{u}_{N_1, L_1}(\tau_1, \xi_1, \eta_1)
\ha{w}_{N_0, L_0}(\tau, \xi, \eta)
}
d\sigma_1 d\sigma \Bigr| \notag \\
& \qquad \qquad \qquad \qquad
\lesssim  N_1^{-\frac{1}{4}} (L_0 L_1 L_2)^{\frac{1}{2}} \|\ha{u}_{N_1, L_1}\|_{L^2}
\| \ha{v}_{N_2, L_2}\|_{L^2}
\|\ha{w}_{{N_0, L_0}}\|_{L^2},\label{prop3.14-1}
\end{align}
where $d \sigma_1 = d\tau_1 d \xi_1 d \eta_1$, $d \sigma = d\tau d \xi d \eta$
 and $**$ denotes $(\tau_2, \xi_2, \eta_2) = (\tau_1 + \tau, \xi_1+ \xi, \eta_1 + \eta).$
\end{prop}
As we mentioned, the key ingredient to show Proposition \ref{prop3.14} is $|\xi+ \eta|$ in (LHS) of
\eqref{prop3.14-1}. Thus we introduce the decomposition of $\mathcal{I}_{2}$ to determine the size of $|\xi+\eta|$.
\begin{defn}
Let $M \geq 2^{11}$ be dyadic. We define
\begin{align*}
\mathcal{I}_{2}^M &  = \bigl( {\mathfrak{D}}_{{2^{-2}M \times 3}}^{M} \times {\mathfrak{D}}_{2^{-2}M \times 3}^{M} \bigr) \setminus \bigl( {\mathfrak{D}}_{{2^{-1}M \times 3}}^{2M} \times {\mathfrak{D}}_{2^{-1}M \times 3}^{2M} \bigr)  ,\\
\tilde{\mathcal{I}}_2^M & = \bigl( \tilde{{\mathfrak{D}}}_{2^{-2}M \times 3}^{M} \times
\tilde{{\mathfrak{D}}}_{2^{-2}M \times 3}^{M} \bigr) \setminus
\bigl( \tilde{{\mathfrak{D}}}_{2^{-1}M \times 3}^{2M} \times
\tilde{{\mathfrak{D}}}_{2^{-1}M \times 3}^{2M} \bigr).
\end{align*}
It is easy to see that
\begin{equation*}
(r \cos \theta, r \sin \theta) \in \mathcal{I}_{2}^M \ \iff \ M^{-1} \pi \leq \min
\Bigl( \Bigl| \theta-\frac{3 \pi}{4}\Bigr|,
\Bigl| \theta + \frac{\pi}{4} \Bigr| \Bigr) \leq 2 M^{-1} \pi
\end{equation*}
and, for a dyadic number $M_0 \geq 2^{11}$,
\begin{equation*}
\mathcal{I}_{2} = \bigcup_{2^{11} \leq M < M_0} \mathcal{I}_{2}^M \, \cup \,
\bigl( {\mathfrak{D}}_{{2^{-2}M_0 \times 3}}^{M_0} \times {\mathfrak{D}}_{2^{-2}M_0 \times 3}^{M_0} \bigr).
\end{equation*}
\end{defn}
Note that if $(\xi_1, \eta_1) \times (\xi,\eta) \in \mathcal{I}_{2}^M$ then $|\xi+\eta| \lesssim M^{-1} N_1 $.
\begin{prop}\label{prop3.15}
Assume \textnormal{(i)}, \textnormal{(ii)'} in \textit{Case} \textnormal{\ref{case-2}}
and $\min(N_0, N_1, N_2) = N_2.$ Let $M\geq 2^{11}$ and $A \geq 2^{20} M$ be dyadic, $|j_1-j| \leq 32$ and
\begin{equation*}
\left( {\mathfrak{D}}_{j_1}^A \times {\mathfrak{D}}_{j}^A \right) \subset \mathcal{I}_{2}^M.
\end{equation*}
Assume that $||(\xi_1, \eta_1)|- |(\xi_2-\xi_1, \eta_2-\eta_1)|| \geq 2^{-3} N_1$.
Then we have
\begin{align}
& \Bigl\| \chi_{G_{N_2, L_2}} \int \ha{u}_{N_1, L_1}|_{\tilde{\mathfrak{D}}_{j_1}^A} (\tau_1, \xi_1, \eta_1) \ha{w}_{N_0, L_0}|_{\tilde{\mathfrak{D}}_{j}^A} (\tau_2 - \tau_1, \xi_2-\xi_1, \eta_2- \eta_1) d\sigma_1
\Bigr\|_{L_{\xi_2, \eta_2, \tau_2}^2} \notag \\
& \qquad \qquad \qquad \qquad
\lesssim A^{-\frac{1}{2}} M^{\frac{1}{2}} N_1^{-\frac{1}{2}} (L_0 L_1)^{\frac{1}{2}} \|\ha{u}_{N_1, L_1}|_{\tilde{\mathfrak{D}}_{j_1}^A}\|_{L^2}
\|\ha{w}_{N_0, L_0}|_{\tilde{\mathfrak{D}}_{j}^A} \|_{L^2}.\label{bilinearStrichartz-7}
\end{align}
Similarly, if we assume $||(\xi, \eta)|- |(\xi_1+\xi, \eta_1+\eta)|| \geq 2^{-3} N_1$, then we have
\begin{align}
& \Bigl\| \chi_{G_{N_1, L_1} \cap \tilde{\mathfrak{D}}_{j_1}^A} \int \ha{w}_{N_0, L_0}|_{\tilde{\mathfrak{D}}_{j}^A}
(\tau, \xi, \eta)
\ha{v}_{N_2, L_2} (\tau_1+ \tau, \xi_1+\xi, \eta_1+ \eta) d\sigma \Bigr\|_{L_{\xi_1, \eta_1, \tau_1}^2}
\notag \\
& \qquad \qquad \qquad \qquad
\lesssim A^{-\frac{1}{2}} M^{\frac{1}{2}} N_1^{-\frac{1}{2}} (L_0 L_2)^{\frac{1}{2}}
\|\ha{w}_{N_0, L_0}|_{\tilde{\mathfrak{D}}_{j}^A} \|_{L^2}
\|\ha{v}_{N_2, L_2}\|_{L^2}. \label{bilinearStrichartz-8}
\end{align}
If we assume $||(\xi_1, \eta_1)|- |(\xi_1+\xi, \eta_1+\eta)|| \geq 2^{-3} N_1$, then we have
\begin{align}
& \Bigl\| \chi_{G_{N_0,L_0} \cap \tilde{\mathfrak{D}}_{j}^A} \int \ha{v}_{N_2, L_2} (\tau_1+ \tau, \xi_1+\xi, \eta_1+ \eta)  \ha{u}_{N_1, L_1}|_{\tilde{\mathfrak{D}}_{j_1}^A} (\tau_1, \xi_1, \eta_1) d \sigma_1 \Bigr\|_{L_{\xi, \eta, \tau}^2} \notag \\
& \qquad \qquad \qquad \qquad
\lesssim A^{-\frac{1}{2}} M^{\frac{1}{2}} N_1^{-\frac{1}{2}} (L_1 L_2)^{\frac{1}{2}} \|\ha{v}_{N_2, L_2} \|_{L^2}
\|\ha{u}_{N_1, L_1}|_{\tilde{\mathfrak{D}}_{j_1}^A} \|_{L^2} .\label{bilinearStrichartz-9}
\end{align}
\end{prop}
\begin{proof}
We only consider the first estimate \eqref{bilinearStrichartz-7}.
The other two estimates can be obtained in a similar way.
The proof is almost the same as that for \eqref{bilinearStrichartz-6} in Proposition \ref{prop3.12}. Since $\left( {\mathfrak{D}}_{j_1}^A \times {\mathfrak{D}}_{j}^A \right) \subset \mathcal{I}_{2}^M$, without loss of generality, we can assume that $(\xi_1,\eta_1) \in {\mathfrak{D}}_{{2^{-2}M \times 3}}^{M} \setminus {\mathfrak{D}}_{2^{-1}M \times 3}^{2M}$.
As we saw in the proof of
Proposition \ref{prop3.2}, it suffices to show the estimate
\begin{equation}
\sup_{(\tau_2, \xi_2, \eta_2) \in G_{N_2, L_2}}|E(\tau_2, \xi_2, \eta_2)| \lesssim A^{-1}M N_1^{-1} L_1 L_2, \label{est01-prop3.15}
\end{equation}
where
\begin{equation*}
E(\tau_2, \xi_2, \eta_2) = \{ (\tau_1, \xi_1, \eta_1) \in G_{N_1, L_1} \cap \tilde{\mathfrak{D}}_{j_1}^A
 |  (\tau_2-\tau_1, \xi_2- \xi_1, \eta_2-\eta_1) \in G_{N_0,L_0} \cap \tilde{\mathfrak{D}}_{j}^A\}.
\end{equation*}
For fixed $(\xi, \eta)$, it follows from $(\tau_1, \xi_1, \eta_1) \in G_{N_1, L_1}$ and $(\tau_2-\tau_1, \xi_2- \xi_1, \eta_2-\eta_1) \in G_{N_0,L_0}$ that
\begin{equation}
\sup_{(\tau_2, \xi_2, \eta_2) \in G_{N_2, L_2}} | \{ \tau_1 \, | \, (\tau_1, \xi_1, \eta_1) \in E(\tau_2, \xi_2, \eta_2) \}|
\lesssim \min(L_0, L_1).\label{est02-prop3.15}
\end{equation}
Let$(\xi_1,\eta_1) = (r_1 \cos \theta_1, r_1 \sin \theta_1)$ and
$(\xi_2-\xi_1,\eta_2-\eta_1)= (r \cos \theta, r \sin \theta)$. Recall that the assumption
$(\xi_1,\eta_1) \in {\mathfrak{D}}_{{2^{-2}M \times 3}}^{M} \setminus {\mathfrak{D}}_{2^{-1}M \times 3}^{2M}$
implies $|\sin \theta_1 + \cos \theta_1| \geq M^{-1}$.
Furthermore, $|j_1-j| \leq 32$ gives $ |(\cos \theta_1,\sin \theta_1 ) - (\cos \theta, \sin \theta)| \leq 2^{7} A^{-1}$ or $ |(\cos \theta_1,\sin \theta_1 ) + (\cos \theta, \sin \theta)| \leq 2^{7} A^{-1}$.
Thus, we get
\begin{align*}
& |\partial_{r_1}  \Phi(\xi_1,\eta_1,-\xi_2,-\eta_2)|  = |(\cos \theta_1 \partial_{\xi_1} + \sin \theta_1
\partial_{\eta_1})\Phi(\xi_1,\eta_1,-\xi_2,-\eta_2)|\\
 = & \bigl| \cos \theta_1 \bigl( \xi_1 - (\xi_2-\xi_1)  \bigr) \bigl( \xi_1 + (\xi_2-\xi_1) \bigr)
+ \sin \theta_1 \bigl( \eta_1 - (\eta_2-\eta_1) \bigr) \bigl( \eta_1 + (\eta_2-\eta_1) \bigr) \bigr| \\
\geq & \, (r_1+r)|r_1-r| | \cos^3 \theta_1 + \sin^3 \theta_1|-2^{10} A^{-1} N_1^2\\
 \geq & \,  2^{-4} (1- 2^{-1}\sin 2\theta_1)  |\cos \theta_1 + \sin \theta_1| N_1^2 - 2^{10}A^{-1} N_1^2\\
\geq & \,  2^{-8}M^{-1}N_1^2 .
\end{align*}
Here we used
$|r_1-r| = ||(\xi_1, \eta_1)|- |(\xi_2-\xi_1, \eta_2-\eta_1)|| \geq 2^{-3} N_1$
and
$A \geq 2^{20} M$.
It is also observed that
\begin{equation*}
|\Phi(\xi_1,\eta_1,-\xi_2,-\eta_2)+ (\tau_2-\xi_2^3 -\eta_2^3)| \lesssim \max(L_0,L_1).
\end{equation*}
Consequently, for fixed $\theta_1$, we see $r_1$ is confined to a set of measure at most $\max(L_0,L_1)M/N_1^2$ and we get
\begin{align*}
  | \{ (\xi_1, \eta_1) \ | \ (\tau_1, \xi_1, \eta_1) \in E(\tau_2, \xi_2,\eta_2) \} |
= &  \int_{\theta_1} \int_{r_1} {\chi}_{E(\tau_2,\xi_2,\eta_2)} (|\xi_1|, \theta_1) r_1 d r_1 d \theta_1 \\
\lesssim &A^{-1}M N_1^{-1} \max (L_1, L_2).
\end{align*}
This and \eqref{est02-prop3.15} give \eqref{est01-prop3.15}.
\end{proof}
\begin{prop}\label{prop3.16}
Assume \textnormal{(i)}, \textnormal{(ii)'} in \textit{Case} \textnormal{\ref{case-2}}
and $\min(N_0, N_1, N_2) = N_2.$ Let $A$ and $M$ be dyadic numbers such that $2^{11} \leq M \leq N_1$ and $A \geq \max (2^{25}, M)$. Suppose that
$j_1$, $j$ satisfy $16 \leq |j_1-j| \leq 32$ and
\begin{equation*}
\left( {\mathfrak{D}}_{j_1}^A \times {\mathfrak{D}}_{j}^A \right) \subset \mathcal{I}_{2}^M.
\end{equation*}
Then we have
\begin{equation}
\begin{split}
&
\Bigl|\int_{**}{  \ha{v}_{{N_2, L_2}}(\tau_2, \xi_2, \eta_2)
\ha{u}_{N_1, L_1}|_{\tilde{\mathfrak{D}}_{j_1}^A}(\tau_1, \xi_1, \eta_1)
\ha{w}_{N_0, L_0}|_{\tilde{\mathfrak{D}}_{j}^A}(\tau, \xi, \eta)
}
d\sigma_1 d\sigma \Bigr|\\
& \qquad \qquad  \lesssim  (AM)^{\frac{1}{2}} N_1^{-2} (L_0 L_1 L_2)^{\frac{1}{2}} \|\ha{u}_{N_1, L_1}|_{\tilde{\mathfrak{D}}_{j_1}^A} \|_{L^2}
\| \ha{v}_{N_2, L_2}\|_{L^2}
\|\ha{w}_{{N_0, L_0}}|_{\tilde{\mathfrak{D}}_{j}^A}  \|_{L^2},\label{prop3.16-1}
\end{split}
\end{equation}
where $d \sigma_1 = d\tau_1 d \xi_1 d \eta_1$, $d \sigma = d\tau d \xi d \eta$
 and $**$ denotes $(\tau_2, \xi_2, \eta_2) = (\tau_1 + \tau, \xi_1+ \xi, \eta_1 + \eta).$
\end{prop}
\begin{proof}
It is easily observed that if $|\Phi(\xi_1,\eta_1, \xi, \eta)| \gtrsim A^{-1}N_1^3$, then by using the Strichartz estimates \eqref{Strichartz-1} with $p=q=4$ and Proposition \ref{prop3.15}, we get \eqref{prop3.16-1}.
Therefore, we assume that $|\Phi(\xi_1,\eta_1, \xi, \eta)| \leq 2^{-20}A^{-1}N_1^3$ hereafter.

We first observe that $|\Phi(\xi_1,\eta_1, \xi, \eta)| \leq 2^{-20} A^{-1}N_1^3$ implies that
\begin{equation}
| \, |(\xi_1,\eta_1)|-|(\xi, \eta)| \, | \leq 2^{15} A^{-1} M N_1.\label{est01-prop3.16}
\end{equation}
The proof is almost the same as that of Lemma \ref{lemma3.11}. Let $(\xi_1,\eta_1) = (r_1 \cos \theta_1, r_1 \sin \theta_1)$,
$(\xi,\eta) = (r \cos \theta, r \sin \theta)$.
Because $\left( {\mathfrak{D}}_{j_1}^A \times {\mathfrak{D}}_{j}^A \right) \subset \mathcal{I}_{2}^M$, without loss of generality, we assume that $(\xi_1,\eta_1) \in {\mathfrak{D}}_{{2^{-2}M \times 3}}^{M} \setminus {\mathfrak{D}}_{2^{-1}M \times 3}^{2M}$.
Furthermore, it follows from $|j_1-j|\leq 32$ that
\begin{equation*}
 |(\cos \theta_1,\sin \theta_1 ) - (\cos \theta, \sin \theta)| \leq 2^{7} A^{-1}
\ \textnormal{or} \  |(\cos \theta_1,\sin \theta_1 ) + (\cos \theta, \sin \theta)| \leq 2^{7} A^{-1}.
\end{equation*}
The inequality $|(\cos \theta_1,\sin \theta_1 ) - (\cos \theta, \sin \theta)| \leq 2^{7} A^{-1}$ yields
\begin{align*}
 |\Phi(\xi_1,\eta_1, \xi, \eta)|&  = r_1 r |  \cos \theta_1 \cos \theta (r_1 \cos \theta_1 + r \cos \theta) +
\sin \theta_1 \sin \theta (r_1 \sin \theta_1 + r \sin \theta) | \\
 & \geq   r_1  r \, ( r_1+r  ) \, (1- 2^{-1} \sin 2\theta_1)
 |(\cos \theta_1 + \sin \theta_1)| \\
& \geq 2^{-10} M^{-1} N_1^3,
\end{align*}
which contradicts $|\Phi(\xi_1,\eta_1, \xi, \eta)| \leq 2^{-20}A^{-1}N_1^3$. Similarly, for the latter case, we observe
\begin{align*}
& |\Phi(\xi_1,\eta_1, \xi, \eta)| \leq  2^{-20}A^{-1}N_1^3\\
\Longrightarrow & \, r_1  r \, | r_1-r  | \,
 |(\cos \theta_1 + \sin \theta_1) (1- 2^{-1} \sin 2\theta_1)| \leq 2^{10} A^{-1} N_1^3\\
\Longrightarrow & \, | r_1-r | \leq 2^{15} A^{-1} M N_1.
\end{align*}
Thus \eqref{est01-prop3.16} holds true.
We turn to show \eqref{prop3.16-1}. In view of \eqref{est01-prop3.16}, by the almost orthogonality, we may assume that $r_1$ and $r$ are restricted to sets of measure $\sim A^{-1} M N_1$, respectively. Precisely, we can assume
\begin{equation*}
\operatorname{supp} \ha{u}_{N_1, L_1} \subset \mathbb{S}_{A^{-1} M N_1}^{\ell_1}, \qquad
\operatorname{supp} \ha{w}_{N_0, L_0} \subset \mathbb{S}_{A^{-1} M N_1}^{\ell},
\end{equation*}
where $|\ell_1-\ell| \leq 2^{10}$,
\begin{equation*}
\mathbb{S}_\delta^{\ell}= \{ (\tau, \xi,\eta ) \in \R^3 \, | \, N_1+\ell \delta \leq \langle {\xi} \rangle \leq N_1 + (\ell+1)\delta \}.
\end{equation*}
Note that the set $\mathbb{S}_{A^{-1} M N_1}^{\ell_1} \cap {\mathfrak{D}}_{j_1}^A$ is contained in a rectangle whose short side length is $\sim A^{-1} N_1$ and long side length is $\sim A^{-1} M N_1$.
Thus, for $A' = 2^{-100} A$, we can find approximately $M$ number of tiles $\{{\mathcal{T}}_{k_1}^{A'}\}_{k_1}$ such that
\begin{equation}
 \mathbb{S}_{A^{-1} M N_1}^{\ell_1} \cap {\mathfrak{D}}_{j_1}^A\cap {\mathcal{T}}_{k_1}^{A'} \not= \emptyset, \quad \mathbb{S}_{A^{-1} M N_1}^{\ell_1} \cap {\mathfrak{D}}_{j_1}^A \subset \bigcup_{\# k_1 \sim M} {\mathcal{T}}_{k_1}^{A'}.\label{prop3.16-tiles}
\end{equation}
We can find that \eqref{prop3.16-1} reduces to the following estimate.
\begin{equation}
\begin{split}
& \Bigl|\int_{\R^3 \times \R^3} { h (\tau_1+\tau_2, \xi_1+\xi_2, \eta_1+\eta_2)
f(\tau_1, \xi_1, \eta_1)  g(\tau_2, \xi_2, \eta_2)
}
d\sigma_1 d\sigma_2 \Bigr|\\
& \qquad \qquad \qquad \qquad \lesssim A^{\frac{1}{2}} N_1^{-2} (L_0 L_1 L_2)^{\frac{1}{2}} \|f \|_{L^2}
\|g\|_{L^2}
\| h \|_{L^2}.
\end{split}\label{est02-prop3.16}
\end{equation}
Here $f$, $g$, $h$ satisfy
\begin{equation*}
\operatorname{supp} f \subset G_{N_1, L_1} \cap \tilde{\mathcal{T}}_{k_1}^{A'}, \quad
\operatorname{supp} g \subset G_{N_0, L_0} \cap \mathbb{S}_{A^{-1} M N_1}^{\ell} \cap {\tilde{\mathfrak{D}}_{j}^A}, \quad
\operatorname{supp} h \subset G_{N_2, L_2}.
\end{equation*}
Indeed, by \eqref{prop3.16-tiles} and the almost orthogonality, we can see that
\begin{align*}
&
\Bigl|\int_{**}{  \ha{v}_{{N_2, L_2}}(\tau_2, \xi_2, \eta_2)
\ha{u}_{N_1, L_1}|_{\tilde{\mathfrak{D}}_{j_1}^A \cap  \mathbb{S}_{A^{-1} M N_1}^{\ell_1}}(\tau_1, \xi_1, \eta_1)
\ha{w}_{N_0, L_0}|_{\tilde{\mathfrak{D}}_{j}^A \cap  \mathbb{S}_{A^{-1} M N_1}^{\ell}}(\tau, \xi, \eta)
}
d\sigma_1 d\sigma \Bigr|\\
\lesssim & \sum_{\# k_1 \sim M}
\Bigl|\int_{**}{  \ha{v}_{{N_2, L_2}}(\tau_2, \xi_2, \eta_2)
\ha{u}_{N_1, L_1}|_{ \tilde{\mathcal{T}}_{k_1}^{A'}}(\tau_1, \xi_1, \eta_1)
\ha{w}_{N_0, L_0}|_{\tilde{\mathfrak{D}}_{j}^A \cap  \mathbb{S}_{A^{-1} M N_1 }^{\ell}}(\tau, \xi, \eta)
}
d\sigma_1 d\sigma \Bigr|\\
 \underset{\eqref{est02-prop3.16}}{\lesssim} &
A^{\frac{1}{2}} N_1^{-2} (L_0 L_1 L_2)^{\frac{1}{2}}
\sum_{\# k_1 \sim M} \|\ha{u}_{N_1, L_1}|_{ \tilde{\mathcal{T}}_{k_1}^{A'}} \|_{L^2}
\|\ha{v}_{N_2, L_2}\|_{L^2}
\| \ha{w}_{N_0, L_0}|_{\tilde{\mathfrak{D}}_{j}^A \cap  \mathbb{S}_{A^{-1} M N_1 }^{\ell}} \|_{L^2}\\
\lesssim &   (AM)^{\frac{1}{2}} N_1^{-2} (L_0 L_1 L_2)^{\frac{1}{2}}
\|\ha{u}_{N_1, L_1}|_{\tilde{\mathfrak{D}}_{j_1}^A \cap  \mathbb{S}_{A^{-1} M N_1}^{\ell_1}}\|_{L^2}
\| \ha{v}_{N_2, L_2}\|_{L^2}
\|\ha{w}_{N_0, L_0}|_{\tilde{\mathfrak{D}}_{j}^A \cap  \mathbb{S}_{A^{-1} M N_1}^{\ell}} \|_{L^2}.
\end{align*}
Since the proof of \eqref{est02-prop3.16} is the same as that for \eqref{est01-prop3.11}, we omit the details.
\end{proof}
\begin{proof}[Proof of Proposition \ref{prop3.14}]
Let $A_0 \geq 2^{25}$ and $M_0 \geq 2^{11}$ be dyadic which will be chosen later. Define that
\begin{equation*}
J_{A}^{\mathcal{I}_2^M} = \{ (j_1, j) \, | \, 0 \leq j_1,j \leq A -1, \ \left( {\mathfrak{D}}_{j_1}^A \times {\mathfrak{D}}_{j}^A \right) \subset \mathcal{I}_2^M\}.
\end{equation*}
We introduce the angular decomposition of $\mathcal{I}_{2}^M$ which is defined by
\begin{equation*}
\mathcal{I}_2^M =   \bigcup_{64 \leq A \leq A_0} \ \bigcup_{\tiny{\substack{(j_1,j) \in J_{A}^{\mathcal{I}_2^M}\\ 16 \leq |j_1 - j|\leq 32}}}
{\mathfrak{D}}_{j_1}^A \cross {\mathfrak{D}}_{j}^A \cup
\bigcup_{\tiny{\substack{(j_1,j) \in J_{A_0}^{\mathcal{I}_2^M}\\|j_1 - j|\leq 16}}}
{\mathfrak{D}}_{j_1}^{A_0} \cross {\mathfrak{D}}_{j}^{A_0}.
\end{equation*}
Recall that
\begin{equation*}
\mathcal{I}_{2} = \bigcup_{2^{11} \leq M < M_0} \mathcal{I}_{2}^M \, \cup \,
\left( {\mathfrak{D}}_{{2^{-2}M_0 \times 3}}^{M_0} \times {\mathfrak{D}}_{2^{-2}M_0 \times 3}^{M_0} \right).
\end{equation*}
Hence, by using the following notation for simplicity
\begin{equation*}
I_A = \Bigl|\int_{**}{  \ha{v}_{{N_2, L_2}}(\tau_2, \xi_2, \eta_2)
\ha{u}_{N_1, L_1}|_{\tilde{\mathfrak{D}}_{j_1}^A}(\tau_1, \xi_1, \eta_1)
\ha{w}_{N_0, L_0}|_{\tilde{\mathfrak{D}}_{j}^A}(\tau, \xi, \eta)
}
d\sigma_1 d\sigma \Bigr|,
\end{equation*}
we have
\begin{align*}
& \textnormal{(LHS) of \eqref{prop3.14-1}} \\
\lesssim & \sum_{2^{11} \leq M \leq M_0}  \sum_{2^{25} \leq A \leq A_0}
\sum_{\tiny{\substack{(j_1,j) \in J_{A}^{\mathcal{I}_2^M}\\ 16 \leq |j_1 - j|\leq 32}}} N_1 M^{-1} I_A
+ \sum_{2^{11} \leq M \leq M_0} \sum_{\tiny{\substack{(j_1,j) \in J_{A_0}^{\mathcal{I}_2^M}\\|j_1 - j|\leq 16}}}
N_1 M^{-1}I_{A_0}  \\
 + & N_1 M_0^{-1}
\Bigl|\int_{**}{  \ha{v}_{{N_2, L_2}}(\tau_2, \xi_2, \eta_2)
\ha{u}_{N_1, L_1}|_{\tilde{\mathfrak{D}}_{{2^{-2}M_0 \times 3}}^{M_0} }(\tau_1, \xi_1, \eta_1)
\ha{w}_{N_0, L_0}|_{\tilde{\mathfrak{D}}_{2^{-2}M_0 \times 3}^{M_0}}(\tau, \xi, \eta)
}
d\sigma_1 d\sigma \Bigr|.
\end{align*}
For the first term, by Proposition \ref{prop3.16}, we get
\begin{align*}
& \sum_{2^{25} \leq A \leq A_0}
\sum_{\tiny{\substack{(j_1,j) \in J_{A}^{\mathcal{I}_2^M}\\ 16 \leq |j_1 - j|\leq 32}}} N_1 M^{-1} I_A \\
\lesssim & \sum_{2^{25} \leq A \leq A_0}
\sum_{\tiny{\substack{(j_1,j) \in J_{A}^{\mathcal{I}_2^M}\\ 16 \leq |j_1 - j|\leq 32}}}
A^{\frac{1}{2}} M^{-\frac{1}{2}}N_1^{-1} (L_0 L_1 L_2)^{\frac{1}{2}} \|\ha{u}_{N_1, L_1}|_{\tilde{\mathfrak{D}}_{j_1}^A} \|_{L^2}
\| \ha{v}_{N_2, L_2}\|_{L^2}
\|\ha{w}_{{N_0, L_0}}|_{\tilde{\mathfrak{D}}_{j}^A} \|_{L^2}\\
\lesssim & \ {A_0}^{\frac{1}{2}} M^{-\frac{1}{2}}N_1^{-1} (L_0 L_1 L_2)^{\frac{1}{2}} \|\ha{u}_{N_1, L_1} \|_{L^2}
\| \ha{v}_{N_2, L_2}\|_{L^2}
\|\ha{w}_{{N_0, L_0}} \|_{L^2}.
\end{align*}
Thus, if we choose $A_0$ as the minimal dyadic number which is greater than $N_1^{3/2}$, the first term is bounded by (RHS) of \eqref{prop3.14-1}. Next, it is easily observed that
\begin{equation*}
\max(| \, |(\xi_1,\eta_1)| - |(\xi,  \eta)| \, |, \,
| \, |(\xi_1,\eta_1)| - |(\xi_2, \eta_2)| \, |, \,
| \, |(\xi,\eta)| - |(\xi_2, \eta_2)| \, |) \geq 2^{-3} N_1,
\end{equation*}
where $(\xi_2,\eta_2)=(\xi_1+\xi, \eta_1+\eta)$ and $(\xi_1,\eta_1)\times (\xi,\eta) \in
{\mathfrak{D}}_{j_1}^{A_0} \cross {\mathfrak{D}}_{j}^{A_0}$ with $|j_1-j| \leq 16$.
This implies that we can use one of \eqref{bilinearStrichartz-7}-\eqref{bilinearStrichartz-9}
in Proposition \ref{prop3.15} and we have
\begin{align*}
& \sum_{2^{11} \leq M \leq M_0} \sum_{\tiny{\substack{(j_1,j) \in J_{A_0}^{\mathcal{I}_2^M}\\|j_1 - j|\leq 16}}}
N_1 M^{-1}I_{A_0}  \\
\lesssim &\sum_{2^{11} \leq M \leq M_0} \sum_{\tiny{\substack{(j_1,j) \in J_{A_0}^{\mathcal{I}_2^M}\\|j_1 - j|\leq 16}}}
{A_0}^{-\frac{1}{2}} M^{-\frac{1}{2}}N_1^{\frac{1}{2}} (L_0 L_1 L_2)^{\frac{1}{2}} \|\ha{u}_{N_1, L_1}|_{\tilde{\mathfrak{D}}_{j_1}^{A_0}} \|_{L^2}
\| \ha{v}_{N_2, L_2}\|_{L^2}
\|\ha{w}_{{N_0, L_0}}|_{\tilde{\mathfrak{D}}_{j}^{A_0}} \|_{L^2}\\
\lesssim & \ \textnormal{(RHS) of \eqref{prop3.14-1}}.
\end{align*}
We now set $M_0 = N_1$. By the Strichartz estimate \eqref{Strichartz-1} with $p=q=4$, we can easily confirm that the last term is also bounded by \textnormal{(RHS) of \eqref{prop3.14-1}}.
\end{proof}
Lastly, we consider the case (I) $(\xi_1,\eta_1)\times (\xi,\eta) \subset \mathcal{I}_1$. By symmetry of $(\xi_1,\xi)$ and $(\eta_1,\eta)$, it suffices to show the following estimate.
\begin{prop}\label{prop3.17}
Assume \textnormal{(i)}, \textnormal{(ii)'} in \textit{Case} \textnormal{\ref{case-2}}
and $\min(N_0, N_1, N_2) = N_2$.
\begin{equation}
\begin{split}
&
\Bigl|\int_{**}{  \ha{v}_{{N_2, L_2}}(\tau_2, \xi_2, \eta_2)
\ha{u}_{N_1, L_1}|_{\tilde{\mathfrak{D}}_{0}^{2^{11}}} (\tau_1, \xi_1, \eta_1)
\ha{w}_{N_0, L_0}|_{\tilde{\mathfrak{D}}_{0}^{2^{11}}} (\tau, \xi, \eta)
}
d\sigma_1 d\sigma \Bigr|\\
& \qquad \qquad \qquad \qquad
\lesssim  N_1^{-1} N_2^{-\frac{1}{4}} L_0^{\frac{1}{4}} (L_1 L_2)^{\frac{1}{2}} \|\ha{u}_{N_1, L_1}\|_{L^2}
\| \ha{v}_{N_2, L_2}\|_{L^2}
\|\ha{w}_{{N_0, L_0}}\|_{L^2},\label{prop3.17-1}
\end{split}
\end{equation}
where $d \sigma_1 = d\tau_1 d \xi_1 d \eta_1$, $d \sigma = d\tau d \xi d \eta$
 and $**$ denotes $(\tau_2, \xi_2, \eta_2) = (\tau_1 + \tau, \xi_1+ \xi, \eta_1 + \eta).$
\end{prop}
\begin{defn}
Let $A \geq 2^{25}$ and $K$ be dyadic which satisfy $2^{10} \leq K \leq 2^{-10}A$. We define that
\begin{align*}
\mathfrak{J}_A^K & = \Bigl\{ j \in \N \ | \ \frac{A}{K} \leq j \leq 2 \frac{A}{K}, \quad A- 2 \frac{A}{K} \leq j \leq A-  \frac{A}{K} \Bigr\},\\
\mathfrak{J}_A & = \left\{ j \in \N \ | \ 0 \leq j \leq 2^{10}, \quad A- 2^{10}  \leq j \leq A- 1 \right\}.
\end{align*}
\end{defn}
\begin{prop}\label{prop3.18}
Assume \textnormal{(i)}, \textnormal{(ii)'} in \textit{Case} \textnormal{\ref{case-2}}
and $\min(N_0, N_1, N_2) = N_2.$ Let $A \geq 2^{25}$ be dyadic, $|j_1-j| \leq 32$ and
\begin{equation*}
\left( {\mathfrak{D}}_{j_1}^A \times {\mathfrak{D}}_{j}^A \right) \subset \mathcal{I}_1.
\end{equation*}
Then we have
\begin{align}
\begin{split}\label{bilinearStrichartz-11}
& \Bigl\| \chi_{G_{N_1, L_1} \cap \tilde{\mathfrak{D}}_{j_1}^A} \int \ha{w}_{N_0, L_0}|_{\tilde{\mathfrak{D}}_{j}^A}
(\tau, \xi, \eta)
\ha{v}_{N_2, L_2} (\tau_1+ \tau, \xi_1+\xi, \eta_1+ \eta) d\sigma \Bigr\|_{L_{\xi_1, \eta_1, \tau_1}^2} \\
& \qquad \qquad \qquad \qquad \qquad
\lesssim (A N_1 )^{-\frac{1}{2}} (L_0 L_2)^{\frac{1}{2}}
\|\ha{w}_{N_0, L_0}|_{\tilde{\mathfrak{D}}_{j}^A} \|_{L^2}
\|\ha{v}_{N_2, L_2}\|_{L^2},
\end{split}\\
\begin{split}\label{bilinearStrichartz-12}
& \Bigl\| \chi_{G_{N_0,L_0} \cap \tilde{\mathfrak{D}}_{j}^A} \int \ha{v}_{N_2, L_2} (\tau_1+ \tau, \xi_1+\xi, \eta_1+ \eta)  \ha{u}_{N_1, L_1}|_{\tilde{\mathfrak{D}}_{j_1}^A} (\tau_1, \xi_1, \eta_1) d \sigma_1 \Bigr\|_{L_{\xi, \eta, \tau}^2} \\
& \qquad \qquad \qquad \qquad \qquad
\lesssim (A N_1 )^{-\frac{1}{2}} (L_1 L_2)^{\frac{1}{2}} \|\ha{v}_{N_2, L_2} \|_{L^2}
\|\ha{u}_{N_1, L_1}|_{\tilde{\mathfrak{D}}_{j_1}^A} \|_{L^2} .
\end{split}
\end{align}
In addition to the above assumptions,\\
\textnormal{(1)} assume $j_1 \in \mathfrak{J}_A^K$, then we have
\begin{align}
& \Bigl\| \chi_{G_{N_2, L_2}} \int \ha{u}_{N_1, L_1}|_{\tilde{\mathfrak{D}}_{j_1}^A} (\tau_1, \xi_1, \eta_1) \ha{w}_{N_0, L_0}|_{\tilde{\mathfrak{D}}_{j}^A} (\tau_2 - \tau_1, \xi_2-\xi_1, \eta_2- \eta_1) d\sigma_1
\Bigr\|_{L_{\xi_2, \eta_2, \tau_2}^2} \notag \\
& \qquad \qquad \qquad \qquad \qquad
\lesssim K^{\frac{1}{4}} N_1^{-\frac{1}{2}} (L_0 L_1)^{\frac{1}{2}} \|\ha{u}_{N_1, L_1}|_{\tilde{\mathfrak{D}}_{j_1}^A}\|_{L^2}
\|\ha{w}_{N_0, L_0}|_{\tilde{\mathfrak{D}}_{j}^A} \|_{L^2}.\label{bilinearStrichartz-13}
\end{align}
\textnormal{(2)} Assume $j_1 \in \mathfrak{J}_A$ and $16 \leq |j_1-j| \leq 32$, then we have
\begin{align}
& \Bigl\| \chi_{G_{N_2, L_2}} \int \ha{u}_{N_1, L_1}|_{\tilde{\mathfrak{D}}_{j_1}^A} (\tau_1, \xi_1, \eta_1) \ha{w}_{N_0, L_0}|_{\tilde{\mathfrak{D}}_{j}^A} (\tau_2 - \tau_1, \xi_2-\xi_1, \eta_2- \eta_1) d\sigma_1
\Bigr\|_{L_{\xi_2, \eta_2, \tau_2}^2} \notag \\
& \qquad \qquad \qquad \qquad \qquad
\lesssim A^{\frac{1}{4}} N_1^{-\frac{1}{2}} (L_0 L_1)^{\frac{1}{2}} \|\ha{u}_{N_1, L_1}|_{\tilde{\mathfrak{D}}_{j_1}^A}\|_{L^2}
\|\ha{w}_{N_0, L_0}|_{\tilde{\mathfrak{D}}_{j}^A} \|_{L^2}.\label{bilinearStrichartz-14}
\end{align}
\end{prop}
\begin{proof}
First we recall that $N_2 \leq 2^{-14} N_1$ holds by Lemma \ref{lemma3.11}.
\eqref{bilinearStrichartz-11} and \eqref{bilinearStrichartz-12} are obtained by the same proof as that for
\eqref{bilinearStrichartz-5} and \eqref{bilinearStrichartz-6} in Proposition \ref{prop3.12}, respectively.
Thus we omit the proofs of them.
For \eqref{bilinearStrichartz-13}, since $j_1 \in \mathfrak{J}_A^K$ and $N_2 \leq 2^{-14} N_1$, it is easy to see
\begin{equation*}
|\xi_1| \sim |\xi_2- \xi_1| \sim N_1, \quad |\eta_1| \sim |\eta_2-\eta_1| \sim K^{-1} N_1
\end{equation*}
in (LHS) of \eqref{bilinearStrichartz-13}. Thus the Strichartz estimates \eqref{Strichartz-1} with
$p=q=4$ yield the claim.
Lastly, we prove \eqref{bilinearStrichartz-14}.
We deduce from $16 \leq |j_1-j|\leq 32$ that we may assume $|\eta_1| \geq A^{-1}N_1$ in (LHS) of \eqref{bilinearStrichartz-14} without loss of generality. If $|\eta_2-\eta_1| \sim A^{-1} N_1$, then by utilizing the Strichartz estimates \eqref{Strichartz-1} again, we get \eqref{bilinearStrichartz-14}.
Thus we assume that $|\eta_2-\eta_1| \leq 2^{-10} A^{-1} N_1$ in (LHS) of \eqref{bilinearStrichartz-14}.
Recall that it suffices to show the following estimate.
\begin{equation}
\sup_{(\tau_2, \xi_2, \eta_2) \in G_{N_2, L_2}}|E(\tau_2, \xi_2, \eta_2)| \lesssim A^{\frac{1}{2}} N_1^{-1} L_0 L_1, \label{est01-prop3.18}
\end{equation}
where
\begin{equation*}
E(\tau_2, \xi_2, \eta_2) := \{ (\tau_1, \xi_1, \eta_1) \in G_{N_1, L_1} \cap \tilde{\mathfrak{D}}_{j_1}^A
\, | \, (\tau_2-\tau_1, \xi_2- \xi_1, \eta_2-\eta_1) \in G_{N_0,L_0} \cap  \tilde{\mathfrak{D}}_{j}^A \}.
\end{equation*}
We decompose the proof of \eqref{est01-prop3.18} into the two cases, $|\xi_2| \leq  A^{-3/2}N_1$ and $|\xi_2| \geq  A^{-3/2}N_1$.\\
\underline{Case $|\xi_2| \leq  A^{-3/2} N_1$}

By the almost orthogonality, we may assume that $\xi_1$ is restricted to a set whose measure is
$\sim N_1 A^{-3/2}$. For $(\tau_1, \xi_1,\eta_1) \in E(\tau_2, \xi_2, \eta_2) $, a simple calculation yields
\begin{align*}
|3\Phi(\xi_1,\eta_1,-\xi_2,-\eta_2) & + \tau_2 -\xi_2^3 - \eta_2^3|  \lesssim \max (L_0,L_1),\\
 |\partial_{\eta_1}\Phi(\xi_1,\eta_1,-\xi_2,-\eta_2)| & =
|(\eta_1 + (\eta_2-\eta_1) )\, (\eta_1- (\eta_2-\eta_1))|\\
& \geq (|\eta_1| - |\eta_2-\eta_1| )^2 \gtrsim  A^{-2} N_1^2.
\end{align*}
For fixed $\xi_1$, the above inequalities imply that $\eta_1$ is confined to a set of measure
$\sim A^2 N_1^{-2} \max(L_0,L_1)$. Thus we have
\[
\int_{\xi_1} \int_{\eta_1} \int_{\tau_1} {\chi}_{E(\tau_2,\xi_2,\eta_2)} (\tau_1, \xi_1, \eta_1) d \tau_1  d \eta_1 d \xi_1
\lesssim  A^{\frac{1}{2}} N_1^{-1} (L_0 L_1)^{\frac{1}{2}}.
\]
\underline{Case $|\xi_2| \geq  A^{-3/2}N_1$}

In this case, we can observe that
\begin{equation*}
 |\partial_{\xi_1}\Phi(\xi_1,\eta_1,-\xi_2,-\eta_2)|  =
|\xi_2 \, (2\xi_1- \xi_2)| \geq N_1^2 A^{-\frac{3}{2}}.
\end{equation*}
This and $(\xi_1,\eta_1) \in {\mathfrak{D}}_{j_1}^A$ with $j_1 \in \mathfrak{J}_A$ give \eqref{est01-prop3.18} in the same manner as above.
\end{proof}
Next we show the crucial estimate under the conditions
$(\xi_1,\eta_1) \times (\xi,\eta) \in \mathfrak{D}_{j_1}^A \times \mathfrak{D}_{j}^A$ where
$j_1 \in \mathfrak{J}_A$ and $16 \leq |j_1-j| \leq 32$.
\begin{prop}\label{prop3.19}
Assume \textnormal{(i)}, \textnormal{(ii)'} in \textit{Case} \textnormal{\ref{case-2}}
and $\min(N_0, N_1, N_2) = N_2.$ Let $A \geq 2^{25}$ be dyadic,
$j_1 \in \mathfrak{J}_A^K$, $j$ satisfy $16 \leq |j_1-j| \leq 32$.
Then we have
\begin{equation}
\begin{split}
&
\Bigl|\int_{**}{  \ha{v}_{{N_2, L_2}}(\tau_2, \xi_2, \eta_2)
\ha{u}_{N_1, L_1}|_{\tilde{\mathfrak{D}}_{j_1}^A}(\tau_1, \xi_1, \eta_1)
\ha{w}_{N_0, L_0}|_{\tilde{\mathfrak{D}}_{j}^A}(\tau, \xi, \eta)
}
d\sigma_1 d\sigma \Bigr|\\
& \qquad \qquad \lesssim  (AK)^{\frac{1}{2}} N_1^{-2} (L_0 L_1 L_2)^{\frac{1}{2}} \|\ha{u}_{N_1, L_1}|_{\tilde{\mathfrak{D}}_{j_1}^A} \|_{L^2}
\| \ha{v}_{N_2, L_2}\|_{L^2}
\|\ha{w}_{{N_0, L_0}}|_{\tilde{\mathfrak{D}}_{j}^A}  \|_{L^2},\label{prop3.19-1}
\end{split}
\end{equation}
where $d \sigma_1 = d\tau_1 d \xi_1 d \eta_1$, $d \sigma = d\tau d \xi d \eta$
 and $**$ denotes $(\tau_2, \xi_2, \eta_2) = (\tau_1 + \tau, \xi_1+ \xi, \eta_1 + \eta).$
\end{prop}
\begin{proof}
We decompose the proof into the two cases, $|\xi_2| \geq  2^{10}A^{-1}K^{-1/2}N_1$ and
$|\xi_2| \leq 2^{10}A^{-1}K^{-1/2}N_1$.\\
\underline{Case $|\xi_2| \geq 2^{10}A^{-1}K^{-1/2}N_1$}

We show $|\Phi(\xi_1, \eta_1, \xi, \eta)| \gtrsim A^{-1} K^{-1/2} N_1^{3}$. This, combined with Proposition \ref{prop3.18}, immediately yields \eqref{prop3.19-1}. We write $(\xi_1,\eta_1) = (r_1 \cos \theta_1, r_1 \sin \theta_1)$,
$(\xi,\eta) = (r \cos \theta, r \sin \theta)$.
Since $j_1 \in \mathfrak{J}_A^K$ and $16 \leq |j_1-j| \leq 32$,
we may assume that $|\sin \theta_1| \sim K^{-1}$,
$|\cos \theta_1| \geq 4/5$,
$|\cos \theta_1 + \cos \theta| \leq 2^5 A^{-1} K^{-1}$ and $|\sin \theta_1 + \sin \theta| \leq 2^5 A^{-1}$.
We observe that
\begin{align*}
|\Phi(\xi_1, \eta_1, \xi, \eta)| & = |\xi_1 \xi (\xi_1 +\xi) + \eta_1 \eta (\eta_1+ \eta)|\\
& \geq 2^{-1} r_1 r |r_1-r| - 2^{10} A^{-1} K^{-1} N_1^3.
\end{align*}
Thus, it suffices to show $|\xi_1+\xi| \leq 2^{-2} |r_1-r|$. A simple calculation gives
\begin{equation*}
|\xi_1 + \xi| = |r_1 \cos \theta_1 + r \cos \theta| \leq \frac{4}{5} |r_1 - r|  + 2^5 A^{-1} K^{-1} N_1.
\end{equation*}
This and $|\xi_1+\xi| \geq 2^{10}A^{-1}K^{-1/2}N_1$ yield $|\xi_1+\xi| \leq 2^{-2} |r_1-r|$.\\
\underline{Case $|\xi_2| \leq 2^{10}A^{-1}K^{-1/2}N_1$}

Similarly to the proofs of Propositions \ref{prop3.4}, \ref{prop3.13}, \ref{prop3.16}, we will show \eqref{prop3.19-1}
by the nonlinear Loomis-Whitney inequality.
It is clear that the assumptions $|\xi_2| \leq 2^{10}A^{-1}K^{-1/2}N_1$ and $(\xi_1,\eta_1) \times (\xi, \eta) \in \mathfrak{D}_{j_1}^A \times \mathfrak{D}_{j}^A$ imply that $(\xi_2,\eta_2)$ is restricted to a rectangle set whose short side is parallel to $\xi$-axis and its length is $\sim A^{-1}K^{-1/2}N_1$, long side length is $\sim A^{-1}N_1 $. Thus, by the almost orthogonality, it suffices to prove the following estimate.
\begin{equation}
\begin{split}
& \Bigl|\int_{\R^3 \times \R^3} { h (\tau_1+\tau_2, \xi_1+\xi_2, \eta_1+\eta_2)
f(\tau_1, \xi_1, \eta_1)  g(\tau_2, \xi_2, \eta_2)
}
d\sigma_1 d\sigma_2 \Bigr|\\
& \qquad \qquad \qquad \qquad \lesssim (AK)^{\frac{1}{2}} N_1^{-2} (L_0 L_1 L_2)^{\frac{1}{2}} \|f \|_{L^2}
\|g\|_{L^2}
\| h \|_{L^2}.
\end{split}\label{est01-prop3.19}
\end{equation}
Here $f$, $g$, $h$ satisfy
\begin{equation*}
\operatorname{supp} f \subset G_{N_1, L_1} \cap \tilde{\mathfrak{R}}_1^{A}, \quad
\operatorname{supp} g \subset G_{N_0, L_0} \cap \tilde{\mathfrak{R}}_2^{A}, \quad
\operatorname{supp} h \subset G_{N_2, L_2} \cap \tilde{\mathfrak{R}}_3^{A},
\end{equation*}
where $\tilde{\mathfrak{R}}_1^{A}$, $\tilde{\mathfrak{R}}_2^{A}$, $\tilde{\mathfrak{R}}_3^{A}$ are prisms which are defined by using $\alpha_i$, $\beta_i \in \R$ $(i=1,2,3)$ as follows:
\begin{align*}
\tilde{\mathfrak{R}}_1^{A} & = \R \times \mathfrak{R}_1^{A}, \quad
\tilde{\mathfrak{R}}_2^{A} = \R \times \mathfrak{R}_2^{A}, \quad
\tilde{\mathfrak{R}}_3^{A} = \R \times \mathfrak{R}_3^{A},\\
\mathfrak{R}_1^{A} & = \{ (\xi,\eta) \,
|\, \alpha_1\leq \xi \leq \alpha_1+2^{-30}A^{-1}K^{-\frac{1}{2}}N_1 , \
\beta_1 \leq \eta \leq \beta_1 + 2^{-30}A^{-1}N_1\},\\
\mathfrak{R}_2^{A} & = \{ (\xi,\eta)  \,
|\, \alpha_2 \leq  \xi \leq \alpha_2+2^{-30}A^{-1}K^{-\frac{1}{2}}N_1 , \
\beta_2 \leq \eta \leq \beta_2 + 2^{-30}A^{-1} N_1\},\\
\mathfrak{R}_3^{A} & = \{ (\xi,\eta)  \,
|\, \alpha_3 \leq  \xi \leq \alpha_3+2^{-30}A^{-1}K^{-\frac{1}{2}}N_1 , \
\beta_3 \leq \eta \leq \beta_3 + 2^{-30}A^{-1} N_1\}.
\end{align*}
Here we choose $\alpha_i$, $\beta_i$ to satisfy
\begin{align*}
& |\alpha_3| \leq 2^{10}A^{-1}K^{-\frac{1}{2}}N_1, \quad |\beta_3| \sim K^{-1} N_1,\\
(\mathfrak{R}_1^{A}  + \mathfrak{R}_2^{A}& )  \cap \mathfrak{R}_3^{A} \not= \emptyset, \quad
\mathfrak{R}_1^{A} \cap \mathfrak{D}_{j_1}^A \not= \emptyset, \quad
\mathfrak{R}_2^{A} \cap \mathfrak{D}_{j}^A \not= \emptyset.
\end{align*}
By the same argument as in the proof of Proposition \ref{prop3.14}, \eqref{est01-prop3.19} reduces to the equation
\begin{equation}
\| \tilde{f} |_{S_1} * \tilde{g} |_{S_2} \|_{L^2(S_3)} \lesssim (AK)^{\frac{1}{2}}
\| \tilde{f} \|_{L^2(S_1)} \| \tilde{g} \|_{L^2(S_2)},\label{est02-prop3.19}
\end{equation}
where the notations $\tilde{f}$, $\tilde{g}$, $S_1$, $S_2$, $S_3$ are the same as in the proof of Proposition \ref{prop3.4}:
\begin{align*}
& \tilde{f}  (\tau_1, \xi_1 , \eta_1)  = f (N_1^3 \tau_1 , N_1 \xi_1, N_1 \eta_1), \\
& \tilde{g} (\tau_2, \xi_2, \eta_2)  = g (N_1^3 \tau_2, N_1 \xi_2, N_1 \eta_2), \\
& S_1  = \{ \phi_{\tilde{c_1}} (\xi, \eta) = (\xi^3 + \eta^3 + \tilde{c_1}, \, \xi, \, \eta) \in \R^3 \ | \
(N_1 \xi, N_1\eta) \in \mathfrak{R}_1^{A} \}, \\
& S_2  = \{ \phi_{\tilde{c_2}} (\xi, \eta) = (\xi^3 + \eta^3 + \tilde{c_2}, \, \xi, \, \eta) \in \R^3 \ | \ (N_1 \xi, N_1 \eta) \in \mathfrak{R}_2^{A} \},\\
& S_3  = \Bigl\{ (\psi (\xi,\eta), \xi, \eta) \in \R^3  \ | \  (N_1\xi, N_1 \eta) \in \mathfrak{R}_3^{A},
\  \psi (\xi,\eta) =  \xi^3 + \eta^3 + \frac{c_0'}{N_1^{3}} \Bigr\}.
\end{align*}
Note that \eqref{est02-prop3.19} can not be shown
by simply applying the nonlinear Loomis-Whitney inequality. This is because the hypersurfaces $S_1$, $S_2$, $S_3$ do not satisfy the necessary  diameter condition to apply the nonlinear Loomis-Whitney inequality.
To be more specific, the transversality of $S_1$, $S_2$, $S_3$ is comparable to $(A K)^{-1}$.
On the other hand, the diameters of them are comparable to $A^{-1}$.
To overcome this, we perform another scaling $(\tau,\xi,\eta) \to ( \tau , \,  \xi, \, K^{1/2} \eta)$ to define
\begin{align*}
 \tilde{f}_K (\tau_1, \xi_1 , \eta_1) & = \tilde{f} ( \tau_1 ,  \xi_1, K^{\frac{1}{2}} \eta_1), \\
 \tilde{g}_K (\tau_2, \xi_2, \eta_2) & = \tilde{g} ( \tau_2,  \xi_2, K^{\frac{1}{2}} \eta_2), \\
 \tilde{h}_K (\tau, \xi, \eta) & = \tilde{h} ( \tau,  \xi, K^{\frac{1}{2}} \eta).
\end{align*}
By using the above functions, \eqref{est02-prop3.19} can be rewritten as
\begin{equation}
\| \tilde{f}_K |_{S_1^K} * \tilde{g}_K |_{S_2^K} \|_{L^2(S_3^K)} \lesssim A^{\frac{1}{2}} K^{\frac{1}{4}}
\| \tilde{f}_K \|_{L^2(S_1^K)} \| \tilde{g}_K \|_{L^2(S_2^K)},\label{est03-prop3.19}
\end{equation}
where
\begin{align*}
& S_1^K  = \{ \phi_{\tilde{c_1}}^K (\xi, \eta) = (\xi^3 + K^{\frac{3}{2}}\eta^3 + \tilde{c_1}, \, \xi, \, \eta)  \ | \
(N_1 \xi, K^{\frac{1}{2}}N_1\eta) \in \mathfrak{R}_1^{A} \}, \\
& S_2^K  = \{ \phi_{\tilde{c_2}}^K (\xi, \eta) = (\xi^3 + K^{\frac{3}{2}} \eta^3 + \tilde{c_2}, \, \xi, \, \eta) \ | \ (N_1 \xi, K^{\frac{1}{2}} N_1 \eta) \in \mathfrak{R}_2^{A} \},\\
& S_3^K  = \Bigl\{ (\psi^K (\xi,\eta), \xi, \eta)   \ | \  (N_1\xi, K^{\frac{1}{2}} N_1 \eta) \in \mathfrak{R}_3^{A},
\  \psi^K (\xi,\eta) =  \xi^3 + K^{\frac{3}{2}} \eta^3 + \frac{c_0'}{N_1^{3}} \Bigr\}.
\end{align*}
We now show that $S_1^K$, $S_2^K$, $S_3^K$ satisfy the necessary conditions to use the nonlinear Loomis-Whitney inequality.
We see that $(N_1 \xi, K^{\frac{1}{2}}N_1\eta) \in \mathfrak{R}_i^{A}$ $(i=1,2,3)$ means that $(\xi,\eta)$ is confined to square set whose side length is $2^{-30} A^{-1} K^{-1/2}$ and
$|\xi| \leq 2$, $|\eta| \leq 2^{5} K^{-3/2}$.
This implies that $\textnormal{diam}(S_i) \leq 2^{-20} A^{-1} K^{-1/2}$.
Define ${\mathfrak{n}}_1(\lambda_1)$, ${\mathfrak{n}}_2(\lambda_2)$,
${\mathfrak{n}}_3(\lambda_3)$ denote unit normals on $\lambda_1 \in S_1^K$,
$\lambda_2 \in S_2^K$, $\lambda_3 \in S_3^K$, respectively. Let
\begin{equation*}
\lambda_1=\phi_{\tilde{c_1}}^K (\xi_1, \eta_1), \quad \lambda_2 =  \phi_{\tilde{c_2}}^K (\xi_2,\eta_2), \quad
\lambda_3 = (\psi^K (\xi,\eta), \xi,\eta).
\end{equation*}
We can write that
\begin{align*}
& {\mathfrak{n}}_1(\lambda_1) = \frac{1}{\sqrt{1+ 9 \xi_1^4 + 9 K^3\eta_1^4}}
\bigl(-1, \ 3  \xi_1^2, \ 3 K^{\frac{3}{2}} \eta_1^2 \bigr), \\
& {\mathfrak{n}}_2(\lambda_2) = \frac{1}{\sqrt{1+ 9 \xi_2^4 + 9 K^3\eta_2^4}}
\bigl(-1, \ 3  \xi_2^2, \ 3 K^{\frac{3}{2}} \eta_2^2 \bigr),\\
& {\mathfrak{n}}_3(\lambda_3) = \frac{1}{\sqrt{1+ 9 \xi^4 + 9 K^3\eta^4}}
\bigl(-1, \ 3  \xi^2, \ 3 K^{\frac{3}{2}} \eta^2 \bigr).
\end{align*}
Since $(N_1 \xi_i, K^{\frac{1}{2}}N_1\eta_i) \in \mathfrak{R}_i^{A}$, we can easily verify that the hepersurfaces $S_1^K$, $S_2^K$, $S_3^K$ satisfy the following.
\begin{equation}
\sup_{\lambda_i, \lambda_i' \in S_i^K} \frac{|\mathfrak{n}_i(\lambda_i) -
\mathfrak{n}_i(\lambda_i')|}{|\lambda_i - \lambda_i'|}
+ \frac{|\mathfrak{n}_i(\lambda_i) (\lambda_i - \lambda_i')|}{|\lambda_i - \lambda_i'|^2} \leq 2^{10}.\label{normals00-prop3.19}
\end{equation}
We choose $(\xi_1', \eta_1')$, $(\xi_2', \eta_2')$, $(\xi', \eta')$ to satisfy
\begin{align*}
& (\xi_1', \eta_1') + (\xi_2', \eta_2') = (\xi', \eta'), \\
\phi_{\tilde{c_1}}^K (\xi_1', \eta_1') & \in S_1^K, \ \ \phi_{\tilde{c_2}}^K (\xi_2', \eta_2')\in S_2^K, \ \
(\psi^K (\xi',\eta'), \xi', \eta') \in S_3^K.
\end{align*}
Let $\lambda_1' = \phi_{\tilde{c_1}}^K (\xi_1', \eta_1')$, $\lambda_2' = \phi_{\tilde{c_2}}^K (\xi_2', \eta_2')$,
$\lambda_3' = (\psi^K (\xi',\eta'), \xi',\eta')$.
For any $\lambda_1 \in S_1^K$, $\lambda_2 \in S_2^K$, $\lambda_3 \in S_3^K$, \eqref{normals00-prop3.19}
implies
\begin{align}
& |{\mathfrak{n}}_1(\lambda_1) - {\mathfrak{n}}_1(\lambda_1')|
\leq 2^{-10} A^{-1} K^{-\frac{1}{2}} ,\label{normal01-prop3.19}\\
& |{\mathfrak{n}}_2(\lambda_2) - {\mathfrak{n}}_2(\lambda_2')|
\leq 2^{-10} A^{-1} K^{-\frac{1}{2}} ,\label{normal02-prop3.19}\\
& |{\mathfrak{n}}_3(\lambda_3) - {\mathfrak{n}}_3(\lambda_3')|
\leq 2^{-10} A^{-1} K^{-\frac{1}{2}}.\label{normal03-prop3.19}
\end{align}
Lastly, we show that the hypersurfaces satisfy the desirable transversality condition. From \eqref{normal01-prop3.19}-\eqref{normal03-prop3.19}, it suffices to show
\begin{equation}
 |\textnormal{det} N(\lambda_1', \lambda_2', \lambda_3')| \geq  2^{-5} A^{-1} K^{-\frac{1}{2}} \label{trans-prop3.19}.
\end{equation}
We observe that
\begin{align*}
|\textnormal{det} N(\lambda_1', \lambda_2', \lambda_3')| \geq &
2^{-7}  \left|\textnormal{det}
\begin{pmatrix}
-1 & -1 & - 1 \\
 3 (\xi_1')^2  &  3 (\xi_2')^2  & 3 (\xi')^2 \\
 3 K^{\frac{3}{2}} (\eta_1')^2   & 3 K^{\frac{3}{2}} (\eta_2')^2  & 3 K^{\frac{3}{2}} (\eta')^2
\end{pmatrix} \right| \notag \\
\geq & 2^{-4} K^{\frac{3}{2}} |\xi_1' \eta_2' - \xi_2' \eta_1' | \,
| \xi_1' \eta_2' +  \xi_2' \eta_1' + 2 (\xi_1' \eta_1' + \xi_2' \eta_2')|.
\end{align*}
Hence, it suffices to show
\begin{align}
|\mathfrak{A}(\xi_1', \eta_1', \xi_2' ,\eta_2')| & :=  |\xi_1' \eta_2' - \xi_2' \eta_1' | \geq
A^{-1} K^{-\frac{1}{2}},\label{angular02-prop3.19}\\
|F(\xi_1', \eta_1', \xi_2' ,\eta_2')|&= | \xi_1' \eta_2' +  \xi_2' \eta_1' + 2 (\xi_1' \eta_1' + \xi_2' \eta_2')| \geq
 2^{-1}K^{-\frac{3}{2}}.\label{transversal02-prop3.19}
\end{align}
\eqref{angular02-prop3.19} and \eqref{transversal02-prop3.19} are equivalent to
\begin{align}
|\mathfrak{A}(N_1\xi_1', K^{\frac{1}{2}} N_1\eta_1', N_1 \xi_2' , K^{\frac{1}{2}} N_1 \eta_2')| &
= K^{\frac{1}{2}}N_1^2 |\xi_1' \eta_2' - \xi_2' \eta_1' | \geq
 A^{-1} N_1^2,\label{angular03-prop3.19}\\
|F(N_1\xi_1', K^{\frac{1}{2}} N_1\eta_1', N_1 \xi_2' , K^{\frac{1}{2}} N_1 \eta_2')|&
= K^{\frac{1}{2}} N_1^2 | \xi_1' \eta_2' +  \xi_2' \eta_1' + 2 (\xi_1' \eta_1' + \xi_2' \eta_2')| \notag \\
 & \geq 2^{-1} K^{-1} N_1^2,\label{transversal03-prop3.19}
\end{align}
respectively.
We deduce from $(N_1 \xi_1', K^{\frac{1}{2}} N_1 \eta_1') \in \mathfrak{R}_1^{A}$,
$(N_1 \xi_2', K^{\frac{1}{2}} N_1 \eta_2') \in \mathfrak{R}_2^{A}$,
$\mathfrak{R}_1^{A} \cap \mathfrak{D}_{j_1}^A \not= \emptyset$ and
$\mathfrak{R}_2^{A} \cap \mathfrak{D}_{j}^A \not= \emptyset$ that
\eqref{angular03-prop3.19} holds.
Letting $(\xi_1, \eta_1) \in \mathfrak{R}_1^{A}$ and
$(\xi_2,\eta_2) \in \mathfrak{R}_2^{A}$, \eqref{transversal03-prop3.19} is obtained as follows.
\begin{align*}
|F(\xi_1,\eta_1,\xi_2,\eta_2)| & =|\xi_1 \eta_2+ \xi_2 \eta_1 + 2(\xi_1 \eta_1 + \xi_2 \eta_2)|\\
& \geq 2| \xi_2 \eta_1 + \xi_1 \eta_1 + \xi_2 \eta_2| - |\xi_1 \eta_2- \xi_2 \eta_1| \\
& \geq 2|\xi_2 \eta_2| - 2|\xi_1 + \xi_2| |\eta_2| - 2^7 A^{-1} N_1^2 \\
& \geq K^{-1} N_1^2 -2^8 A^{-1} N_1^2 \\
& \geq 2^{-1}K^{-1} N_1^2.
\end{align*}
Here we used $|\xi_1 + \xi_2| \leq 2^{11} A^{-1} K^{-1/2} N_1$ which follows from
$(\mathfrak{R}_1^{A}  + \mathfrak{R}_2^{A} )  \cap \mathfrak{R}_3^{A}$. As a result, we get
\eqref{trans-prop3.19} which, combined with \eqref{normal01-prop3.19}-\eqref{normal03-prop3.19}, completes the proof of \eqref{est03-prop3.19}.
\end{proof}
We next consider the case $j_1 \in \mathfrak{J}$. The goal is to establish the following estimate.
\begin{prop}\label{prop3.20}
Assume \textnormal{(i)}, \textnormal{(ii)'} in \textit{Case} \textnormal{\ref{case-2}}
and $\min(N_0, N_1, N_2) = N_2.$ Let $A$ be dyadic which satisfy $2^{25} \leq A \leq N_1 L_0^{-1/3}$,
$j_1 \in \mathfrak{J}$, $j$ satisfy $16 \leq |j_1-j| \leq 32$.
Then we have
\begin{equation}
\begin{split}
&
\Bigl|\int_{**}{  \ha{v}_{{N_2, L_2}}(\tau_2, \xi_2, \eta_2)
\ha{u}_{N_1, L_1}|_{\tilde{\mathfrak{D}}_{j_1}^A}(\tau_1, \xi_1, \eta_1)
\ha{w}_{N_0, L_0}|_{\tilde{\mathfrak{D}}_{j}^A}(\tau, \xi, \eta)
}
d\sigma_1 d\sigma \Bigr|\\
& \qquad \qquad \lesssim  A^{\frac{1}{4}} N_1^{-\frac{5}{4}} L_0^{\frac{1}{4}} (L_1 L_2)^{\frac{1}{2}} \|\ha{u}_{N_1, L_1}|_{\tilde{\mathfrak{D}}_{j_1}^A} \|_{L^2}
\| \ha{v}_{N_2, L_2}\|_{L^2}
\|\ha{w}_{{N_0, L_0}}|_{\tilde{\mathfrak{D}}_{j}^A}  \|_{L^2},\label{prop3.20-1}
\end{split}
\end{equation}
where $d \sigma_1 = d\tau_1 d \xi_1 d \eta_1$, $d \sigma = d\tau d \xi d \eta$
 and $**$ denotes $(\tau_2, \xi_2, \eta_2) = (\tau_1 + \tau, \xi_1+ \xi, \eta_1 + \eta).$
\end{prop}
In contrast to the case $j_1 \in \mathfrak{J}_A^K$,
there exists $(\xi_1,\eta_1) \times (\xi,\eta) \in \mathfrak{D}_{j_1}^A \times \mathfrak{D}_{j}^A$ such that $F(\xi_1, \eta_1, \xi,\eta) = 0$.
This suggests that the same argument as in the proof of Proposition \ref{prop3.19} is no longer available.
To overcome this difficulty, we again perform the Whitney type decomposition. To do so, we define the following decomposition.
\begin{defn}
Let $A \geq 2^{25}$ and $d \geq 2^{20}$ be dyadic and $k = ( k_{(1)}, k_{(2)} ) \in \Z^2$. We define rectangle-tiles
$\{ \mathcal{T}_k^{A,d}\}_{k \in \Z^2}$ whose short side is parallel to $\xi$-axis and its length is $A^{-3/2} d^{-1} N_1$, long side length is $A^{-1} d^{-1} N_1 $ and prisms $\{ \tilde{\mathcal{T}}_k^{A,d}\}_{k \in \Z^2}$ as follows:
\begin{align*}
& \mathcal{T}_k^{A,d} : = \{ (\xi,\eta) \in \R^2 \, | \, \xi \in A^{-\frac{3}{2}} d^{-1} N_1
[ k_{(1)}, k_{(1)} + 1), \ \eta \in A^{-1} d^{-1} N_1 [ k_{(2)}, k_{(2)} + 1)  \},\\
& \tilde{\mathcal{T}}_k^{A,d}  : = \R \times \mathcal{T}_k^{A,d}.
\end{align*}
\end{defn}
The following two estimates play a crucial role in the proof of Proposition \ref{prop3.20}.
\begin{prop}\label{prop3.21}
Assume \textnormal{(i)}, \textnormal{(ii)'} in \textit{Case} \textnormal{\ref{case-2}}
and $\min(N_0, N_1, N_2) = N_2.$ Let $A \geq 2^{25}$ and $d \geq 2^{20}$ be dyadic,
$j_1 \in \mathfrak{J}$, $j$ satisfy $16 \leq |j_1-j| \leq 32$. We assume that $k_1$, $k \in \Z^2$ satisfy
\begin{align*}
& \bigl( \mathcal{T}_{k_1}^{A,d} \times \mathcal{T}_{k}^{A,d} \bigr)
\cap ( \mathfrak{D}_{j_1}^A \times \mathfrak{D}_{j}^A )
\not= \emptyset,\\
|\Phi(\xi_1, \eta_1, \xi, \eta)| & \geq A^{-\frac{3}{2}} d^{-1}  N_1^{3}  \ \ \textnormal{for any} \ (\xi_1, \eta_1)
\times (\xi,\eta) \in
\mathcal{T}_{k_1}^{A,d} \times \mathcal{T}_{k}^{A,d}.
\end{align*}
Then we have
\begin{equation}
\begin{split}
&
\Bigl|\int_{**}{  \ha{v}_{{N_2, L_2}}(\tau_2, \xi_2, \eta_2)
\ha{u}_{N_1, L_1}|_{\tilde{\mathcal{T}}_{k_1}^{A,d}}(\tau_1, \xi_1, \eta_1)
\ha{w}_{N_0, L_0}|_{\tilde{\mathcal{T}}_k^{A,d}}(\tau, \xi, \eta)
}
d\sigma_1 d\sigma \Bigr|\\
& \qquad \qquad \lesssim  A d^{\frac{1}{2}} N_1^{-2} (L_0 L_1 L_2)^{\frac{1}{2}} \|\ha{u}_{N_1, L_1}|_{\tilde{\mathcal{T}}_{k_1}^{A,d}} \|_{L^2}
\| \ha{v}_{N_2, L_2}\|_{L^2}
\|\ha{w}_{N_0, L_0}|_{\tilde{\mathcal{T}}_k^{A,d}}  \|_{L^2}\label{prop3.21-1}
\end{split}
\end{equation}
where $d \sigma_1 = d\tau_1 d \xi_1 d \eta_1$, $d \sigma = d\tau d \xi d \eta$
 and $**$ denotes $(\tau_2, \xi_2, \eta_2) = (\tau_1 + \tau, \xi_1+ \xi, \eta_1 + \eta).$
\end{prop}
\begin{proof}
It is clear that \eqref{bilinearStrichartz-11}, \eqref{bilinearStrichartz-12}, \eqref{bilinearStrichartz-14} in Proposition
\ref{prop3.18} and the assumption $|\Phi(\xi_1, \eta_1, \xi, \eta)|  \geq A^{-3/2} d^{-1}  N_1^{3}$
yield \eqref{prop3.21-1}.
\end{proof}
\begin{prop}\label{prop3.22}
Assume \textnormal{(i)}, \textnormal{(ii)'} in \textit{Case} \textnormal{\ref{case-2}}
and $\min(N_0, N_1, N_2) = N_2.$ Let $A \geq 2^{25}$ and $d \geq 2^{20}$ be dyadic,
$j_1 \in \mathfrak{J}$, $j$ satisfy $16 \leq |j_1-j| \leq 32$. We assume that $k_1$, $k \in \Z^2$ satisfy
\begin{align*}
& \bigl( \mathcal{T}_{k_1}^{A,d} \times \mathcal{T}_{k}^{A,d} \bigr)
\cap ( \mathfrak{D}_{j_1}^A \times \mathfrak{D}_{j}^A )
\not= \emptyset,\\
 | F (\xi_1, \eta_1, \xi ,\eta)| &  \geq A^{-1} d^{-1} N_1^2  \ \ \textnormal{for any} \ (\xi_1, \eta_1)
\times (\xi,\eta) \in
\mathcal{T}_{k_1}^{A,d} \times \mathcal{T}_{k}^{A,d}.
\end{align*}
Then we have
\begin{equation}
\begin{split}
&
\Bigl|\int_{**}{  \ha{v}_{{N_2, L_2}}(\tau_2, \xi_2, \eta_2)
\ha{u}_{N_1, L_1}|_{\tilde{\mathcal{T}}_{k_1}^{A,d}}(\tau_1, \xi_1, \eta_1)
\ha{w}_{N_0, L_0}|_{\tilde{\mathcal{T}}_k^{A,d}}(\tau, \xi, \eta)
}
d\sigma_1 d\sigma \Bigr|\\
& \qquad \qquad \lesssim  A d^{\frac{1}{2}} N_1^{-2} (L_0 L_1 L_2)^{\frac{1}{2}} \|\ha{u}_{N_1, L_1}|_{\tilde{\mathcal{T}}_{k_1}^{A,d}} \|_{L^2}
\| \ha{v}_{N_2, L_2}\|_{L^2}
\|\ha{w}_{N_0, L_0}|_{\tilde{\mathcal{T}}_k^{A,d}}  \|_{L^2},\label{prop3.22-1}
\end{split}
\end{equation}
where $d \sigma_1 = d\tau_1 d \xi_1 d \eta_1$, $d \sigma = d\tau d \xi d \eta$
 and $**$ denotes $(\tau_2, \xi_2, \eta_2) = (\tau_1 + \tau, \xi_1+ \xi, \eta_1 + \eta).$
\end{prop}
\begin{proof}
The strategy of the proof is completely the same as that for the proof of Proposition \ref{prop3.19} for
the case $|\xi_2| \leq 2^{10}A^{-1}K^{-1/2}N_1$.
Since we may assume that $(\xi_1,\eta_1) \in \mathcal{T}_{k_1}^{A,d}$ and
$(\xi,\eta) \in \mathcal{T}_{k}^{A,d}$, it suffices to prove the following:
\begin{equation}
\begin{split}
& \Bigl|\int_{\R^3 \times \R^3} { h (\tau_1+\tau_2, \xi_1+\xi_2, \eta_1+\eta_2)
f(\tau_1, \xi_1, \eta_1)  g(\tau_2, \xi_2, \eta_2)
}
d\sigma_1 d\sigma_2 \Bigr|\\
& \qquad \qquad \qquad \qquad \lesssim A d^{\frac{1}{2}} N_1^{-2} (L_0 L_1 L_2)^{\frac{1}{2}} \|f \|_{L^2}
\|g\|_{L^2}
\| h \|_{L^2}.
\end{split}\label{est01-prop3.22}
\end{equation}
Here $f$, $g$, $h$ satisfy
\begin{equation*}
\operatorname{supp} f \subset G_{N_1, L_1} \cap \tilde{\mathfrak{R}}_1^{A,d}, \quad
\operatorname{supp} g \subset G_{N_0, L_0} \cap \tilde{\mathfrak{R}}_2^{A,d}, \quad
\operatorname{supp} h \subset G_{N_2, L_2} \cap \tilde{\mathfrak{R}}_3^{A,d},
\end{equation*}
where $\tilde{\mathfrak{R}}_1^{A,d}$, $\tilde{\mathfrak{R}}_2^{A,d}$, $\tilde{\mathfrak{R}}_3^{A,d}$ are prisms which are defined as follows:
\begin{align*}
\tilde{\mathfrak{R}}_1^{A,d} & = \R \times \mathfrak{R}_1^{A,d}, \quad
\tilde{\mathfrak{R}}_2^{A,d} = \R \times \mathfrak{R}_2^{A,d}, \quad
\tilde{\mathfrak{R}}_3^{A,d} = \R \times \mathfrak{R}_3^{A,d},\\
\mathfrak{R}_1^{A,d} & = \{ (\xi,\eta) \,
|\, \alpha_1\leq \xi \leq \alpha_1+2^{-30}A^{-\frac{3}{2} }d^{-1}N_1 , \
\beta_1 \leq \eta \leq \beta_1 + 2^{-30}A^{-1} d^{-1}N_1\},\\
\mathfrak{R}_2^{A,d} & = \{ (\xi,\eta)  \,
|\, \alpha_2 \leq  \xi \leq \alpha_2+2^{-30}A^{-\frac{3}{2} }d^{-1}N_1 , \
\beta_2 \leq \eta \leq \beta_2 + 2^{-30}A^{-1} d^{-1}N_1\},\\
\mathfrak{R}_3^{A,d} & = \{ (\xi,\eta)  \,
|\, \alpha_3 \leq \xi \leq \alpha_3+2^{-30}A^{-\frac{3}{2} }d^{-1}N_1 , \
\beta_3 \leq \eta \leq \beta_3 + 2^{-30}A^{-1} d^{-1}N_1\}.
\end{align*}
Here we choose $\alpha_i$, $\beta_i \in \R$ to satisfy
\begin{align*}
 |\beta_3| \leq 2^{10} A^{-1} N_1, \
(\mathfrak{R}_1^{A}  + \mathfrak{R}_2^{A} )  \cap \mathfrak{R}_3^{A} \not= \emptyset, \quad
\mathfrak{R}_1^{A} \cap \mathcal{T}_{k_1}^{A,d} \not= \emptyset, \quad
\mathfrak{R}_2^{A} \cap \mathcal{T}_{k}^{A,d} \not= \emptyset.
\end{align*}
As we saw in the proofs of Propositions \ref{prop3.14} and \ref{prop3.19}, it suffices to show
\begin{equation}
\| \tilde{f} |_{S_1} * \tilde{g} |_{S_2} \|_{L^2(S_3)} \lesssim  A d^{\frac{1}{2}}
\| \tilde{f} \|_{L^2(S_1)} \| \tilde{g} \|_{L^2(S_2)},\label{est02-prop3.22}
\end{equation}
where the notations $\tilde{f}$, $\tilde{g}$, $S_1$, $S_2$, $S_3$ are the same as in the proof of Propositions \ref{prop3.4} and \ref{prop3.19}:
\begin{align*}
& \tilde{f}  (\tau_1, \xi_1 , \eta_1)  = f (N_1^3 \tau_1 , N_1 \xi_1, N_1 \eta_1), \\
& \tilde{g} (\tau_2, \xi_2, \eta_2)  = g (N_1^3 \tau_2, N_1 \xi_2, N_1 \eta_2), \\
& S_1  = \{ \phi_{\tilde{c_1}} (\xi, \eta) = (\xi^3 + \eta^3 + \tilde{c_1}, \, \xi, \, \eta) \in \R^3 \ | \
(N_1 \xi, N_1\eta) \in \mathfrak{R}_1^{A,d} \}, \\
& S_2  = \{ \phi_{\tilde{c_2}} (\xi, \eta) = (\xi^3 + \eta^3 + \tilde{c_2}, \, \xi, \, \eta) \in \R^3 \ | \ (N_1 \xi, N_1 \eta) \in \mathfrak{R}_2^{A,d} \},\\
& S_3  = \Bigl\{ (\psi (\xi,\eta), \xi, \eta) \in \R^3  \ | \  (N_1\xi, N_1 \eta) \in \mathfrak{R}_3^{A,d},
\  \psi (\xi,\eta) =  \xi^3 + \eta^3 + \frac{c_0'}{N_1^{3}} \Bigr\}.
\end{align*}
Similarly to the proof of Proposition \ref{prop3.19}, we perform the scaling $(\tau,\xi,\eta) \to ( \tau , \,  \xi, \, A^{1/2} \eta)$ to define
\begin{align*}
 \tilde{f}_A (\tau_1, \xi_1 , \eta_1) & = \tilde{f} ( \tau_1 ,  \xi_1, A^{\frac{1}{2}} \eta_1), \\
 \tilde{g}_A (\tau_2, \xi_2, \eta_2) & = \tilde{g} ( \tau_2,  \xi_2, A^{\frac{1}{2}} \eta_2), \\
 \tilde{h}_A (\tau, \xi, \eta) & = \tilde{h} ( \tau,  \xi, A^{\frac{1}{2}} \eta).
\end{align*}
Now \eqref{est02-prop3.22} can be rewritten as
\begin{equation}
\| \tilde{f}_A |_{S_1^A} * \tilde{g}_A |_{S_2^A} \|_{L^2(S_3^A)} \lesssim A^{\frac{3}{4}} d^{\frac{1}{2}}
\| \tilde{f}_A \|_{L^2(S_1^A)} \| \tilde{g}_A \|_{L^2(S_2^A)},\label{est03-prop3.22}
\end{equation}
where
\begin{align*}
& S_1^A  = \{ \phi_{\tilde{c_1}}^A (\xi, \eta) = (\xi^3 + A^{\frac{3}{2}}\eta^3 + \tilde{c_1}, \, \xi, \, \eta)  \ | \
(N_1 \xi, A^{\frac{1}{2}}N_1\eta) \in \mathfrak{R}_1^{A,d} \}, \\
& S_2^A  = \{ \phi_{\tilde{c_2}}^A (\xi, \eta) = (\xi^3 + A^{\frac{3}{2}} \eta^3 + \tilde{c_2}, \, \xi, \, \eta) \ | \ (N_1 \xi, A^{\frac{1}{2}} N_1 \eta) \in \mathfrak{R}_2^{A,d} \},\\
& S_3^A  = \Bigl\{ (\psi^A (\xi,\eta), \xi, \eta)   \ | \  (N_1\xi, A^{\frac{1}{2}} N_1 \eta) \in
\mathfrak{R}_3^{A,d},
\  \psi^A (\xi,\eta) =  \xi^3 + A^{\frac{3}{2}} \eta^3 + \frac{c_0'}{N_1^{3}} \Bigr\}.
\end{align*}
We verify that $S_1^A$, $S_2^A$, $S_3^A$ satisfy the necessary conditions to apply Proposition \ref{prop2.7}.
Let ${\mathfrak{n}}_1(\lambda_1)$, ${\mathfrak{n}}_2(\lambda_2)$,
${\mathfrak{n}}_3(\lambda_3)$ be unit normals on $\lambda_1 \in S_1^A$,
$\lambda_2 \in S_2^A$, $\lambda_3 \in S_3^A$, respectively. We define
\begin{equation*}
\lambda_1=\phi_{\tilde{c_1}}^A (\xi_1, \eta_1), \quad \lambda_2 =  \phi_{\tilde{c_2}}^A (\xi_2,\eta_2), \quad
\lambda_3 = (\psi^A (\xi,\eta), \xi,\eta).
\end{equation*}
The unit normals can be written explicitly as
\begin{align*}
& {\mathfrak{n}}_1(\lambda_1) = \frac{1}{\sqrt{1+ 9 \xi_1^4 + 9 A^3\eta_1^4}}
\bigl(-1, \ 3  \xi_1^2, \ 3 A^{\frac{3}{2}} \eta_1^2 \bigr), \\
& {\mathfrak{n}}_2(\lambda_2) = \frac{1}{\sqrt{1+ 9 \xi_2^4 + 9 A^3\eta_2^4}}
\bigl(-1, \ 3  \xi_2^2, \ 3 A^{\frac{3}{2}} \eta_2^2 \bigr),\\
& {\mathfrak{n}}_3(\lambda_3) = \frac{1}{\sqrt{1+ 9 \xi^4 + 9 A^3\eta^4}}
\bigl(-1, \ 3  \xi^2, \ 3 A^{\frac{3}{2}} \eta^2 \bigr).
\end{align*}
$(N_1 \xi_i, A^{\frac{1}{2}}N_1\eta_i) \in \mathfrak{R}_i^{A}$ implies that the hypersurfaces $S_1^A$, $S_2^A$, $S_3^A$ satisfy the certain regularity conditions.
\begin{equation}
\sup_{\lambda_i, \lambda_i' \in S_i^A} \frac{|\mathfrak{n}_i(\lambda_i) -
\mathfrak{n}_i(\lambda_i')|}{|\lambda_i - \lambda_i'|}
+ \frac{|\mathfrak{n}_i(\lambda_i) (\lambda_i - \lambda_i')|}{|\lambda_i - \lambda_i'|^2} \leq 2^{10},\label{normals00-prop3.22}
\end{equation}
and diameter conditions
\begin{equation}
\textnormal{diam}(S_i) \leq 2^{-20} A^{-\frac{3}{2}} d^{-1}. \label{diam-prop3.22}
\end{equation}
Choose $(\xi_1', \eta_1')$, $(\xi_2', \eta_2')$, $(\xi', \eta')$ such that
\begin{align*}
& (\xi_1', \eta_1') + (\xi_2', \eta_2') = (\xi', \eta'), \\
\phi_{\tilde{c_1}}^A (\xi_1', \eta_1') & \in S_1^A, \ \ \phi_{\tilde{c_2}}^A (\xi_2', \eta_2')\in S_2^A, \ \
(\psi^A (\xi',\eta'), \xi', \eta') \in S_3^A.
\end{align*}
Define $\lambda_1' = \phi_{\tilde{c_1}}^A (\xi_1', \eta_1')$, $\lambda_2' = \phi_{\tilde{c_2}}^A (\xi_2', \eta_2')$,
$\lambda_3' = (\psi^A (\xi',\eta'), \xi',\eta')$.
For any $\lambda_1 \in S_1^A$, $\lambda_2 \in S_2^A$, $\lambda_3 \in S_3^A$, \eqref{normals00-prop3.22}
and \eqref{diam-prop3.22}
imply
\begin{align}
& |{\mathfrak{n}}_1(\lambda_1) - {\mathfrak{n}}_1(\lambda_1')|
\leq 2^{-10} A^{-\frac{3}{2}} d^{-1} ,\label{normal01-prop3.22}\\
& |{\mathfrak{n}}_2(\lambda_2) - {\mathfrak{n}}_2(\lambda_2')|
\leq 2^{-10} A^{-\frac{3}{2}} d^{-1} ,\label{normal02-prop3.22}\\
& |{\mathfrak{n}}_3(\lambda_3) - {\mathfrak{n}}_3(\lambda_3')|
\leq 2^{-10} A^{-\frac{3}{2}} d^{-1}.\label{normal03-prop3.22}
\end{align}
We lastly observe that the hypersurfaces satisfy transversality condition. Since \eqref{normal01-prop3.22}-\eqref{normal03-prop3.22}, we only need to show
\begin{equation}
 |\textnormal{det} N(\lambda_1', \lambda_2', \lambda_3')| \geq  2^{-5}  A^{-\frac{3}{2}} d^{-1} \label{trans-prop3.22}.
\end{equation}
As in the proof of Proposition \ref{prop3.19}, we calculate that
\begin{align*}
|\textnormal{det} N(\lambda_1', \lambda_2', \lambda_3')| \geq &
2^{-7}  \left|\textnormal{det}
\begin{pmatrix}
-1 & -1 & - 1 \\
 3 (\xi_1')^2  &  3 (\xi_2')^2  & 3 (\xi')^2 \\
 3 A^{\frac{3}{2}} (\eta_1')^2   & 3 A^{\frac{3}{2}} (\eta_2')^2  & 3 A^{\frac{3}{2}} (\eta')^2
\end{pmatrix} \right| \notag \\
\geq & 2^{-4} A^{\frac{3}{2}} |\xi_1' \eta_2' - \xi_2' \eta_1' | \,
| \xi_1' \eta_2' +  \xi_2' \eta_1' + 2 (\xi_1' \eta_1' + \xi_2' \eta_2')|.
\end{align*}
Hence, it suffices to show
\begin{align}
|\mathfrak{A}(\xi_1', \eta_1', \xi_2' ,\eta_2')| & =  |\xi_1' \eta_2' - \xi_2' \eta_1' | \geq
A^{-\frac{3}{2}},\label{angular02-prop3.22}\\
|F(\xi_1', \eta_1', \xi_2' ,\eta_2')|&= | \xi_1' \eta_2' +  \xi_2' \eta_1' + 2 (\xi_1' \eta_1' + \xi_2' \eta_2')| \geq
 2^{-1}A^{-\frac{3}{2}} d^{-1},\label{transversal02-prop3.22}
\end{align}
which are equivalent to
\begin{align}
|\mathfrak{A}(N_1\xi_1', A^{\frac{1}{2}} N_1\eta_1', N_1 \xi_2' , A^{\frac{1}{2}} N_1 \eta_2')| &
 \geq
 A^{-1} N_1^2,\label{angular03-prop3.22}\\
|F(N_1\xi_1', A^{\frac{1}{2}} N_1\eta_1', N_1 \xi_2' , A^{\frac{1}{2}} N_1 \eta_2')|
 & \geq 2^{-1} A^{-1} d^{-1} N_1^2,\label{transversal03-prop3.22}
\end{align}
respectively.
Note that $(N_1 \xi_1', A^{\frac{1}{2}} N_1 \eta_1') \in \mathfrak{R}_1^{A,d}$,
$(N_1 \xi_2', A^{\frac{1}{2}} N_1 \eta_2') \in \mathfrak{R}_2^{A}$ and
recall that $\mathfrak{R}_1^{A} \cap \mathcal{T}_{k_1}^{A,d} \not= \emptyset$,
$\mathfrak{R}_2^{A} \cap \mathcal{T}_{k}^{A,d} \not= \emptyset$.
The assumption
$\mathcal{T}_{k_1}^{A,d} \cap \mathfrak{D}_{j_1}^A \not= \emptyset$,
$\mathcal{T}_{k}^{A,d} \cap \mathfrak{D}_{j}^A \not= \emptyset$ immediately yield \eqref{angular03-prop3.22}.
Additionally, since we assume that
\begin{equation*}
 | F (\xi_1, \eta_1, \xi ,\eta)|   \geq A^{-1} d^{-1} N_1^2  \ \ \textnormal{for any} \ (\xi_1, \eta_1)
\times (\xi,\eta) \in
\mathcal{T}_{k_1}^{A,d} \times \mathcal{T}_{k}^{A,d},
\end{equation*}
we have \eqref{transversal03-prop3.22}.
Therefore \eqref{angular02-prop3.22} and \eqref{transversal02-prop3.22} hold true.
Consequently, we obtain \eqref{trans-prop3.22} which verify the desired estimate \eqref{est03-prop3.22}.
\end{proof}
\begin{defn}[Whitney type decomposition]
Let $A \geq 2^{25}$ and $d \geq 2^{20}$ be dyadic and $j_1$, $j \in \mathfrak{J}_A$. Recall that
\begin{align*}
\Phi (\xi_1, \eta_1, \xi ,\eta) & = \xi_1 \xi(\xi_1 + \xi) + \eta_1 \eta (\eta_1 + \eta), \\
F (\xi_1, \eta_1, \xi ,\eta) & = \xi_1 \eta +  \xi \eta_1 + 2 (\xi_1 \eta_1 + \xi \eta).
\end{align*}
We define $Z_{A,d,j_1,j}^1 $ as the set of $(k_1, k) \in \Z^2 \times \Z^2 $ such that
\begin{equation*}
\begin{cases}
 |\Phi(\xi_1, \eta_1, \xi, \eta)| \geq A^{-\frac{3}{2}} d^{-1} N_1^3  \ \ \textnormal{for any} \
(\xi_1, \eta_1) \times (\xi,\eta) \in
\mathcal{T}_{k_1}^{A,d} \times \mathcal{T}_{k}^{A,d}, \\
 \bigl( \mathcal{T}_{k_1}^{A,d} \times \mathcal{T}_{k}^{A,d} \bigr)
\cap ( \mathfrak{D}_{j_1}^A \times \mathfrak{D}_{j}^A )
\not= \emptyset,\\
 |\xi_1+ \xi| \leq 2^{10}A^{-\frac32}N_1  \ \ \textnormal{for any} \
(\xi_1, \eta_1) \times (\xi,\eta) \in
\mathcal{T}_{k_1}^{A,d} \times \mathcal{T}_{k}^{A,d}.
\end{cases}
\end{equation*}
Similarly, we define $Z_{A,d,j_1,j}^2$
as the set of $(k_1, k) \in \Z^2 \times \Z^2$ such that
\begin{equation*}
\begin{cases}
 |F(\xi_1, \eta_1, \xi, \eta)| \geq A^{-1} d^{-1} N_1^2  \ \ \textnormal{for any} \
(\xi_1, \eta_1) \times (\xi,\eta) \in
\mathcal{T}_{k_1}^{A,d} \times \mathcal{T}_{k}^{A,d}, \\
 \bigl( \mathcal{T}_{k_1}^{A,d} \times \mathcal{T}_{k}^{A,d} \bigr)
\cap ( \mathfrak{D}_{j_1}^A \times \mathfrak{D}_{j}^A )
\not= \emptyset,\\
|\xi_1+ \xi| \leq 2^{10}A^{-\frac32}N_1  \ \ \textnormal{for any} \
(\xi_1, \eta_1) \times (\xi,\eta) \in
\mathcal{T}_{k_1}^{A,d} \times \mathcal{T}_{k}^{A,d},
\end{cases}
\end{equation*}
and
\begin{equation*}
Z_{A,d}^{j_1,j} = Z_{A,d,j_1,j}^1 \, \cup \, Z_{A,d,j_1,j}^2 ,
\quad R_{A,d}^{j_1,j} = \bigcup_{(k_1, k) \in Z_{A,d}^{j_1,j}} \mathcal{T}_{k_1}^{A,d} \times
\mathcal{T}_{k}^{A,d} \subset \R^2 \times \R^2.
\end{equation*}
It is clear that $d_1 \leq d_2 \Longrightarrow  R_{A,d_1}^{j_1,j} \subset  R_{A,d_2}^{j_1,j}$.
Further, we define
\begin{equation*}
Q_{A,d}^{j_1,j} =
\begin{cases}
R_{A,d}^{j_1,j} \setminus R_{A,\frac{d}{2}}^{j_1,j} \quad \textnormal{for} \  d \geq 2^{21},\\
 \ R_{A,2^{20}}^{j_1,j}  \quad \, \qquad \textnormal{for} \  d = 2^{20},
\end{cases}
\end{equation*}
and a set of pairs of integer pair $\widehat{Z}_{A,d}^{j_1,j} \subset Z_{A,d}^{j_1,j}$ as
\begin{equation*}
\bigcup_{(k_1, k) \in \widehat{Z}_{A,d}^{j_1,j}} \mathcal{T}_{k_1}^{A,d} \times
\mathcal{T}_{k}^{A,d} = Q_{A,d}^{j_1,j}.
\end{equation*}
Clearly, $\widehat{Z}_{A,d}^{j_1,j}$ is uniquely defined and
\begin{equation*}
d_1 \not= d_2 \Longrightarrow Q_{A,d_1}^{j_1,j} \cap Q_{A,d_2}^{j_1,j} = \emptyset, \quad
\bigcup_{2^{20} \leq d \leq d_0} Q_{A,d}^{j_1,j} = R_{A,d_0}^{j_1,j}
\end{equation*}
where $d_0 \geq 2^{20}$ is dyadic.
Lastly, we define $\overline{Z}_{A,d}^{j_1,j}$ as the collection of $(k_1, k) \in \Z^2 \times \Z^2$ which satisfies
\begin{equation*}
\begin{cases}
\mathcal{T}_{k_1}^{A,d} \times \mathcal{T}_{k}^{A,d} \not\subset
\displaystyle{\bigcup_{2^{20} \leq d' \leq d}
\bigcup_{(k_1', k') \in \widehat{Z}_{A,d'}^{j_1,j}}}
\bigl( \mathcal{T}_{k_1'}^{A,d'} \times \mathcal{T}_{k'}^{A,d'} \bigr),\\
 \bigl( \mathcal{T}_{k_1}^{A,d} \times \mathcal{T}_{k}^{A,d} \bigr)
\cap ( \mathfrak{D}_{j_1}^A \times \mathfrak{D}_{j}^A )
\not= \emptyset,\\
 |\xi_1+ \xi| \leq 2^{10}A^{-\frac32}N_1  \ \ \textnormal{for any} \
(\xi_1, \eta_1) \times (\xi,\eta) \in
\mathcal{T}_{k_1}^{A,d} \times \mathcal{T}_{k}^{A,d}.
\end{cases}
\end{equation*}
\end{defn}
\begin{lem}\label{lemma3.23}
Let $A \geq 2^{25}$ and $d \geq 2^{20}$ be dyadic and $j_1$, $j \in \mathfrak{J}_A$.
For fixed $k_1 \in \Z^2$, the number of $k_2 \in \Z^2$ such that
$(k_1, k_2) \in \widehat{Z}_{A,d}^{j_1,j}$ is less than $2^{1000}$. Furthermore, the same claim holds true if we replace $\widehat{Z}_{A,d}^{j_1,j}$ by $\overline{Z}_{A,d}^{j_1,j}$.
\end{lem}
\begin{proof}
Without loss of generality, we can assume that $d \geq 2^{100}$.
We follow the same strategy as that for the proof of Lemma \ref{lemma3.6}.
We assume $(k_1, k) \in \widehat{Z}_{A,d}^{j_1,j}$ and define
$k_1'=k_1'(k_1) \in \Z^2$ and $k'=k'(k) \in \Z^2$ such that
$ \mathcal{T}_{k_1}^{A,d} \subset  \mathcal{T}_{k_1'}^{A, d/2}$ and
$ \mathcal{T}_{k}^{A,d} \subset  \mathcal{T}_{k'}^{A,d/2}$, respectively.
$(k_1, k_2) \in \widehat{Z}_{A,d}^{j_1,j}$ suggests that there exist
$(\bar{\xi}_1, \bar{\eta}_1)$, $(\tilde{\xi}_1, \tilde{\eta}_1) \in \mathcal{T}_{k_1'}^{A,d/2}$,
$(\bar{\xi}, \bar{\eta})$, $(\tilde{\xi}, \tilde{\eta}) \in \mathcal{T}_{k'}^{A,d/2}$ which satisfy
\begin{equation}
|\Phi(\bar{\xi}_1, \bar{\eta}_1, \bar{\xi}, \bar{\eta})| \leq 2 A^{-\frac{3}{2}}d^{-1} N_1^3  \quad \textnormal{and} \quad
 |F (\tilde{\xi}_1, \tilde{\eta}_1, \tilde{\xi} ,\tilde{\eta})| \leq 2 A^{-1} d^{-1} N_1^2.\label{lemma3.6-est001}
\end{equation}
Let $(\xi_1', \eta_1')$ be the center of $\mathcal{T}_{k_1}^{A,d}$. Since
$j_1$, $j \in \mathfrak{J}_A$ and
$ \bigl( \mathcal{T}_{k_1}^{A,d} \times \mathcal{T}_{k}^{A,d} \bigr)
\cap ( \mathfrak{D}_{j_1}^A \times \mathfrak{D}_{j}^A )
\not= \emptyset$, we have $|\eta_1| \leq 2^{11}  A^{-1} N_1$, $|\eta| \leq 2^{11} A^{-1} N_1$ for any
$(\xi_1,\eta_1) \times  (\xi,\eta) \in \mathcal{T}_{k_1}^{A,d} \times \mathcal{T}_{k}^{A,d} $.
Therefore, for $(\xi,\eta) \in \mathcal{T}_{k}^{A,d} $,
\eqref{lemma3.6-est001} implies
\begin{equation*}
|\Phi(\xi_1', \eta_1', \xi,\eta)| \leq 2^5 A^{-\frac{3}{2}}d^{-1} N_1^3  \quad \textnormal{and} \quad
 |F (\xi_1', \eta_1', \xi,\eta)| \leq 2^5 A^{-1} d^{-1} N_1^2.
\end{equation*}
Hence, it suffices to show that there exists
$\tilde{k} \in \Z^2$ such that
\begin{equation*}
\left\{ (\xi, \eta)  \in \mathfrak{D}_{j}^A \, \left| \,
\begin{aligned} & |\Phi(\xi_1', \eta_1', \xi,\eta)| \leq 2^5 A^{-\frac{3}{2}}d^{-1} N_1^3, \\
 &  |F (\xi_1', \eta_1', \xi,\eta)| \leq 2^5 A^{-1} d^{-1} N_1^2,
 \end{aligned} \ |\xi_1'+ \xi|  \leq 2^{11}A^{-\frac32}N_1 \right.
\right\}
\subset \mathcal{T}_{\tilde{k}}^{A, 2^{-20}d}.
\end{equation*}
Similarly to the proof of Lemma \ref{lemma3.6},
letting $\xi' = \xi + \xi_1'/2$, $\eta' = \eta + \eta_1'/2$, we see that
\begin{align*}
|\Phi(\xi_1', \eta_1', \xi, \eta) |& = |\xi_1' \xi(\xi_1' + \xi) + \eta_1' \eta (\eta_1' + \eta)|
\leq  2^5 A^{-\frac{3}{2}}d^{-1} N_1^3, \\
 |F (\xi_1', \eta_1', \xi,\eta)| & = |\xi_1' \eta +  \xi \eta_1' + 2 (\xi_1' \eta_1' + \xi \eta)|
\leq 2^5 A^{-1} d^{-1} N_1^2,
\end{align*}
are equivalent to
\begin{align}
\tilde{\Phi} (\xi' ,\eta') & := \Bigl| \xi_1' {\xi'}^2 + \eta_1' {\eta'}^2 - \frac{{\xi_1'}^3 + {\eta_1'}^3}{4} \Bigr|
\leq 2^5 A^{-\frac{3}{2}}d^{-1} N_1^3 , \label{info03-lemma3.23} \\
\tilde{F} (\xi' ,\eta') & := \Bigl| \frac{3}{2} \, \xi_1' \, \eta_1' + 2 \,  \xi'  \, \eta' \Bigr|
\leq 2^5 A^{-1} d^{-1} N_1^2,\label{info04-lemma3.23}
\end{align}
respectively. Note that $|\xi_1'+ \xi|  \leq 2^{11}A^{-\frac32}N_1$ yields
$|\xi'| \geq |\xi|/2 - |\xi_1' + \xi|/2 \geq 2^{-2}N_1$.
We observe that
\begin{equation}
\eqref{info04-lemma3.23} \Longrightarrow  \Bigl| \eta' + \frac{3  \xi_1' \, \eta_1'}{4 \xi'} \Bigr|  \leq \frac{2^4 A^{-1} d^{-1} N_1^2}{|\xi'|}
\leq 2^{6} A^{-1} d^{-1} N_1.\label{est003-lemma3.23}
\end{equation}
Since $|\eta_1'| \leq 2^{11} A^{-1} N_1$, \eqref{info03-lemma3.23} and \eqref{est003-lemma3.23} give
\begin{align}
& \Bigl| \xi_1' {\xi'}^2 + \eta_1' {\eta'}^2 - \frac{{\xi_1'}^3 + {\eta_1'}^3}{4} \Bigr|
\leq 2^5 A^{-\frac{3}{2}}d^{-1} N_1^3 \notag\\
 \xLongrightarrow[ \ ]{\eqref{est003-lemma3.23}} & \Bigl| \xi_1' {\xi'}^2 + \eta_1' \frac{9 {\xi_1'}^2 {\eta_1'}^2}{16 {\xi'}^2} - \frac{{\xi_1'}^3 + {\eta_1'}^3}{4} \Bigr|
\leq 2^{6} A^{-\frac{3}{2}}d^{-1} N_1^3.\label{est004-lemma3.23}
\end{align}
Let
\begin{equation*}
G(\xi') := \xi_1' {\xi'}^2 +  \frac{9 {\xi_1'}^2 {\eta_1'}^3}{16 {\xi'}^2} -
\frac{{\xi_1'}^3 + {\eta_1'}^3}{4}.
\end{equation*}
We deduce from $2^{-2} N_1 \leq |\xi_1'|$, $|\xi'| \leq 2 N_1$ and $|\eta_1'| \leq 2^{11} A^{-1} N_1$ that
\begin{equation*}
\Bigl| \Bigl( \frac{dG}{d\xi'} \Bigr)
 (\xi') \Bigr| = \Bigl| \frac{2 \xi_1'}{{\xi'}^3} \Bigl( {\xi'}^4 -  \frac{9 {\xi_1'} {\eta_1'}^3}{16}
\Bigr) \Bigr| \geq 2^{-5}N_1^2,
\end{equation*}
which, combined with \eqref{est004-lemma3.23}, establish that there exists a constant
$c(\xi_1',\eta_1') \in \R$ such that
\begin{equation*}
|\xi' - c (\xi_1', \eta_1')| \leq 2^{12} A^{-\frac{3}{2}}d^{-1} N_1.
\end{equation*}
This and \eqref{est003-lemma3.23} imply that there exists a constant $c'(\xi_1',\eta_1') \in \R$ such that
\begin{equation*}
|\eta' - c' (\xi_1', \eta_1')| \leq 2^{7} A^{-1}d^{-1} N_1,
\end{equation*}
which completes the proof.
\end{proof}
\begin{proof}[Proof of Proposition \ref{prop3.20}]
Similarly to the proof of Proposition \ref{prop3.19} for the case $|\xi_2| \geq 2^{10}A^{-1}K^{-1/2}N_1$,
we observe that $|\xi_2| \geq 2^{9} A^{-3/2} N_1$ implies
$|\Phi(\xi_1, \eta_1, \xi, \eta)| \gtrsim A^{-3/2} N_1^{3}$.
Since $A \leq N_1 L_0^{-1/3}$, this and Proposition \ref{prop3.18} verify the desired estimate \eqref{prop3.20-1}.
Therefore we may assume $|\xi_2| = |\xi_1 + \xi| \leq 2^{-9} A^{-3/2} N_1$, and then, by the definition of
$\widehat{Z}_{A,d}^{j_1,j}$ and $\overline{Z}_{A,d}^{j_1,j}$, we observe that
\begin{align*}
& \textnormal{(LHS) of \ref{prop3.20-1}}\\
\leq & \sum_{2^{20} \leq d \leq d_0} \sum_{(k_1, k) \in \widehat{Z}_{A,d}^{j_1,j}}
\Bigl|\int_{**}{  \ha{v}_{{N_2, L_2}}(\tau_2, \xi_2, \eta_2)
\ha{u}_{N_1, L_1}|_{\tilde{\mathcal{T}}_{k_1}^{A,d}}(\tau_1, \xi_1, \eta_1)
\ha{w}_{N_0, L_0}|_{\tilde{\mathcal{T}}_k^{A,d}}(\tau, \xi, \eta)
}
d\sigma_1 d\sigma \Bigr|\\
& + \sum_{(k_1, k) \in \widehat{Z}_{A,d_0}^{j_1,j}}
\Bigl|\int_{**}{  \ha{v}_{{N_2, L_2}}(\tau_2, \xi_2, \eta_2)
\ha{u}_{N_1, L_1}|_{\tilde{\mathcal{T}}_{k_1}^{A,d_0}}(\tau_1, \xi_1, \eta_1)
\ha{w}_{N_0, L_0}|_{\tilde{\mathcal{T}}_k^{A,d_0}}(\tau, \xi, \eta)
}
d\sigma_1 d\sigma \Bigr|\\
& =: \sum_{2^{20} \leq d \leq d_0} \sum_{(k_1, k) \in \widehat{Z}_{A,d}^{j_1,j}}  I_1 +
\sum_{(k_1, k) \in \widehat{Z}_{A,d_0}^{j_1,j}}  I_2.
\end{align*}
Here $d_0$ is defined as the minimal dyadic number such that
$d_0 \geq 2^{20} A^{-3/2} N_1^{3/2} L_0^{-1/2}$.
For the former term, it follows from Propositions \ref{prop3.21},
\ref{prop3.22} and Lemma \ref{lemma3.23} that
\begin{align*}
& \sum_{2^{20} \leq d \leq d_0} \sum_{(k_1, k) \in \widehat{Z}_{A,d}^{j_1,j}} I_1 \\
\lesssim & \sum_{2^{20} \leq d \leq d_0}
\sum_{(k_1, k) \in \widehat{Z}_{A,d}^{j_1,j}} A d^{\frac{1}{2}} N_1^{-2} (L_0 L_1 L_2)^{\frac{1}{2}} \|\ha{u}_{N_1, L_1}|_{\tilde{\mathcal{T}}_{k_1}^{A,d}} \|_{L^2}
\| \ha{v}_{N_2, L_2}\|_{L^2}
\|\ha{w}_{N_0, L_0}|_{\tilde{\mathcal{T}}_k^{A,d}}  \|_{L^2}\\
\lesssim & \sum_{2^{20} \leq d \leq d_0} A d^{\frac{1}{2}} N_1^{-2} (L_0 L_1 L_2)^{\frac{1}{2}} \|\ha{u}_{N_1, L_1} \|_{L^2}
\| \ha{v}_{N_2, L_2} \|_{L^2}
\|\ha{w}_{{N_0, L_0}} \|_{L^2}\\
\lesssim & \  A^{\frac{1}{4}} N_1^{-\frac{5}{4}} L_0^{\frac{1}{4}} (L_1 L_2)^{\frac{1}{2}}
\|\ha{u}_{N_1, L_1} \|_{L^2}
\| \ha{v}_{N_2, L_2} \|_{L^2}
\|\ha{w}_{{N_0, L_0}} \|_{L^2}.
\end{align*}
Next we deal with the latter term. We first consider the following bilinear Strichartz estimate.
\begin{equation}
\begin{split}
& \Bigl\| \chi_{ \tilde{\mathcal{T}}_{k}^{A,d_0}} \int \ha{v}_{N_2, L_2} (\tau_1+ \tau, \xi_1+\xi, \eta_1+ \eta)  \ha{u}_{N_1, L_1}|_{\tilde{\mathcal{T}}_{k_1}^{A,d_0}} (\tau_1, \xi_1, \eta_1) d \sigma_1 \Bigr\|_{L_{\xi, \eta, \tau}^2} \\
& \qquad \qquad \qquad \qquad \qquad
\lesssim ( A \, d_0  N_1 )^{-\frac{1}{2}} (L_1 L_2)^{\frac{1}{2}} \|\ha{v}_{N_2, L_2} \|_{L^2}
\|\ha{u}_{N_1, L_1}|_{\tilde{\mathfrak{D}}_{j_1}^A} \|_{L^2},
\end{split}\label{bilinearStrichartz-prop3.20}
\end{equation}
which is established by the same argument as in the proof of \eqref{bilinearStrichartz-6} in Proposition \ref{prop3.12}. Thus, to avoid redundancy, we omit the proof of \eqref{bilinearStrichartz-prop3.20}.
By Lemma \ref{lemma3.23} and \eqref{bilinearStrichartz-prop3.20}, we obtain
\begin{align*}
& \sum_{(k_1, k) \in \widehat{Z}_{A,d_0}^{j_1,j}}  I_2 \\
\lesssim & \sum_{(k_1, k) \in \widehat{Z}_{A,d_0}^{j_1,j}}
( A \, d_0  N_1 )^{-\frac{1}{2}} (L_1 L_2)^{\frac{1}{2}}
\|\ha{u}_{N_1, L_1}|_{\tilde{\mathcal{T}}_{k_1}^{A,d}} \|_{L^2}
\| \ha{v}_{N_2, L_2}\|_{L^2}
\|\ha{w}_{N_0, L_0}|_{\tilde{\mathcal{T}}_k^{A,d}}  \|_{L^2}\\
\lesssim & \  A^{\frac{1}{4}} N_1^{-\frac{5}{4}} L_0^{\frac{1}{4}} (L_1 L_2)^{\frac{1}{2}}
\|\ha{u}_{N_1, L_1} \|_{L^2}
\| \ha{v}_{N_2, L_2} \|_{L^2}
\|\ha{w}_{{N_0, L_0}} \|_{L^2}.
\end{align*}
This completes the proof.
\end{proof}
\begin{proof}[Proof of Proposition \ref{prop3.17}]
Without loss of generality, we may assume that
$\ha{w}_{{N_0, L_0}}$, $\ha{u}_{N_1, L_1}$,
$\ha{v}_{N_2, L_2} $ are non-negative.
We define that
\begin{equation*}
J_{A}^{\mathcal{I}_1} = \{ (j_1, j) \, | \, 0 \leq j_1,j \leq A -1, \ ( {\mathfrak{D}}_{j_1}^A \times {\mathfrak{D}}_{j}^A ) \subset ( \mathfrak{D}_{0}^{2^{11}} \times
\mathfrak{D}_{0}^{2^{11}} )\}.
\end{equation*}
We divide the proof into the three cases.\\
(I) $\ 1 \leq N_2 \lesssim L_0^{1/3},$ $ \ $
\textnormal{(I \hspace{-0.15cm}I)} $N_2 \gg L_0^{1/3}, \ |\xi_2| \geq 2^{-10}N_2$, $ \ $
\textnormal{(I \hspace{-0.16cm}I \hspace{-0.16cm}I)}  $N_2 \gg L_0^{1/3}, \ |\xi_2| \leq 2^{-10} N_2$.

For the first case, we perform the decomposition as
\begin{equation*}
\mathfrak{D}_{0}^{2^{11}} \times
\mathfrak{D}_{0}^{2^{11}} =    \bigcup_{\tiny{(j_1,j) \in J_{N_1/N_2}^{\mathcal{I}_1}}}
{\mathfrak{D}}_{j_1}^{N_1/N_2} \cross {\mathfrak{D}}_{j}^{N_1/N_2}.
\end{equation*}
It is clear that we may assume that $j_1$ and $j$ satisfy that $|j_1-j| \leq 2$ here.
Therefore, by \eqref{bilinearStrichartz-12} in Proposition \ref{prop3.18}, we get
\begin{align*}
& \textnormal{(LHS) of \eqref{prop3.17-1}} \\
\lesssim & \sum_{\tiny{\substack{(j_1,j) \in J_{N_1/N_2}^{\mathcal{I}_1}\\ |j_1 - j|\leq 2}}}
N_1^{-1} N_2^{\frac{1}{2}}  (L_1 L_2)^{\frac{1}{2}}
\|\ha{u}_{N_1, L_1}|_{\tilde{\mathfrak{D}}_{j_1}^{N_1/N_2}} \|_{L^2}
\| \ha{v}_{N_2, L_2}\|_{L^2}
\|\ha{w}_{{N_0, L_0}}|_{\tilde{\mathfrak{D}}_{j}^{N_1/N_2}}  \|_{L^2}\\
\lesssim & \, N_1^{-1} N_2^{\frac{1}{2}}  (L_1 L_2)^{\frac{1}{2}}
\|\ha{u}_{N_1, L_1}\|_{L^2}
\| \ha{v}_{N_2, L_2}\|_{L^2}
\|\ha{w}_{{N_0, L_0}}  \|_{L^2}\\
\lesssim & \, N_1^{-1} N_2^{-\frac{1}{4}} L_0^{\frac{1}{4}} (L_1 L_2)^{\frac{1}{2}}
\|\ha{u}_{N_1, L_1}\|_{L^2}
\| \ha{v}_{N_2, L_2}\|_{L^2}
\|\ha{w}_{{N_0, L_0}}  \|_{L^2}\\
= & \, \textnormal{(RHS) of \eqref{prop3.17-1}}.
\end{align*}
Here we used $N_2^{1/2} \lesssim N_2^{-1/4} L_0^{1/4}$ which is equivalent to $N_2 \lesssim L_0^{1/3}$.

Next we consider the case $N_2 \gg L_0^{1/3}, \ |\xi_2| \geq 2^{-10}N_2$.
In this case we immediately obtain $|\Phi(\xi_1, \eta_1, \xi, \eta)| \gtrsim N_1^{2} N_2$.
Indeed, we get
\begin{align*}
|\Phi(\xi_1, \eta_1, \xi, \eta)| & \geq |\xi_1 \xi (\xi_1+ \xi)| - |\eta_1 \eta (\eta_1 + \eta)|\\
& \geq 2^{-12} N_1^2 N_2 - 2^{-20} N_1^2 N_2\\
& \gtrsim N_1^2 N_2.
\end{align*}
Therefore, it suffices to show the following bilinear Strichartz estimate.
\begin{align}
& \Bigl\| \chi_{G_{N_2, L_2}} \int \ha{u}_{N_1, L_1}|_{\tilde{\mathfrak{D}}_{j_1}^{\frac{N_1}{N_2}}} (\tau_1, \xi_1, \eta_1) \ha{w}_{N_0, L_0}|_{\tilde{\mathfrak{D}}_{j}^{\frac{N_1}{N_2}}} (\tau_2 - \tau_1, \xi_2-\xi_1, \eta_2- \eta_1) d\sigma_1
\Bigr\|_{L_{\xi_2, \eta_2, \tau_2}^2} \notag \\
& \qquad \qquad \qquad \qquad \quad
\lesssim  N_1^{-\frac{1}{2}} (L_0 L_1)^{\frac{1}{2}} \|\ha{u}_{N_1, L_1}|_{\tilde{\mathfrak{D}}_{j_1}^A}\|_{L^2}
\|\ha{w}_{N_0, L_0}|_{\tilde{\mathfrak{D}}_{j}^A} \|_{L^2},\label{prop3.17-bilinear01}
\end{align}
where $(j_1,j) \in J_{N_1/N_2}^{\mathcal{I}_1}$, $|j_1-j| \leq 2$ and $|\xi_2| \geq 2^{-10} N_2$.
\eqref{prop3.17-bilinear01} can be given by the estimate \vspace{-1mm}
\begin{equation}
\sup_{\substack{(\tau_2, \xi_2, \eta_2) \in G_{N_2, L_2} \\
|\xi_2| \geq  2^{-10} N_2}}|E(\tau_2, \xi_2, \eta_2)| \lesssim N_1^{-1} L_0 L_1 \label{est01-prop3.17}\vspace{-2mm}
\end{equation}
where
\begin{equation*}
E(\tau_2, \xi_2, \eta_2) := \bigl\{ (\tau_1, \xi_1, \eta_1) \in G_{N_1, L_1} \cap
\tilde{\mathfrak{D}}_{j_1}^{\frac{N_1}{N_2}}
\, | \, (\tau_2-\tau_1, \xi_2- \xi_1, \eta_2-\eta_1) \in G_{N_0,L_0} \cap
\tilde{\mathfrak{D}}_{j}^{\frac{N_1}{N_2}} \bigr\}.
\end{equation*}
In the same manner as in the proof of Proposition \ref{prop3.18}, $ |\partial_{\xi_1}\Phi(\xi_1,\eta_1,-\xi_2,-\eta_2)|  =
|\xi_2 \, (2\xi_1- \xi_2)| \gtrsim N_1 N_2$ yields \eqref{est01-prop3.17}.
By utilizing the bilinear Strichartz estimates \eqref{bilinearStrichartz-11}, \eqref{bilinearStrichartz-12},
\eqref{prop3.17-bilinear01} and $|\Phi(\xi_1, \eta_1, \xi, \eta)| \gtrsim N_1^{2} N_2$, we easily establish
\eqref{prop3.17-1}.

Lastly, we consider the case $N_2 \gg L_0^{1/3}, \ |\xi_2| \leq 2^{-10}N_2$.
Let $A_0$ be the maximal dyadic number which satisfies $A_0 \leq 2^{10} N_1 /N_2$.
Note that $N_2 \gg L_0^{\frac{1}{3}}$ implies $A_0\leq N_1/L_0^{1/3}$.
By the Whitney type decomposition of angular variables, we have
\begin{equation*}
\mathfrak{D}_{0}^{2^{11}} \times
\mathfrak{D}_{0}^{2^{11}} =   \bigcup_{64 \leq A \leq A_0} \ \bigcup_{\tiny{\substack{(j_1,j) \in J_{A}^{\mathcal{I}_1}\\ 16 \leq |j_1 - j|\leq 32}}}
{\mathfrak{D}}_{j_1}^A \cross {\mathfrak{D}}_{j}^A \cup
\bigcup_{\tiny{\substack{(j_1,j) \in J_{A_0}^{\mathcal{I}_1}\\|j_1 - j|\leq 16}}}
{\mathfrak{D}}_{j_1}^{A_0} \cross {\mathfrak{D}}_{j}^{A_0}.
\end{equation*}\vspace{-2mm}
We can observe that
\begin{align*}
 (\xi_1,\eta_1) \times (\xi,\eta)
\in \bigcup_{64 \leq A \leq 2^{-10}A_0} \ \bigcup_{\tiny{\substack{(j_1,j) \in J_{A}^{\mathcal{I}_1}\\ 16 \leq |j_1 - j|\leq 32}}}
{\mathfrak{D}}_{j_1}^A \cross {\mathfrak{D}}_{j}^A
& \Longrightarrow |(\xi_1+\xi, \eta_1+\eta)| \geq 2^3 N_2,\\
 (\xi_1,\eta_1) \times (\xi,\eta)
\in \bigcup_{\tiny{\substack{(j_1,j) \in J_{A_0}^{\mathcal{I}_1}\\|j_1 - j|\leq 16}}}
{\mathfrak{D}}_{j_1}^{A_0} \cross {\mathfrak{D}}_{j}^{A_0} \ \ \textnormal{and} \ \
 |(\xi_1+ & \xi,  \eta_1+\eta)|  \geq  N_2 \\[-8pt]
&  \Longrightarrow |\xi_1+\xi|  \geq 2^{-10} N_2.
\end{align*}\vspace{-2mm}
Therefore, here we may assume that
\begin{equation*}
\operatorname{supp}  \ha{u}_{N_1, L_1} \times \operatorname{supp}   \ha{w}_{N_0, L_0}
\subset \bigcup_{2^{-10}A_0 \leq A \leq A_0} \ \bigcup_{\tiny{\substack{(j_1,j) \in J_{A}^{\mathcal{I}_1}\\ 16 \leq |j_1 - j|\leq 32}}}
{\tilde{\mathfrak{D}}}_{j_1}^A \cross \tilde{{\mathfrak{D}}}_{j}^A.
\end{equation*}
By Propositions \ref{prop3.19} and \ref{prop3.20}, we get
\begin{align*}
& \textnormal{(LHS) of \eqref{prop3.17-1}} \\
\leq & \sum_{2^{-10}A_0 \leq A \leq A_0} \sum_{\tiny{\substack{(j_1,j) \in J_{A}^{\mathcal{I}_1}\\ 16 \leq  |j_1 - j|\leq 32}}}
\Bigl|\int_{**}{  \ha{v}_{{N_2, L_2}}(\tau_2, \xi_2, \eta_2)
\ha{u}_{N_1, L_1}|_{\tilde{\mathfrak{D}}_{j_1}^A}(\tau_1, \xi_1, \eta_1)
\ha{w}_{N_0, L_0}|_{\tilde{\mathfrak{D}}_{j}^A}(\tau, \xi, \eta)
}
d\sigma_1 d\sigma \Bigr|\\
 \lesssim & \sum_{2^{-10}A_0 \leq A \leq A_0}
\sum_{\tiny{\substack{(j_1,j) \in J_{A}^{\mathcal{I}_1}\\ 16 \leq  |j_1 - j|\leq 32}}}
A^{\frac{1}{4}} N_1^{-\frac{5}{4}} L_0^{\frac{1}{4}} (L_1 L_2)^{\frac{1}{2}} \|\ha{u}_{N_1, L_1}|_{\tilde{\mathfrak{D}}_{j_1}^A} \|_{L^2}
\| \ha{v}_{N_2, L_2}\|_{L^2}
\|\ha{w}_{{N_0, L_0}}|_{\tilde{\mathfrak{D}}_{j}^A}  \|_{L^2}\\
\lesssim & \sum_{2^{-10}A_0 \leq A \leq A_0}
A^{\frac{1}{4}} N_1^{-\frac{5}{4}} L_0^{\frac{1}{4}} (L_1 L_2)^{\frac{1}{2}} \|\ha{u}_{N_1, L_1} \|_{L^2}
\| \ha{v}_{N_2, L_2} \|_{L^2}
\|\ha{w}_{{N_0, L_0}} \|_{L^2}\\
\lesssim & \ \textnormal{(RHS) of \eqref{prop3.17-1}} .
\end{align*}
This completes the proof.
\end{proof}
Lastly, we establish \eqref{desired-est-12-23} for the case $\min(N_0, N_1, N_2) = N_0$.
Similarly to the case $\min(N_0, N_1, N_2) = N_2$, we divide the proof of \eqref{desired-est-12-23} into the three cases.\\
(I)' $(\xi_1,\eta_1)\times (\xi_2,\eta_2) \in \mathcal{I}_1 $, $ \ $
\textnormal{(I \hspace{-0.15cm}I)'} $(\xi_1,\eta_1)\times (\xi_2,\eta_2) \in \mathcal{I}_2$, $ \ $
\textnormal{(I \hspace{-0.16cm}I \hspace{-0.16cm}I)'}
$(\xi_1,\eta_1)\times (\xi_2,\eta_2) \in \mathcal{I}_3$.

Note that the cases (I)' and \textnormal{(I \hspace{-0.16cm}I \hspace{-0.16cm}I)'} can be treated in the same manners as that for the case $\min(N_0, N_1, N_2) = N_2$.
Therefore, we omit the proofs for the cases (I)' and \\
\textnormal{(I \hspace{-0.16cm}I \hspace{-0.16cm}I)'},
and focus on the case \textnormal{(I \hspace{-0.16cm}I)'}. The estimate \eqref{desired-est-12-23}
for the case $\min(N_0, N_1, N_2) = N_0$ can be obtained by the following proposition.
\begin{prop}\label{prop3.25}
Assume \textnormal{(i)}, \textnormal{(ii)'} in \textit{Case} \textnormal{\ref{case-2}}
and $\min(N_0, N_1, N_2) = N_0$. Then we have
\begin{align}
&
\Bigl|\int_{*}{ |\xi+\eta|  \ha{w}_{N_0, L_0}(\tau, \xi, \eta)  \chi_{\mathcal{I}_2}((\xi_1,\eta_1), (\xi_2,\eta_2))
\ha{u}_{N_1, L_1}(\tau_1, \xi_1, \eta_1)
 \ha{v}_{{N_2, L_2}}(\tau_2, \xi_2, \eta_2)
}
d\sigma_1 d\sigma_2 \Bigr| \notag \\
& \qquad \qquad \qquad \qquad
\lesssim N_0^{\frac{1}{4}} N_1^{-\frac{1}{2}} (L_0 L_1 L_2)^{\frac{1}{2}} \|\ha{u}_{N_1, L_1}\|_{L^2}
\| \ha{v}_{N_2, L_2}\|_{L^2}
\|\ha{w}_{{N_0, L_0}}\|_{L^2},\label{prop3.25-1}
\end{align}
where $d \sigma_j = d\tau_j d \xi_j d \eta_j$ and $*$ denotes $(\tau, \xi, \eta) = (\tau_1 + \tau_2, \xi_1+ \xi_2, \eta_1 + \eta_2).$
\end{prop}
We can show Proposition \eqref{prop3.25} by the same argument as that for the proof of Proposition \ref{prop3.14}.
We introduce the two propositions which are easily verified in the same ways as that for Propositions
\ref{prop3.15} and \ref{prop3.16}, respectively. Thus, we omit the proofs.
\begin{prop}\label{prop3.26}
Assume \textnormal{(i)}, \textnormal{(ii)'} in \textit{Case} \textnormal{\ref{case-2}}
and $\min(N_0, N_1, N_2) = N_0.$ Let $M\geq 2^{12}$ and $A \geq 2^{20} M$ be dyadic, $|j_1-j_2| \leq 32$ and
\begin{equation*}
\left( {\mathfrak{D}}_{j_1}^A \times {\mathfrak{D}}_{j_2}^A \right) \subset \mathcal{I}_{2}^M.
\end{equation*}
If we assume that $||(\xi_1, \eta_1)|- |(\xi-\xi_1, \eta-\eta_1)|| \geq 2^{-3} N_1$, then we have
\begin{align}
& \Bigl\| \chi_{G_{N_0,L_0} } \int  \ha{u}_{N_1, L_1}|_{\tilde{\mathfrak{D}}_{j_1}^A} (\tau_1, \xi_1, \eta_1)
\ha{v}_{N_2, L_2}|_{\tilde{\mathfrak{D}}_{j_2}^A} (\tau- \tau_1, \xi - \xi_1, \eta - \eta_1)
d \sigma_1 \Bigr\|_{L_{\xi, \eta, \tau}^2} \notag \\
& \qquad \qquad \qquad \qquad
\lesssim A^{-\frac{1}{2}} M^{\frac{1}{2}} N_1^{-\frac{1}{2}} (L_1 L_2)^{\frac{1}{2}}
\|\ha{u}_{N_1, L_1}|_{\tilde{\mathfrak{D}}_{j_1}^A} \|_{L^2}
\|\ha{v}_{N_2, L_2}|_{\tilde{\mathfrak{D}}_{j_2}^A} \|_{L^2}
.\label{bilinearStrichartz-prop3.26-01}
\end{align}
Assume that $||(\xi_1, \eta_1)|- |(\xi_1+\xi_2, \eta_1+ \eta_2)|| \geq 2^{-3} N_1$.
Then we have
\begin{align}
& \Bigl\| \chi_{G_{N_2, L_2} \cap \tilde{\mathfrak{D}}_{j_2}^A}
 \int \ha{u}_{N_1, L_1}|_{\tilde{\mathfrak{D}}_{j_1}^A} (\tau_1, \xi_1, \eta_1) \ha{w}_{N_0, L_0}
(\tau_1+\tau_2, \xi_1+\xi_2, \eta_1+ \eta_2)) d\sigma_1
\Bigr\|_{L_{\xi_2, \eta_2, \tau_2}^2} \notag \\
& \qquad \qquad \qquad \qquad
\lesssim A^{-\frac{1}{2}} M^{\frac{1}{2}} N_1^{-\frac{1}{2}} (L_0 L_1)^{\frac{1}{2}} \|\ha{u}_{N_1, L_1}|_{\tilde{\mathfrak{D}}_{j_1}^A}\|_{L^2}
\|\ha{w}_{N_0, L_0} \|_{L^2}.\label{bilinearStrichartz-prop3.26-02}
\end{align}
Similarly, if we assume that $||(\xi_2, \eta_2)|- |(\xi_1+\xi_2, \eta_1+\eta_2)|| \geq 2^{-3} N_1$,
then we have
\begin{align}
& \Bigl\| \chi_{G_{N_1, L_1} \cap \tilde{\mathfrak{D}}_{j_1}^A} \int
\ha{v}_{N_2, L_2}|_{\tilde{\mathfrak{D}}_{j_2}^A} (\tau_2, \xi_2, \eta_2)
\ha{w}_{N_0, L_0}
(\tau_1+\tau_2, \xi_1+\xi_2, \eta_1+ \eta_2)   d\sigma_2 \Bigr\|_{L_{\xi_1, \eta_1, \tau_1}^2}
\notag \\
& \qquad \qquad \qquad \qquad
\lesssim A^{-\frac{1}{2}} M^{\frac{1}{2}} N_1^{-\frac{1}{2}} (L_0 L_2)^{\frac{1}{2}}
\|\ha{v}_{N_2, L_2}|_{\tilde{\mathfrak{D}}_{j_2}^A}\|_{L^2}
\|\ha{w}_{N_0, L_0} \|_{L^2} . \label{bilinearStrichartz-prop3.26-03}
\end{align}
\end{prop}
\begin{prop}\label{prop3.27}
Assume \textnormal{(i)}, \textnormal{(ii)'} in \textit{Case} \textnormal{\ref{case-2}}
and $\min(N_0, N_1, N_2) = N_0.$ Let $A$ and $M$ be dyadic numbers such that $2^{12} \leq M \leq N_1$ and $A \geq \max (2^{25}, M)$. Suppose that
$j_1$, $j_2$ satisfy $16 \leq |j_1-j_2| \leq 32$ and
\begin{equation*}
( {\mathfrak{D}}_{j_1}^A \times {\mathfrak{D}}_{j_2}^A ) \subset \mathcal{I}_{2}^M.
\end{equation*}
Then we have
\begin{equation}
\begin{split}
&
\Bigl|\int_{*}{ \ha{w}_{N_0, L_0}(\tau, \xi, \eta)
\ha{u}_{N_1, L_1}|_{\tilde{\mathfrak{D}}_{j_1}^A}(\tau_1, \xi_1, \eta_1)
\ha{v}_{{N_2, L_2}}|_{\tilde{\mathfrak{D}}_{j_2}^A} (\tau_2, \xi_2, \eta_2)
}
d\sigma_1 d\sigma_2 \Bigr|\\
& \qquad \qquad  \lesssim  (AM)^{\frac{1}{2}} N_1^{-2} (L_0 L_1 L_2)^{\frac{1}{2}} \|\ha{u}_{N_1, L_1}|_{\tilde{\mathfrak{D}}_{j_1}^A} \|_{L^2}
\| \ha{v}_{{N_2, L_2}}|_{\tilde{\mathfrak{D}}_{j_2}^A}\|_{L^2}
\|\ha{w}_{{N_0, L_0}}  \|_{L^2},\label{prop3.27-1}
\end{split}
\end{equation}
where $d \sigma_j = d\tau_j d \xi_j d \eta_j$ and $*$ denotes $(\tau, \xi, \eta) = (\tau_1 + \tau_2, \xi_1+ \xi_2, \eta_1 + \eta_2).$
\end{prop}
\begin{proof}[Proof of Proposition \ref{prop3.25}]
First, since $|\xi+\eta| \lesssim N_0$, it suffices to show
\begin{align}
&
\Bigl|\int_{*}{ |\xi+\eta|^{\frac{3}{4}}  \ha{w}_{N_0, L_0}(\tau, \xi, \eta)  \chi_{\mathcal{I}_2}((\xi_1,\eta_1), (\xi_2,\eta_2))
\ha{u}_{N_1, L_1}(\tau_1, \xi_1, \eta_1)
 \ha{v}_{{N_2, L_2}}(\tau_2, \xi_2, \eta_2)
}
d\sigma_1 d\sigma_2 \Bigr| \notag \\
& \qquad \qquad \qquad \qquad
\lesssim  N_1^{-\frac{1}{2}} (L_0 L_1 L_2)^{\frac{1}{2}} \|\ha{u}_{N_1, L_1}\|_{L^2}
\| \ha{v}_{N_2, L_2}\|_{L^2}
\|\ha{w}_{{N_0, L_0}}\|_{L^2}.\label{prop3.25-2}
\end{align}
Let $A_0 \geq 2^{25}$ and $M_0 \geq 2^{11}$ be dyadic which will be chosen later. Recall
\begin{equation*}
J_{A}^{\mathcal{I}_2^M} = \{ (j_1, j_2) \, | \, 0 \leq j_1,j_2 \leq A -1, \ \left( {\mathfrak{D}}_{j_1}^A \times {\mathfrak{D}}_{j_2}^A \right) \subset \mathcal{I}_2^M\}
\end{equation*}
and the decomposition of $\mathcal{I}_{2}^M$
\begin{equation*}
\mathcal{I}_2^M =   \bigcup_{64 \leq A \leq A_0} \ \bigcup_{\tiny{\substack{(j_1,j_2) \in J_{A}^{\mathcal{I}_2^M}\\ 16 \leq |j_1 - j_2|\leq 32}}}
{\mathfrak{D}}_{j_1}^A \cross {\mathfrak{D}}_{j_2}^A \cup
\bigcup_{\tiny{\substack{(j_1,j_2) \in J_{A_0}^{\mathcal{I}_2^M}\\|j_1 - j_2|\leq 16}}}
{\mathfrak{D}}_{j_1}^{A_0} \cross {\mathfrak{D}}_{j_2}^{A_0},
\end{equation*}
and the decompositionof $\mathcal{I}_{2}$
\begin{equation*}
\mathcal{I}_{2} = \bigcup_{2^{11} \leq M < M_0} \mathcal{I}_{2}^M \, \cup \,
\bigl( {\mathfrak{D}}_{{2^{-2}M_0 \times 3}}^{M_0} \times {\mathfrak{D}}_{2^{-2}M_0 \times 3}^{M_0} \bigr).
\end{equation*}
By using the notation
\begin{equation*}
I_A:= \Bigl|\int_{*}{ \ha{w}_{N_0, L_0}(\tau, \xi, \eta)
\ha{u}_{N_1, L_1}|_{\tilde{\mathfrak{D}}_{j_1}^A}(\tau_1, \xi_1, \eta_1)
\ha{v}_{{N_2, L_2}}|_{\tilde{\mathfrak{D}}_{j_2}^A} (\tau_2, \xi_2, \eta_2)
}
d\sigma_1 d\sigma_2 \Bigr|,
\end{equation*}
we observe that
\begin{align*}
& \textnormal{(LHS) of \eqref{prop3.25-2}} \\
\lesssim & \sum_{2^{11} \leq M \leq M_0}  \sum_{2^{25} \leq A \leq A_0}
\sum_{\tiny{\substack{(j_1,j_2) \in J_{A}^{\mathcal{I}_2^M}\\ 16 \leq |j_1 - j_2|\leq 32}}}
N_1^{\frac{3}{4}} M^{-\frac{3}{4}} I_A
+ \sum_{2^{11} \leq M \leq M_0} \sum_{\tiny{\substack{(j_1,j_2) \in J_{A_0}^{\mathcal{I}_2^M}\\
|j_1 - j_2|\leq 16}}}
N_1^{\frac{3}{4}} M^{-\frac{3}{4}} I_{A_0}  \\
 + & N_1^{\frac{3}{4}} M_0^{-\frac{3}{4}}
\Bigl|\int_{*}{ \ha{w}_{N_0, L_0}(\tau, \xi, \eta)
\ha{u}_{N_1, L_1}|_{\tilde{\mathfrak{D}}_{{3M_0/4}}^{M_0} }(\tau_1, \xi_1, \eta_1)
\ha{v}_{{N_2, L_2}}|_{\tilde{\mathfrak{D}}_{{3M_0/4}}^{M_0} } (\tau_2, \xi_2, \eta_2)
}
d\sigma_1 d\sigma_2 \Bigr|.
\end{align*}
Here we used the inequality $|\xi+\eta| \lesssim M^{-1} N_1$. For the first term, by Proposition \ref{prop3.27}, we get
\begin{align*}
& \sum_{2^{25} \leq A \leq A_0}
\sum_{\tiny{\substack{(j_1,j_2) \in J_{A}^{\mathcal{I}_2^M}\\ 16 \leq |j_1 - j_2|\leq 32}}}
N_1^{\frac{3}{4}} M^{-\frac{3}{4}} I_A \\
\lesssim & \sum_{2^{25} \leq A \leq A_0}
\sum_{\tiny{\substack{(j_1,j_2) \in J_{A}^{\mathcal{I}_2^M}\\ 16 \leq |j_1 - j_2|\leq 32}}}
A^{\frac{1}{2}} M^{-\frac{1}{4}}N_1^{-\frac{5}{4}} (L_0 L_1 L_2)^{\frac{1}{2}} \|\ha{u}_{N_1, L_1}|_{\tilde{\mathfrak{D}}_{j_1}^A} \|_{L^2}
\| \ha{v}_{{N_2, L_2}}|_{\tilde{\mathfrak{D}}_{j_2}^A}\|_{L^2}
\|\ha{w}_{{N_0, L_0}}  \|_{L^2}\\
\lesssim & \ {A_0}^{\frac{1}{2}} M^{-\frac{1}{4}}N_1^{-\frac{5}{4}}
(L_0 L_1 L_2)^{\frac{1}{2}} \|\ha{u}_{N_1, L_1} \|_{L^2}
\| \ha{v}_{N_2, L_2}\|_{L^2}
\|\ha{w}_{{N_0, L_0}} \|_{L^2}.
\end{align*}
Thus, if we define $A_0$ as the minimal dyadic number which is greater than $N_1^{3/2}$, the first term is bounded by (RHS) of \eqref{prop3.25-2}. Next, we easily see
\begin{equation*}
\max(| \, |(\xi_1,\eta_1)| - |(\xi_2,  \eta_2)| \, |, \,
| \, |(\xi_1,\eta_1)| - |(\xi, \eta)| \, |, \,
| \, |(\xi,\eta)| - |(\xi_2, \eta_2)| \, |) \geq 2^{-3} N_1,
\end{equation*}
where $(\xi,\eta)=(\xi_1+\xi_2, \eta_1+\eta_2)$ and $(\xi_1,\eta_1)\times (\xi_2,\eta_2) \in
{\mathfrak{D}}_{j_1}^{A_0} \cross {\mathfrak{D}}_{j_2}^{A_0}$ with $|j_1-j_2| \leq 16$.
Thus, by using one of \eqref{bilinearStrichartz-prop3.26-01}-\eqref{bilinearStrichartz-prop3.26-03}
in Proposition \ref{prop3.26}, we get
\begin{align*}
& \sum_{2^{11} \leq M \leq M_0} \sum_{\tiny{\substack{(j_1,j_2) \in J_{A_0}^{\mathcal{I}_2^M}\\|j_1 - j_2|\leq 16}}}
N_1^{\frac{3}{4}} M^{-\frac{3}{4}} I_{A_0}  \\
\lesssim &\sum_{2^{11} \leq M \leq M_0} \sum_{\tiny{\substack{(j_1,j_2) \in J_{A_0}^{\mathcal{I}_2^M}\\
|j_1 - j_2|\leq 16}}}
{A_0}^{-\frac{1}{2}} M^{-\frac{1}{4}}N_1^{\frac{1}{4}} (L_0 L_1 L_2)^{\frac{1}{2}} \|\ha{u}_{N_1, L_1}|_{\tilde{\mathfrak{D}}_{j_1}^A} \|_{L^2}
\| \ha{v}_{{N_2, L_2}}|_{\tilde{\mathfrak{D}}_{j_2}^A}\|_{L^2}
\|\ha{w}_{{N_0, L_0}}  \|_{L^2}\\
\lesssim & \ \textnormal{(RHS) of \eqref{prop3.25-2}}.
\end{align*}
Let $M_0 = N_1$. By the Strichartz estimate \eqref{Strichartz-1} with $p=q=4$, we can find that the last term is bounded by \textnormal{(RHS) of \eqref{prop3.25-2}}.
\end{proof}
\section{Negative Result}
In this section, we show the following negative result.
\begin{thm}\label{not-c2-re}
Let $s < - \frac{1}{4}$. Then for any $T>0$, the data-to-solution map
$ u_0 \mapsto u$
of \eqref{ZK'}, as a map from the unit ball in
$H^s(\R^2)$ to
$C([0,T]; H^{s})$ fails to be $C^2$.
\end{thm}
\begin{proof}
Let $N \geq 1$ be a dyadic number.
We define the two vectors $\vec{v}$, $\vec{v}^{\bot}$ in $\R^2$ as
\begin{equation*}
\vec{v} =(3 \sqrt[3]{9}, \sqrt[3]{100}), \quad \vec{v}^{\bot} =(-\sqrt[3]{100}, 3 \sqrt[3]{9}),
\end{equation*}
and the three points $\alpha$, $\beta$, $\gamma \in \R^2$ as
\begin{equation*}
\alpha = ( \sqrt[3]{2} N, \, \sqrt[3]{75} N ), \quad \beta = \Bigl( - 3 \sqrt[3]{2} N, \, - \frac{\sqrt[3]{75}}{5} N \Bigr), \quad \gamma = \alpha + \beta.
\end{equation*}
Clearly, $\vec{v} \, \bot \, \vec{v}^{\bot}$.
The rectangle sets $A$, $B$ and $C \subset \R^2$ are defined as follows:
\begin{align*}
A & = \bigl\{ (\xi,\eta) \in \R^2 \,
| \, (\xi , \eta)= \alpha +  a N^{- \frac{1}{2}} \vec{v}
+   b N^{-2}  \vec{v}^{\bot}, \, -1 < a, b< 1 \bigr\}, \\
B & = \bigl\{ (\xi,\eta) \in \R^2 \,
| \, (\xi , \eta)= \beta + a N^{- \frac{1}{2}} \vec{v}
+ b N^{-2}   \vec{v}^{\bot}, \, -1 < a, b< 1 \bigr\}, \\
C & = \bigl\{ (\xi,\eta) \in \R^2 \,
| \, (\xi , \eta)= \gamma +  a N^{- \frac{1}{2}} \vec{v}
+b N^{-2}   \vec{v}^{\bot}  , \, -2^{-2} < a, b< 2^{-2} \bigr\}.
\end{align*}
We easily see $\mu(A) = \mu(B) \sim \mu(C) \sim N^{-5/2}$ where $\mu$ is the usual Lebesgue measure.
We set $\varphi$ as
\begin{equation*}
 (\mathcal{F}_{x,y} \varphi) (\xi, \eta) := N^{-s+\frac{5}{4}} (\chi_{A} (\xi, \eta) + \chi_A (-\xi, -\eta)  + \chi_{B} (\xi, \eta) + \chi_B (-\xi, -\eta)).
\end{equation*}
Note that $\varphi$ is a real-valued function and satisfies $\| \varphi\|_{H^s} \sim 1$. Indeed, since $\mathcal{F}_{x,y}\varphi(\xi,\eta) = \mathcal{F}_{x,y}\varphi(-\xi,-\eta)$ and $\mathcal{F}_{x,y}\varphi$ is real, we get
\begin{equation*}
\overline{\varphi}(x,y) = \overline{\mathcal{F}^{-1}_{\xi,\eta} ((\mathcal{F}_{x,y}\varphi)(\xi,\eta)) (x,y)} =
\mathcal{F}^{-1}_{\xi,\eta} ((\mathcal{F}_{x,y}\varphi)(-\xi,-\eta)) (x,y) = \varphi(x,y).
\end{equation*}
Let $\xi$, $\eta$, $\xi_1$, $\eta_1 \in \mathbb{R}$ satisfy $(\xi_1, \eta_1) \in A$ and
$(\xi-\xi_1, \eta-\eta_1) \in B$. By a simple calculation, we observe that
$|\Phi(\xi_1,\eta_1, -\xi, -\eta)| = |\xi \xi_1 (\xi - \xi_1) + \eta \eta_1 (\eta -  \eta_1)| \lesssim 1$. Thus,
if we assume that $T>0$ is sufficient small, then for any $0< t \leq T$, we have
\begin{equation*}
\Re (e^{-it \Phi(\xi_1,\eta_1, -\xi, -\eta) } )\geq 1/2 \quad \textnormal{for any} \
(\xi_1, \eta_1) \times (\xi-\xi_1, \eta-\eta_1) \in A \times B,
\end{equation*}
where $\Re (z)$ denotes the real part of $z \in \C$.
Consequently,
by Plancherel's theorem, we obtain
\begin{align*}
&  \Bigl\| \int_0^t{e^{-(t-t')(\partial_\xi^3 + \partial_\eta^3)}  (\partial_1 + \partial_2) \Bigl(
(e^{-t'(\partial_\xi^3 + \partial_\eta^3)}\varphi) \, (e^{-t'(\partial_\xi^3 + \partial_\eta^3)} \varphi) \Bigr) }
d t'\Bigr\|_{H^{s}}\\
\gtrsim & N^{-  s + \frac{7}{2}}
\Bigl\|  \int_0^t \int_{\R^2}{e^{-3it'\Phi(\xi_1,\eta_1, -\xi, -\eta)}
\mathcal{F}_{x,y} \varphi (\xi_1, \eta_1) \mathcal{F}_{x,y} \varphi (\xi-\xi_1, \eta-\eta_1)  d\xi_1 d\eta_1 } d t'\Bigr\|_{L^2_{\xi,\eta}}\\
\gtrsim & N^{-  s + \frac{7}{2}}
\Bigl\| \chi_C (\xi,\eta) \int_0^t \int_{\R^2} {e^{-3it'\Phi(\xi_1,\eta_1, -\xi, -\eta)}
\chi_A(\xi_1, \eta_1) \chi_B(\xi-\xi_1, \eta-\eta_1)  d\xi_1 d\eta_1 } d t'\Bigr\|_{L^2_{\xi,\eta}}\\
\sim & N^{-  s + \frac{7}{2}}
\Bigl\| \chi_C (\xi,\eta) \int_0^t N^{-\frac{5}{2}} d t'\Bigr\|_{L^2_{\xi,\eta}}\\
\sim & N^{-s - \frac{1}{4}}.
\end{align*}
Here we used $\mu(A) = \mu(B) \sim \mu(C) \sim N^{-5/2}$.
The above estimates imply that if $s<-1/4$ the data-to-solution map fails to be $C^2$.
See Section 6 of \cite{Bo97} and Section 6 of \cite{Ho07} for the details of ``not $C^2$''.
\end{proof}
\begin{rem}
It should be noted how to find the rectangle sets $A$, $B$ in the above proof.
For the above $\alpha$, $\beta$, we can observe that
\begin{equation*}
\Phi (\alpha, \beta) = F(\alpha, \beta) = 0,
\end{equation*}
where the functions $\Phi$ and $F$ are the same as in Section 3.
The condition $\Phi (\alpha, \beta)=0$ is natural and necessary to have $|\xi \xi_1 (\xi - \xi_1) + \eta \eta_1 (\eta -  \eta_1)| \lesssim 1$.
In the proof, the key point is that the short side of rectangle $A$ is parallel to that of rectangle $B$, which is actually given by the latter condition $F(\alpha, \beta) = 0$. To see this, define $\Phi_\beta(\xi_1,\eta_1) = \Phi(\xi_1, \eta_1, \beta)$,
$\Phi_\alpha (\xi_2,\eta_2) = \Phi(\alpha, \xi_2, \eta_2)$.
Since $\Phi_\beta (\alpha)=0$, we may find a rectangle set $R_\alpha$ which satisfies $\alpha \in R_\alpha$ and $|\Phi_\beta (\xi_1, \eta_1)| \lesssim 1$ for any $(\xi_1, \eta_1)  \in R_\alpha$
whose short-side and long-side lengths are $\sim N^{-2}$ and $\sim N^{-1/2}$, respectively.
Clearly, the direction of short side of $R_\alpha$ correspond with that of $\nabla \Phi_\beta(\alpha)$. Similarly, if we define
$R_\beta$ as a rectangle set which satisfies $\beta \in R_\alpha$ and $|\Phi_\alpha (\xi_2, \eta_2)| \lesssim 1$ for any $(\xi_2, \eta_2)  \in R_\beta$, the short side of $R_\beta$ is parallel to $\nabla \Phi_\alpha(\beta)$.
Hence, we get
\begin{equation*}
\textnormal{The short side of $R_\alpha$ is parallel to that of $R_\beta$} \iff
\nabla \Phi_\beta(\alpha) \ \textnormal{ is parallel to } \nabla \Phi_\alpha(\beta).
\end{equation*}
Let $\varphi(\xi, \eta) = \xi^3+ \eta^3$. We easily see
\begin{equation*}
3 \Phi(\xi_1, \eta_1, \xi_2,\eta_2) =
\varphi (\xi_1+\xi_2 ,\eta_1+ \eta_2) - \varphi (\xi_1, \eta_1) - \varphi(\xi_2,\eta_2).
\end{equation*}
We observe that
\begin{align*}
 \begin{pmatrix}
- 3 \partial_2 \Phi_\beta (\alpha) \\
3 \partial_1 \Phi_\beta (\alpha) \\
\end{pmatrix}
\cdot &
\begin{pmatrix}
3 \partial_1 \Phi_\alpha (\beta) \\
3 \partial_2 \Phi_\alpha (\beta) \\
\end{pmatrix}
=  \mathrm{det} \begin{pmatrix}
3 \partial_1 \Phi_\beta (\alpha) & 3 \partial_1 \Phi_\alpha (\beta) \\
3 \partial_2 \Phi_\beta (\alpha) & 3 \partial_2 \Phi_\alpha (\beta) \\
\end{pmatrix}\\
& =  \
\mathrm{det} \begin{pmatrix}
0 & 0 & -1\\
-3\partial_1 \Phi_\beta (\alpha) & -3\partial_1 \Phi_\alpha (\beta) & \partial_1 \varphi(\alpha+\beta)\\
-3\partial_2 \Phi_\beta (\alpha) &  -3\partial_2 \Phi_\alpha (\beta) & \partial_2 \varphi(\alpha+\beta) \\
\end{pmatrix}\\
& =  \
\mathrm{det} \begin{pmatrix}
-1 & -1 & -1\\
\partial_1 \varphi (\alpha) & \partial_1 \varphi (\beta) & \partial_1 \varphi(\alpha+\beta)\\
\partial_2 \varphi (\alpha) &  \partial_2 \varphi (\beta) & \partial_2 \varphi(\alpha+\beta) \\
\end{pmatrix}.
\end{align*}
As a result, since we already saw in Section 3 that
\begin{equation*}
F(\alpha, \beta) = 0 \Longrightarrow \mathrm{det} \begin{pmatrix}
-1 & -1 & -1\\
\partial_1 \varphi (\alpha) & \partial_1 \varphi (\beta) & \partial_1 \varphi(\alpha+\beta)\\
\partial_2 \varphi (\alpha) &  \partial_2 \varphi (\beta) & \partial_2 \varphi(\alpha+\beta) \\
\end{pmatrix} =0,
\end{equation*}
we establish $F(\alpha, \beta)=0 \Rightarrow \nabla \Phi_\beta(\alpha) \parallel \nabla \Phi_\alpha(\beta)$.
\end{rem}
\section{Appendix (Proof of Lemma \ref{lemma3.8})}
We now show Lemma \ref{lemma3.8}. For convenience, we restate the lemma.
\begin{lem}\label{lemma3.8a}
Let $i=1,2$. For fixed $m \in \N \times \Z$, the number of $k \in \Z^2$ such that
$(m, k) \in \widehat{Z}_{A,i}$ is less than $2^{1000}$. On the other hand, for
fixed $k \in \Z^2$, the number of $m \in \N \times \Z$ such that
$(m, k) \in \widehat{Z}_{A,i}$ is less than $2^{1000}$.
Furthermore, the claim holds true whether we replace $\widehat{Z}_{A,i}$ by $\overline{Z}_{A,i}$ in
the above statements.
\end{lem}
\begin{proof}
Clearly, we may assume $A \geq 2^{300}$.
We first show that the number of $m \in \N \times \Z$ such that
$(m, k) \in \widehat{Z}_{A,i}$ is less than $2^{1000}$ for
fixed $k \in \Z^2$.
Let $2^{20} \leq M \leq 2^{-80} A$ be dyadic and $\ell$ denotes the line $\eta=(\sqrt{2}-1)^{4/3}\xi$. We decompose the proof into the two cases:\\
(I) $\displaystyle{\inf_{(\xi,\eta) \in \mathcal{T}_{k}^A}}
\textnormal{dist} ((\xi,\eta), \, \ell)  \leq 2^{80} A^{-1}N_1.$ $\quad$
\textnormal{(I \hspace{-0.15cm}I)} $\displaystyle{\inf_{(\xi,\eta) \in \mathcal{T}_{k}^A}}
\textnormal{dist} ((\xi,\eta), \, \ell)  \geq M^{-1}N_1.$\\
\underline{Case (I)}

Let $A'= 2^{-80}A$. There exists $N >0$ which satisfies
$2^{-1}N_1 \leq N \leq 2 N_1$ and
\begin{equation*}
\displaystyle{\inf_{(\xi,\eta) \in \mathcal{T}_{k}^A}}
| (\xi,\eta) - (N, (\sqrt{2} - 1)^{\frac{4}{3}}N) | \leq  {A'}^{-1} N_1.
\end{equation*}
Thus, by the same argument as in  the proof of Lemma \ref{lemma3.6}, it suffices to show that there exist
$m_1$, $m_2 \in \N \times \Z$ such that $(m_i, k_i) \in \widehat{Z}_{{2^{-100}A'},i} \cup \overline{Z}_{2^{-100}A,i}$ $(i=1,2)$ for some $k_1$, $k_2 \in \Z^2$ and the set $\mathcal{R}_{2^{-100}{A'},m_1,1} \cup \mathcal{R}_{2^{-100}{A'},m_2,2}$ contains
\begin{equation*}
\left\{ (\xi_2, \eta_2)  \in \widehat{\mathcal{K}}_1 \cup \widehat{\mathcal{K}}_2 \, \left| \,
\begin{aligned} & |\Phi(N,  (\sqrt{2} - 1)^{\frac{4}{3}}N , \xi_2, \eta_2) |\leq 2^7 {A'}^{-1} N_1^3, \\
& |F (N,  (\sqrt{2} - 1)^{\frac{4}{3}}N, \xi_2 ,\eta_2)| \leq 2^7 {A'}^{-1} N_1^2.
 \end{aligned} \right.
\right\}
\end{equation*}
The sets $\widehat{\mathcal{K}}_1$ and $\widehat{\mathcal{K}}_2$ are defined as
\begin{align*}
\widehat{\mathcal{K}}_1 & = \bigl\{ (\xi, \eta) \in \R^2 \, | \, \bigl|
\eta - ( \sqrt{2}+ 1 )^{\frac{2}{3}} (\sqrt{2} + \sqrt{3} ) \xi \bigr|
\leq 2^{-15} N_1 \bigr\},\\
\widehat{\mathcal{K}}_2 & = \bigl\{ (\xi, \eta) \in \R^2 \, | \, \bigl|
\eta + ( \sqrt{2}+ 1 )^{\frac{2}{3}} (\sqrt{3} - \sqrt{2} ) \xi \bigr|
\leq 2^{-15} N_1 \bigr\}.
\end{align*}
Let $\xi_2' = \xi_2 + N/2$, $\eta_2' = \eta_2 + (\sqrt{2} - 1)^{\frac{4}{3}}N/2$. We observe that
\begin{align}
& |\Phi(N,  (\sqrt{2} - 1)^{\frac{4}{3}}N , \xi_2, \eta_2) |\leq 2^7 {A'}^{-1} N_1^3\notag \\
\iff & | N \xi_2 (\xi_2 + N)+ (\sqrt{2} - 1)^{\frac{4}{3}}N \eta_2 ((\sqrt{2} - 1)^{\frac{4}{3}}N+\eta_2)  |\leq 2^7 {A'}^{-1} N_1^3, \notag \\
\Longrightarrow & \Bigl| {\xi_2'}^2 + (\sqrt{2} - 1)^{\frac{4}{3}} {\eta_2'}^2 - \frac{3}{2} (\sqrt{2}-1)^2 N^2  \Bigr|
\leq 2^{10} {A'}^{-1} N^2,\label{lemma3.9-est01}
\end{align}
and
\begin{align}
& |F (N,  (\sqrt{2} - 1)^{\frac{4}{3}}N, \xi_2 ,\eta_2)| \leq 2^7 {A'}^{-1} N_1^3\notag \\
\iff & | N \eta_2 + (\sqrt{2} - 1)^{\frac{4}{3}}N \xi_2  + 2((\sqrt{2} - 1)^{\frac{4}{3}}N^2+\xi_2
\eta_2)  |\leq 2^7 {A'}^{-1} N_1^3, \notag \\
\Longrightarrow & \Bigl| 2{\xi_2'} \eta_2' + \frac{3}{2} (\sqrt{2}-1)^{\frac{4}{3}} N^2  \Bigr|
\leq 2^{10} {A'}^{-1} N^2.\label{lemma3.9-est02}
\end{align}
\eqref{lemma3.9-est01} and \eqref{lemma3.9-est02} yield $|\xi_2'| \geq 2^{-4} N$ and $|\eta_2'| \geq 2^{-4} N$. Thus, it follows from \eqref{lemma3.9-est02} that
\begin{equation}
\Bigl|  \eta_2' + \frac{3}{4} (\sqrt{2}-1)^{\frac{4}{3}} \frac{N^2}{\xi_2'}  \Bigr|
\leq 2^{14} {A'}^{-1}N.\label{lemma3.9-est03}
\end{equation}
From \eqref{lemma3.9-est01} and \eqref{lemma3.9-est03}, we get
\begin{align}
& \Bigl| {\xi_2'}^2 + \frac{9}{16} (\sqrt{2} - 1)^{4}\frac{N^4}{{\xi_2'}^2} - \frac{3}{2} (\sqrt{2}-1)^2 N^2  \Bigr|
\leq 2^{15} {A'}^{-1} N^2,\notag \\
\iff & {\xi_2'}^{-2} \Bigl( {\xi_2'}^2 - \frac{3}{4} (\sqrt{2}-1)^2 N^2\Bigr)^2
\leq 2^{15} {A'}^{-1} N^2,\notag \\
\Longrightarrow \, &  \Bigl| {\xi_2'}^2 - \frac{3}{4} (\sqrt{2}-1)^2 N^2\Bigr|
\leq 2^{10} {A'}^{-\frac{1}{2}} N^2,\notag \\
\Longrightarrow \, &  \min_{j= \pm1} \Bigl| {\xi_2'} + j \frac{\sqrt{3}}{2} (\sqrt{2}-1) N\Bigr|
\leq 2^{12} {A'}^{-\frac{1}{2}} N.\label{lemma3.9-est04}
\end{align}
\eqref{lemma3.9-est04} gives
\begin{equation}
\min_{j= \pm1} \Bigl| \frac{3}{4}  (\sqrt{2}-1)^{\frac{4}{3}} \frac{N^2}{\xi_2'}
+ (\sqrt{2}+1)^{\frac{2}{3}} \xi_2' -j \sqrt{3} (\sqrt{2}-1)^{\frac{1}{3}} N\Bigr|
\leq 2^{23}{A'}^{-1}N.\label{lemma3.9-est05}
\end{equation}
\eqref{lemma3.9-est05}, combined with \eqref{lemma3.9-est03}, provides
\begin{align*}
& \min_{j= \pm1} |\eta_2' - (\sqrt{2}+1)^{\frac{2}{3}} \xi_2' +j \sqrt{3} (\sqrt{2}-1)^{\frac{1}{3}} N|\\
\leq & \Bigl| \eta_2' + \frac{3}{4} (\sqrt{2}-1)^{\frac{4}{3}} \frac{N^2}{\xi_2'}  \Bigr| + \min_{j= \pm1}
\Bigl| \frac{3}{4}  (\sqrt{2}-1)^{\frac{4}{3}} \frac{N^2}{\xi_2'}
+ (\sqrt{2}+1)^{\frac{2}{3}} \xi_2' -j \sqrt{3} (\sqrt{2}-1)^{\frac{1}{3}} N\Bigr| \\
\leq & \, 2^{25} {A'}^{-1}N.
\end{align*}
Consequently, $|\Phi(N,  (\sqrt{2} - 1)^{\frac{4}{3}}N , \xi_2, \eta_2) |\leq 2^7 {A'}^{-1} N_1^3$ and
$|F (N,  (\sqrt{2} - 1)^{\frac{4}{3}}N, \xi_2 ,\eta_2)| \leq 2^7 {A'}^{-1} N_1^3$ yield one of the following two cases.
\begin{align}
& \begin{cases}
\bigl| {\xi_2'} - \frac{\sqrt{3}}{2} (\sqrt{2}-1) N\bigr|
\leq 2^{12} {A'}^{-\frac{1}{2}} N,\\
|\eta_2' - (\sqrt{2}+1)^{\frac{2}{3}} \xi_2' - \sqrt{3} (\sqrt{2}-1)^{\frac{1}{3}} N| \leq 2^{25} {A'}^{-1} N.
\end{cases}\label{lemma3.9-est06}\\
& \begin{cases}
\bigl| {\xi_2'} + \frac{\sqrt{3}}{2} (\sqrt{2}-1) N\bigr|
\leq 2^{12} {A'}^{-\frac{1}{2}} N,\\
|\eta_2' - (\sqrt{2}+1)^{\frac{2}{3}} \xi_2' + \sqrt{3} (\sqrt{2}-1)^{\frac{1}{3}} N| \leq 2^{25} {A'}^{-1} N.
\end{cases}\label{lemma3.9-est07}
\end{align}
Since $\xi_2' = \xi_2 + N/2$, $\eta_2' = \eta_2 + (\sqrt{2} - 1)^{\frac{4}{3}}N/2$, \eqref{lemma3.9-est06} and
\eqref{lemma3.9-est07} imply
\begin{equation}
 \begin{cases}
\bigl| {\xi_2} + \frac{1}{2} N - \frac{\sqrt{3}}{2} (\sqrt{2}-1) N\bigr|
\leq 2^{12} {A'}^{-\frac{1}{2}} N,\\
|\eta_2 - (\sqrt{2}+1)^{\frac{2}{3}} \xi_2 + (\sqrt{3}-1) (\sqrt{2}-1)^{\frac{1}{3}} N| \leq 2^{25} {A'}^{-1} N,
\end{cases}\label{lemma3.9-est08}
\end{equation}
and
\begin{equation}
 \begin{cases}
\bigl| {\xi_2} + \frac{1}{2} N  + \frac{\sqrt{3}}{2} (\sqrt{2}-1) N\bigr|
\leq 2^{12} {A'}^{-\frac{1}{2}} N,\\
|\eta_2 - (\sqrt{2}+1)^{\frac{2}{3}} \xi_2 + (\sqrt{3}+1) (\sqrt{2}-1)^{\frac{1}{3}} N| \leq 2^{25} {A'}^{-1} N,
\end{cases}\label{lemma3.9-est09}
\end{equation}
respectively. We calculate that \eqref{lemma3.9-est08} and \eqref{lemma3.9-est09} provided
\begin{equation*}
 | \eta_2-   ( \sqrt{2}+ 1 )^{\frac{2}{3}} (\sqrt{2} + \sqrt{3} ) \xi_2 | \leq 2^{30} {A'}^{-\frac{1}{2}} N,
\end{equation*}
and
 \begin{equation*}
 | \eta_2+   ( \sqrt{2}+ 1 )^{\frac{2}{3}} ( \sqrt{3}-\sqrt{2} ) \xi_2 | \leq 2^{30} {A'}^{-\frac{1}{2}} N,
\end{equation*}
respectively. This completes the proof for Case (I).\\
\underline{Case (I \hspace{-0.15cm}I)}

Similarly to Case (I), we may find $N>0$, $M_0 \in \R$ which satisfies
$2^{-1}N_1 \leq N \leq 2 N_1$, $2^{-1}M \leq |M_0| \leq 2M$ and
\begin{equation*}
\displaystyle{\inf_{(\xi_1,\eta) \in \mathcal{T}_{k}^A}}
| (\xi,\eta) - (N, (\sqrt{2} - 1)^{\frac{4}{3}}N)+ (0, M_0^{-1} N)  | \leq 2^{2} A^{-1} N_1,
\end{equation*}
and we will show that there exist
$m_1$, $m_2 \in \N \times \Z$ such that $(m_i, k_i) \in \widehat{Z}_{2^{-100}A,i} \cup \overline{Z}_{2^{-100}A,i}$ for some $k_1$, $k_2 \in \Z^2$ and  the set $\mathcal{R}_{2^{-100}A,m_1,1} \cup \mathcal{R}_{2^{-100}A,m_2,2}$ contains
\begin{equation*}
\left\{ (\xi_2, \eta_2)  \in \widehat{\mathcal{K}}_1 \cup \widehat{\mathcal{K}}_2 \, \left| \,
\begin{aligned} & |\Phi(N,  (\sqrt{2} - 1)^{\frac{4}{3}}N -M_0^{-1} N , \xi_2, \eta_2) |\leq 2^7 A^{-1} N_1^3, \\
& |F (N,  (\sqrt{2} - 1)^{\frac{4}{3}}N - M_0^{-1} N, \xi_2 ,\eta_2)| \leq 2^7 A^{-1} N_1^2.
 \end{aligned} \right.
\right\}
\end{equation*}
By using $\tilde{\xi}_2 = \xi_2 + N/2$,
$\tilde{\eta}_2 = \eta_2 + (\sqrt{2} - 1)^{\frac{4}{3}}N/2 -  M_0^{-1} N$, we calculate that
\begin{align}
& |\Phi(N,  (\sqrt{2} - 1)^{\frac{4}{3}}N  -M_0^{-1} N, \xi_2, \eta_2) |\leq 2^7 A^{-1} N_1^3\notag \\
\Longrightarrow & \Bigl| {\tilde{\xi}_2}^2 + ( (\sqrt{2} - 1)^{\frac{4}{3}} - M_0^{-1}) {\tilde{\eta}_2}^2
- \frac{3}{2} (\sqrt{2}-1)^2 N^2 +\frac{3}{4} (\sqrt{2}-1)^{\frac{8}{3}}M_0^{-1} N^2 \notag \\
& \qquad \qquad \qquad \qquad \quad -
\frac{3}{4} (\sqrt{2}-1)^{\frac{4}{3}} M_0^{-2}N^2 +\frac{M_0^{-3}}{4} N^2 \Bigr|
\leq 2^{10} A^{-1} N^2,\label{lemma3.9-est101}
\end{align}
and
\begin{align}
& |F (N,  (\sqrt{2} - 1)^{\frac{4}{3}}N -M_0^{-1} N, \xi_2 ,\eta_2)| \leq 2^7 A^{-1} N_1^3\notag \\
\Longrightarrow & \Bigl| 2{\tilde{\xi}_2} \tilde{\eta}_2 + \frac{3}{2} \bigl( (\sqrt{2}-1)^{\frac{4}{3}}
- M_0^{-1} \bigr) N^2  \Bigr|
\leq 2^{10} A^{-1} N^2.\label{lemma3.9-est102}
\end{align}
We deduce from $2^{19} \leq |M_0|$, \eqref{lemma3.9-est101} and \eqref{lemma3.9-est102} that $|\tilde{\xi}_2| \geq 2^{-4} N$ and $|\tilde{\eta}_2| \geq 2^{-4} N$. Thus, from \eqref{lemma3.9-est102}, we get
\begin{equation}
\Bigl|  \tilde{\eta}_2 + \frac{3}{4\tilde{\xi}_2} \bigl( (\sqrt{2}-1)^{\frac{4}{3}}
- M_0^{-1} \bigr)  N^2 \Bigr|
\leq 2^{14} A^{-1}N.\label{lemma3.9-est103}
\end{equation}
By a simple calculation, it follows from \eqref{lemma3.9-est101} and \eqref{lemma3.9-est103} that
\begin{align}
& \Bigl| {\tilde{\xi}_2}^2 +  \frac{9}{16}  \bigl( (\sqrt{2}-1)^{\frac{4}{3}}
- M_0^{-1} \bigr)^3 \frac{N^4}{{\tilde{\xi}_2}^2}
- \frac{3}{2} (\sqrt{2}-1)^2 N^2  \notag \\
& \qquad   +\frac{3}{4} (\sqrt{2}-1)^{\frac{8}{3}}M_0^{-1} N^2 -
\frac{3}{4} (\sqrt{2}-1)^{\frac{4}{3}} M_0^{-2}N^2 +\frac{M_0^{-3}}{4} N^2 \Bigr|
\leq 2^{15} A^{-1} N^2 \notag \\
\iff & {\tilde{\xi}_2}^{-2} \Bigl| \Bigl( {\tilde{\xi}_2}^2 - \frac{3}{4} (\sqrt{2}-1)^2 N^2\Bigr)^2
 -\frac{1}{4}   C(M_0) M_0^{-1} \Bigl( \frac{9}{4}N^2 - \tilde{\xi}_2^2 \Bigr) N^2 \Bigr|
\leq 2^{15} A^{-1} N^2,\label{lemma3.9-este03}
\end{align}
where $C(M_0):= 3 (\sqrt{2}-1)^{\frac{8}{3}} -3 (\sqrt{2}-1)^{\frac{4}{3}}M_0^{-1}
+ M_0^{-2}$. From \eqref{lemma3.9-este03} and $|M_0| \geq 2^{19}$, we easily observe that
$9 N^2/4 - \tilde{\xi}_2^2 \geq 2^{-4} N^2$. Therefore, $M_0$ must be positive to satisfy \eqref{lemma3.9-este03} since $|M_0| \leq 2^{-79}A$. We see that \eqref{lemma3.9-este03} implies
\begin{align*}
& \Bigl| {\tilde{\xi}_2}^2 - \frac{3}{4} (\sqrt{2}-1)^2 N^2
+\frac{1}{2} \sqrt{C(M_0)  \Bigl( \frac{9}{4} - \frac{\tilde{\xi}_2^2}{N^2} \Bigr)} M_0^{-\frac{1}{2}} N^2\Bigr| \notag\\
\times & \Bigl| {\tilde{\xi}_2}^2 - \frac{3}{4} (\sqrt{2}-1)^2 N^2
-\frac{1}{2} \sqrt{C(M_0)  \Bigl( \frac{9}{4} - \frac{\tilde{\xi}_2^2}{N^2} \Bigr)} M_0^{-\frac{1}{2}} N^2\Bigr|\leq 2^{17} A^{-1} N^4,
\end{align*}
which gives
\begin{equation}
\Bigl| {\tilde{\xi}_2}^2  - \frac{3}{4}  (\sqrt{2}-1)^2 N^2
\pm \frac{1}{2} \sqrt{C(M_0)  \Bigl( \frac{9}{4} - \frac{\tilde{\xi}_2^2}{N^2} \Bigr)} M_0^{-\frac{1}{2}} N^2 \Bigr|\leq 2^{20} A^{-1} M_0^{\frac{1}{2}} N^2.\label{lemma3.9-est104}
\end{equation}
Since $M_0>2^{19}$, we can find $\alpha>0$ and $\beta >0$ such that
\begin{align*}
& {\alpha}^2  - \frac{3}{4}  (\sqrt{2}-1)^2 N^2
+ \frac{1}{2} \sqrt{C(M_0)  \left( \frac{9}{4} - \frac{\alpha^2}{N^2} \right)} M_0^{-\frac{1}{2}} N^2 =0,\\
& {\beta}^2  - \frac{3}{4}  (\sqrt{2}-1)^2 N^2
- \frac{1}{2} \sqrt{C(M_0)  \left( \frac{9}{4} - \frac{\beta^2}{N^2} \right)} M_0^{-\frac{1}{2}} N^2 =0.
\end{align*}
Clearly, $\alpha$, $\beta$ satisfy
\begin{equation}
 \Bigl| {\alpha} - \frac{\sqrt{3}}{2} (\sqrt{2}-1) N\Bigr|
\leq 2^{10} M_0^{-\frac{1}{2}} N, \quad
 \Bigl| {\beta} - \frac{\sqrt{3}}{2} (\sqrt{2}-1) N\Bigr|
\leq 2^{10} M_0^{-\frac{1}{2}} N.\label{lemma3.9-este04}
\end{equation}
We observe that \eqref{lemma3.9-est104} and $M_0>2^{19}$ give
\begin{equation*}
\Bigl| \partial_{\tilde{\xi}_2} \Bigl( {\tilde{\xi}_2}^2  - \frac{3}{4} (\sqrt{2}-1)^2 N^2
\pm \frac{1}{2} \sqrt{C(M_0)  \Bigl( \frac{9}{4} - \frac{\tilde{\xi}_2^2}{N^2} \Bigr)} M_0^{-\frac{1}{2}} N^2
\Bigr) \Bigr|\geq 2^{-3} N.
\end{equation*}
Therefore, from \eqref{lemma3.9-est104}, it holds that
\begin{equation*}
|\tilde{\xi}_2 \pm \alpha| \leq 2^{30}A^{-1} M_0^{\frac{1}{2}} N \quad \textnormal{or} \quad
|\tilde{\xi}_2 \pm \beta| \leq 2^{30}A^{-1} M_0^{\frac{1}{2}} N.
\end{equation*}
To avoid redundancy, we here only treat the case
\begin{equation}
|\tilde{\xi}_2 - \alpha| \leq 2^{30}A^{-1} M_0^{\frac{1}{2}} N.\label{lemma3.9-este05}
\end{equation}
The other cases can be treated similarly.
\eqref{lemma3.9-este05} gives
\begin{equation*}
\Bigl| \frac{1}{\tilde{\xi}_2}-\frac{1}{\alpha} + \frac{1}{\alpha^2}( \tilde{\xi}_2 - \alpha) \Bigr|
\leq A^{-1}N^{-1},
\end{equation*}
which yields
\begin{equation}
\Bigl| \frac{3}{4} \bigl( (\sqrt{2}-1)^{\frac{4}{3}}
- M_0^{-1} \bigr) \Bigl( \frac{N^2}{\tilde{\xi}_2}- \frac{N^2}{\alpha}\Bigr)
+ (\sqrt{2}+1)^{\frac{2}{3}} (\tilde{\xi}_2 -\alpha)\Bigr|
\leq 2^{50}A^{-1}N.\label{lemma3.9-est105}
\end{equation}
Here we used the inequality $|1/\tilde{\xi}_2 - 1/ \alpha| \leq 2^{35}A^{-1} M_0^{1/2} N^{-1}$
which is given by \eqref{lemma3.9-este05}.
It follows from \eqref{lemma3.9-est103} and \eqref{lemma3.9-est105} that
\begin{align*}
& \Bigl|\tilde{\eta}_2 - (\sqrt{2}+1)^{\frac{2}{3}} \tilde{\xi}_2 +  (\sqrt{2}+1)^{\frac{2}{3}} \alpha +
\frac{3N^2}{4 \alpha} \bigl( (\sqrt{2}-1)^{\frac{4}{3}}
- M_0^{-1} \bigr)  \Bigr|\\
\leq & \Bigl| \tilde{\eta}_2 + \frac{3}{4}\bigl( (\sqrt{2}-1)^{\frac{4}{3}}
- M_0^{-1} \bigr)  \frac{N^2}{\tilde{\xi}_2}  \Bigr| \\
& \qquad  +
\Bigl| \frac{3}{4} \bigl( (\sqrt{2}-1)^{\frac{4}{3}}
- M_0^{-1} \bigr) \Bigl( \frac{N^2}{\tilde{\xi}_2}- \frac{N^2}{\alpha}\Bigr)
+ (\sqrt{2}+1)^{\frac{2}{3}} (\tilde{\xi}_2 -\alpha)\Bigr| \\
\leq & \, 2^{55} A^{-1}N.
\end{align*}
Since  $\tilde{\xi}_2 = \xi_2 + N/2$ and
$\tilde{\eta}_2 = \eta_2 + (\sqrt{2} - 1)^{\frac{4}{3}}N/2 -  M_0^{-1} N$, this and
\eqref{lemma3.9-este05} imply
\begin{equation}
 \begin{cases}
|\xi_2 + \frac{N}{2} - \alpha| \leq 2^{30}A^{-1} M_0^{\frac{1}{2}} N,\\
|\eta_2 - (\sqrt{2}+1)^{\frac{2}{3}} \xi_2+ C'(\alpha, M_0,N) | \leq 2^{55} A^{-1} N,
\end{cases}\label{lemma3.9-est108}
\end{equation}
where
\begin{equation*}
C'(\alpha, M_0,N) =  (\sqrt{2}+1)^{\frac{2}{3}} \alpha +
\frac{3N^2}{4 \alpha} \bigl( (\sqrt{2}-1)^{\frac{4}{3}}
- M_0^{-1} \bigr) - (\sqrt{2}-1)^{\frac{1}{3}} N - M_0^{-1} N.
\end{equation*}
We deduce from \eqref{lemma3.9-est108} and \eqref{lemma3.9-este04} that
\begin{equation*}
 | \eta_2-   ( \sqrt{2}+ 1 )^{\frac{2}{3}} (\sqrt{2} + \sqrt{3} ) \xi_2 | \leq 2^{15} M_0^{-\frac{1}{2}} N.
\end{equation*}
This completes the proof for Case (I \hspace{-0.15cm}I).

Next we prove that the number of $k \in \Z^2$ such that
$(m, k) \in \widehat{Z}_{A,i}$ $(i=1,2)$ is less than $2^{1000}$ for
fixed $m \in \N \times \Z$. Here we only consider the case $(m, k) \in \widehat{Z}_{A,2}$. The case
$(m, k) \in \widehat{Z}_{A,1}$ can be treated in a similar way.
For any $\mathcal{R}_{A,m,2}$ which satisfies $(m, k) \in \widehat{Z}_{A,2}$, we can find $N$ and $M$ which satisfy $2^{-1}N_1 \leq N \leq 2 N_1$ and $2^{15} \leq |M| \leq 2^{-5} \sqrt{A}$ such that $\mathcal{R}_{A,m,2}$ is contained in the following rectangle set.
\begin{equation*}
\left\{
\begin{aligned} (\xi, \eta) = & \bigl(1, -  ( \sqrt{2}+ 1 )^{\frac{2}{3}} (\sqrt{3} - \sqrt{2} )
\bigr)N +  (1,( \sqrt{2}+ 1 )^{\frac{2}{3}}) M^{-1}
N \\
+  (1, & ( \sqrt{2}+ 1 )^{\frac{2}{3}}) A^{-1} MN c
+ (-1 , ( \sqrt{2}- 1 )^{\frac{2}{3}}) A^{-1}N c', \ \ -2^{10} \leq c, c' \leq 2^{10}.
 \end{aligned}
\right\}
\end{equation*}
Define
\begin{align*}
 \xi_{N,M}= \xi_{N,M} (c,c') & =
\left(1 + M^{-1}+ A^{-1} M c-A^{-1} c' \right)N,\\
 \eta_{N,M}= \eta_{N,M} (c,c') & = \Bigl(-  ( \sqrt{2}+ 1 )^{\frac{2}{3}} (\sqrt{3} - \sqrt{2} )
+ ( \sqrt{2}+ 1 )^{\frac{2}{3}}M^{-1}\\
& \qquad + ( \sqrt{2}+ 1 )^{\frac{2}{3}}A^{-1} M c-
( \sqrt{2}- 1 )^{\frac{2}{3}}A^{-1} c' \Bigr)N.
\end{align*}
By the same argument as in  the proof of Lemma \ref{lemma3.6}, it suffices to show that there exists
$k \in \Z^2$ such that $\mathcal{T}_{k}^{2^{-100}A} \cap \mathcal{K}_0 \not= \emptyset$ and
for all $c$, $c'$ such that $-2^{20} \leq c, c' \leq 2^{20}$, the set
$\mathcal{T}_{k}^{2^{-100}A}$ contains
\begin{equation*}
\left\{ (\xi, \eta)  \in \widehat{\mathcal{K}}_0 \, \left| \,
\begin{aligned} & |\Phi(\xi_{N,M} (c,c'),  \eta_{N,M} (c,c') , \xi, \eta) |\leq 2^7 A^{-1} N^3, \\
& |F (\xi_{N,M} (c,c'),  \eta_{N,M} (c,c'), \xi, \eta)| \leq 2^7 A^{-1} N^2.
 \end{aligned} \right.
\right\}
\end{equation*}
The set $\widehat{\mathcal{K}}_0$ is defined as
\begin{equation*}
\widehat{\mathcal{K}}_0 = \bigl\{ (\xi, \eta) \in \R^2 \, | \, \bigl|
\eta -(\sqrt{2} - 1)^{\frac{4}{3}} \xi \bigr|
\leq 2^{-15} N \bigr\}.
\end{equation*}
Let $\xi'=\xi+\xi_{N,M}/2$ and $\eta'=\eta + \eta_{N,M}/2$ and
\begin{align*}
\tilde{\Phi}(\xi',\eta') & = \xi_{N,M} {\xi'}^2 + \eta_{N,M} {\eta'}^2 - \frac{{\xi_{N,M}}^3 +
{\eta_{N,M}}^3}{4},\\
\tilde{F}(\xi',\eta') & = 2 \xi' \eta' + \frac{3}{2} \, \xi_{N,M} \, \eta_{N,M}.
\end{align*}
Recall that
\begin{align*}
|\Phi(\xi_{N,M},  \eta_{N,M} , \xi, \eta) |\leq 2^7 A^{-1} N_1^3 &
\iff |\tilde{\Phi}(\xi',\eta')| \leq 2^7 A^{-1} N_1^3,\\
 |F (\xi_{N,M},  \eta_{N,M} , \xi, \eta)| \leq 2^7 A^{-1} N_1^2 &
\iff |\tilde{F}(\xi',\eta')| \leq 2^7 A^{-1} N_1^2.
\end{align*}
It is easy to see $|\xi'| \geq 2^{-4} N$ and $|\eta'| \geq 2^{-4} N$ if we assume $|\tilde{\Phi}(\xi',\eta')| \leq 2^7 A^{-1} N_1^3$ and $|\tilde{F}(\xi',\eta')| \leq 2^7 A^{-1} N_1^2.$  In the same manner as in the proof of
Lemma \ref{lemma3.6}, we observe that
\begin{align}
& |\tilde{\Phi}(\xi',\eta')| \leq 2^7 A^{-1} N_1^3 \ \, \textnormal{and} \ \, |\tilde{F}(\xi',\eta')| \leq 2^7 A^{-1} N_1^2
\notag \\
\Longrightarrow & \Bigl| \xi_{N,M} {\xi'}^2 + \eta_{N,M} \frac{9 {\xi_{N,M}}^2 {\eta_{N,M}}^2}{16 {\xi'}^2} - \frac{{\xi_{N,M}}^3 + {\eta_{N,M}}^3}{4} \Bigr|
\leq 2^{15} A^{-1} N^3 \notag \\
\Longrightarrow & \Bigl| \xi_{N,M} {\xi'}^4 - \frac{{\xi_{N,M}}^3 + {\eta_{N,M}}^3}{4} {\xi'}^2 +
\frac{9}{16} {\xi_{N,M}}^2 {\eta_{N,M}}^3  \Bigr|
\leq 2^{20} A^{-1} N^5.
\label{lemma3.9-estee01}
\end{align}
Define $\tilde{\xi} = \xi + N/2 = \xi'-N/2( M^{-1}+ A^{-1} M c-A^{-1} c')$.
It is observed that \eqref{lemma3.9-estee01} gives
\begin{align}
\Bigl| \Bigl( \tilde{\xi}^2- \frac{3}{4}(\sqrt{2}+1)^2(\sqrt{3}-\sqrt{2})^2 N^2 \Bigr)
\Bigl( \tilde{\xi}^2+ \frac{3}{4}(\sqrt{3}-\sqrt{2}) N^2\Bigr) & \notag \\
+ \Bigl( \tilde{\xi}+ \frac{\sqrt{3}}{2}(\sqrt{2}+1)(\sqrt{3}-\sqrt{2}) N \Bigr) H_1(\tilde{\xi})
( M^{-1} + & A^{-1} M c)
 + H_2(\tilde{\xi}) M^{-2} \Bigr| \notag \\
& \leq 2^{20} A^{-1} N^4,\label{lemma3.9-estee02}
\end{align}
where we used the function
\begin{align*}
H_1(\tilde{\xi}) = \tilde{\xi}^3 + \frac{1}{2}(1- 3 \sqrt{2} + 2 \sqrt{3} & + \sqrt{6}) \tilde{\xi}^2
+ \frac{\sqrt{3}}{4}(8 +4 \sqrt{2}-5 \sqrt{3}-4 \sqrt{6})\tilde{\xi}\\
& +\frac{3 \sqrt{3}}{8} (\sqrt{2}+1)(\sqrt{3}-\sqrt{2})(3 + 2 \sqrt{2} - 2 \sqrt{3}),
\end{align*}
and the function $H_2 (\tilde{\xi})$ which satisfies $|H_2(\tilde{\xi})| \leq 2^{3} (\tilde{\xi}^4 + N^4)$ and does not depend on $c$ and $c'$. We omit the detail of $H_2(\tilde{\xi})$ since it is long and complicated.
We deduce from \eqref{lemma3.9-estee02} that
\begin{equation*}
\Bigl| \tilde{\xi}- \frac{\sqrt{3}}{2}(\sqrt{2}+1)(\sqrt{3}-\sqrt{2}) N \Bigr| \leq 2^{-10} N \ \,
\textnormal{or} \ \, \Bigl| \tilde{\xi} + \frac{\sqrt{3}}{2}(\sqrt{2}+1)(\sqrt{3}-\sqrt{2}) N \Bigr| \leq 2^{-10} N.
\end{equation*}
First, we assume that $|\tilde{\xi}- \sqrt{3} (\sqrt{2}+1)(\sqrt{3}-\sqrt{2}) N/2 | \leq 2^{-10} N$. It follows from $ |F (\xi_{N,M},  \eta_{N,M} , \xi, \eta)| \leq 2^7 A^{-1} N_1^2$ that
\begin{equation*}
\Bigl| \frac{\eta}{\xi} - (\sqrt{2} + 1)^{\frac{2}{3}} (\sqrt{2} + \sqrt{3}) \Bigr| \leq 2^{-5},
\end{equation*}
which implies $(\xi,\eta) \notin \widehat{\mathcal{K}}_0$. Next we assume that
 $|\tilde{\xi} + \sqrt{3} (\sqrt{2}+1)(\sqrt{3}-\sqrt{2}) N/2 | \leq 2^{-10} N$.
In this case, we easily get
\begin{equation*}
\Bigl| \partial_{\tilde{\xi}} \Bigl( \bigl( \tilde{\xi}^2- \frac{3}{4}(\sqrt{2}+1)^2(\sqrt{3}-\sqrt{2})^2 N^2 \bigr)
\bigl( \tilde{\xi}^2+ \frac{3}{4}(\sqrt{3}-\sqrt{2}) N^2\bigr) \Bigr) \Bigr| \geq 2^{-2} N^3.
\end{equation*}
Hence, from \eqref{lemma3.9-estee02}, we can find $\alpha_0 \in \R$ such that
$|\alpha_0| \leq 2^5$ and
\begin{equation}
\Bigl| \tilde{\xi} + \frac{\sqrt{3}}{2}(\sqrt{2}+1)(\sqrt{3}-\sqrt{2}) N + \alpha_0 M^{-2} N\Bigr| \leq 2^{30}A^{-1} N.
\label{lemma3.9-estee03}
\end{equation}
Since $\tilde{\xi} = \xi + N/2$, \eqref{lemma3.9-estee03} means
\begin{equation}
\Bigl| \xi + \frac{\sqrt{3}}{2}(\sqrt{2}+1)(\sqrt{3}-\sqrt{2}) N + \frac{N}{2}+ \alpha_0 M^{-2} N\Bigr| \leq 2^{40}A^{-1} N.
\label{lemma3.9-estee04}
\end{equation}
This and  $|F (\xi_{N,M},  \eta_{N,M} , \xi, \eta)| \leq 2^7 A^{-1} N_1^2$ yield
\begin{align}
\Bigl| \bigl( \sqrt{3}(\sqrt{2}+1)(\sqrt{3}-\sqrt{2}) - M^{-1} - A^{-1}Mc \bigr)
\bigl( & \eta+(\sqrt{2}+1)^{\frac{2}{3}}  (\sqrt{2}+\sqrt{6} -2 \sqrt{3}) N \bigr) \notag \\
& - H_3(\eta) M^{-2} N \Bigr| \leq 2^{50} A^{-1} N,\label{lemma3.9-estee05}
\end{align}
where
\begin{equation*}
H_3(\eta)= 2 \alpha_0 \eta - (\sqrt{2}+1)^{\frac{2}{3}} \bigl( (\sqrt{3}- \sqrt{2})\alpha_0 + 2 \bigr) +
(\sqrt{2}+1)^{\frac{2}{3}} \alpha_0 M^{-1}.
\end{equation*}
Since $|M| \geq 2^{15}$, it follows from \eqref{lemma3.9-estee05} that there exists $\beta_0 \in \R$ such that $|\beta_0| \leq 2^5$ and
\begin{equation*}
\bigl|  \eta + (\sqrt{2}+1)^{\frac{2}{3}}  (\sqrt{2}+\sqrt{6} -2 \sqrt{3}) N  + \beta_0 M^{-2} N \bigr| \leq 2^{50}
A^{-1} N.
\end{equation*}
This, combined with \eqref{lemma3.9-estee04}, completes the proof.
\end{proof}
\section*{Acknowledgement}
The author would like to express his appreciation to Professor Sebastian Herr for a lot of helpful suggestions and encouragement. 
He is also grateful to the anonymous referee and Professor Felipe Linares for useful comments. 
This work was supported by the DFG through the CRC 1283 ``Taming uncertainty and profiting from randomness and low regularity in analysis, stochastics and their applications.''

\end{document}